\documentclass[11pt,thmsa,reqno]{amsart}

\usepackage[all]{xy}
\usepackage{color}
\usepackage{mathtools} 
\usepackage{fancybox}
\usepackage{amssymb}
\usepackage[breaklinks=true]{hyperref}
\usepackage{comment}
\usepackage{graphicx}
\usepackage{faktor}
\usepackage{tikz}
\usepackage{tikz-cd}
\usepackage{bm} 
 \newtheorem{theorem}{Theorem}[section]
 \newtheorem*{thm*}{Theorem} 
 \newtheorem{lemma}[theorem]{Lemma}
 \newtheorem{proposition}[theorem]{Proposition}
 
 \newtheorem{corollary}[theorem]{Corollary}
 \newtheorem{definition}[theorem]{Definition}
  \theoremstyle{definition}
 \newtheorem{example}[theorem]{Example}
 \newtheorem{remark}[theorem]{Remark}

 \newtheorem{assumption}[theorem]{Assumption}

\newcommand{\Ham}{\mathsf{Ham}}

\newcommand{\cL}{\mathcal{L}}

\newcommand{\ZZ}{\ensuremath{\mathbb Z}}
\newcommand{\RR}{\ensuremath{\mathbb R}}

\newcommand{\g}{\ensuremath{\mathfrak{g}}}

\newcommand{\vX}{\ensuremath{\mathfrak{X}}}
\newcommand{\dd}{\ensuremath{d}}
\newcommand{\cs}{ \blacktriangleleft } 
\newcommand{\ca}{ \vartriangleleft } 
\newcommand{\bpi}{\bm{\pi}}
\newcommand{\bmu}{\bm{\mu}}
\newcommand{\bS}{\bm{S}} 
\newcommand{\floor}[1]{\left\lfloor #1 \right\rfloor} 
\newcommand{\pr}{\mathrm{pr}}

\usepackage{bbold}
	
	\DeclareMathOperator{\cyc}{\scriptstyle\text{(cyc.)}}
	\DeclareMathOperator{\id}{\text{id}}
	\DeclareMathOperator{\pairing}{\langle\cdot,\cdot\rangle}			
	\DeclareMathOperator{\Preq}{\mathrm{Preq}}

\usepackage{stackengine}
	\newcommand{\pair}[2]{
  		\relax\if@display
			\left({\stackanchor[5pt]{$#1$}{$#2$}}\right)
  		\else
  			\binom{#1}{#2}
  		\fi	
	}	
	\newcommand{\equal}[1]{\smash{\mathrel{\overset{\makebox[3em]{\mbox{\normalfont\tiny\sffamily #1}}}{=}}}}
	\newcommand{\unsh}[1]{\scriptstyle{\left(\stackrel{\scriptscriptstyle{ unsh.}}{#1}\right)}}

\usepackage{cancel}

\usepackage[cal=boondox,scr=boondoxo]{mathalfa}

\usepackage{tikz-cd}
\usepackage{mathtools}
\usepackage{amsfonts}
\usetikzlibrary{decorations.pathmorphing}
\usetikzlibrary{shapes.geometric,fit}
\tikzset{
   dashsquare/.style={rectangle,draw,dashed,inner sep=0pt,blue,fit={#1}}
}

\newcommand{\morphism}[5]{%
  \begin{tikzcd}[
    column sep=2em,
    row sep=-.5ex,
    ampersand replacement=\&
  ]
  #1\colon~ \&[-2.5em]
  #2\vphantom{#3} \arrow[r] \&
  #3\vphantom{#2} \\
  \&
  #4\vphantom{#5} \arrow[r,mapsto] \&
  #5\vphantom{#4} 
  \end{tikzcd}%
}
\newcommand{\isomorphism}[5]{%
  \begin{tikzcd}[
    column sep=2em,
    row sep=-.5ex,
    ampersand replacement=\&
  ]
  #1\colon~ \&[-2.5em]
  #2\vphantom{#3} \arrow[r,"\sim"] \&
  #3\vphantom{#2} \\
  \&
  #4\vphantom{#5} \arrow[r,mapsto] \&
  #5\vphantom{#4} 
  \end{tikzcd}%
}

	\newcommand{\ush}[1]{S_{#1}}

	\newcommand{\Lspace}{\mathsf{L}}
	\newcommand{\Aspace}{\mathsf{A}}
	\newcommand{\Vspace}{\mathsf{V}}
	\newcommand{\Aelem}{\mathsf{a}}
	\newcommand{\Velem}{\mathsf{v}}
	
	\newcommand{\phimap}{\phi}
	\newcommand{\psimap}{\psi}
	\newcommand{\phimapshift}{\Phi}
	
	\newcommand{\Unit}{\mathbb{1}}
	
\tikzcdset{
  cells={font=\everymath\expandafter{\the\everymath\displaystyle}},
}

	\newcommand{\Ss}{\S}

\begin{document}

 \author{Antonio Michele Miti}
\email{miti@mpim-bonn.mpg.de}
\address{Max Planck Institute for Mathematics, Vivatsgasse 7, 53111, Bonn, Germany.}
\author{Marco Zambon}
\email{marco.zambon@kuleuven.be}
\address{KU Leuven, Department of Mathematics, Celestijnenlaan 200B box 2400, BE-3001 Leuven, Belgium.}


 \title[ ]{Observables on multisymplectic manifolds\\ and higher Courant algebroids}

\begin{abstract}
Let $\omega$ be a closed, non-degenerate differential form of arbitrary degree. Associated to it there are an  $L_{\infty}$-algebra of observables, and an $L_{\infty}$-algebra of sections of the higher Courant algebroid twisted by $\omega$. 
Our main result is the existence of an  $L_{\infty}$-embedding of the former into the latter. We display explicit formulae for the embedding, involving the Bernoulli numbers. 
When $\omega$ is an integral symplectic form, the embedding can be realized geometrically via the prequantization construction, and when $\omega$ is a 3-form the embedding was found by Rogers in 2010. 
Further, in the presence of homotopy moment maps, we show that the embedding is compatible with gauge transformations.
\end{abstract}

\maketitle

\setcounter{tocdepth}{1} 
\tableofcontents

\section*{Introduction}
Before describing in full generality our main results, which hold for closed non-degenerate differential forms of \emph{any degree}, we explain them in the case of 2-forms.

{\bf The symplectic case.}
Given a symplectic manifold $(M,\omega)$, there are two Lie algebras that one can associate to it:
\begin{itemize}
\item [a)] $(C^{\infty}(M),\{\;,\;\})$, the smooth functions on $M$ endowed with the Poisson bracket.  
\item [b)] The sections $\vX(M)\oplus C^{\infty}(M)$ of the Lie algebroid $TM\oplus \RR$, with Lie bracket  twisted by $\omega$ as follows: $$\left[\pair{X}{f},\pair{Y}{g\vphantom{f}}\right]_{{\omega}}:=\pair{[X,Y]}{X(g)-Y(f)+{\iota_X\iota_Y\omega}}~.$$	
\end{itemize}
We denote these Lie algebras respectively by $C^{\infty}(M)_{{\omega}}$ and  by $\Gamma((TM\oplus \RR)_{\omega})$.
The relation between them is provided by the following   injective Lie algebra morphism, which 
depends only on $\omega$ :
\begin{equation}\label{eq:intropreqmap}
	\morphism{\psimap_{\omega}}
	{C^{\infty}(M)_{\omega}}
	{\Gamma(TM\oplus \RR)_{{\omega}}}
	{f}
	{\pair{v_f}{f}}
	~,
\end{equation}
where $v_f$ is the Hamiltonian vector field of $f$.

This Lie algebra embedding has a clear geometric interpretation whenever ${\frac{1}{2\pi}}[\omega]$ is an integer cohomology class.
In that case, one can make a choice of prequantization circle bundle with connection {having curvature $\omega$}. 
From this, one constructs a Lie algebra embedding (called \emph{prequantization map}) of the Poisson algebra of functions on $M$ into the invariant vector fields on the prequantization  bundle. 
The latter form the sections of a Lie algebroid over $M$, called Atiyah algebroid, which is isomorphic to the central extension $(TM\oplus \RR)_{\omega}$ of the tangent bundle $TM$. 
Hence we obtain a Lie algebra embedding   
$C^{\infty}(M)_{{\omega}}\to \Gamma((TM\oplus \RR)_{\omega})$, which agrees with \eqref{eq:intropreqmap} above.

{If $\omega'$ is another symplectic form cohomologous to $\omega$, in general there is no way to compare the Lie algebra embedding \eqref{eq:intropreqmap} to the one associated with $\omega'$, because the domains are different as Lie algebras. We therefore restrict the comparison to finite dimensional Lie subalgebras of the Poisson algebra of functions, as we now explain.}
Assume $(M,\omega)$ is endowed with a moment map for the action of some connected Lie group $G$, whose corresponding comoment map {(a Lie algebra morphism)} we denote $J^*\colon\g\to C^{\infty}(M)$.
Any choice of invariant 1-form $\alpha\in \Omega^1(M)^G$ provides another $G$-invariant symplectic form $\omega+d\alpha$, assuming this is non-degenerate; {notice that all symplectic structures cohomologous to $\omega$ are of this form, when $G$ is compact.}
The choice of $\alpha$ can be used to twist some of the above data, obtaining in particular a moment map   $J_{\alpha}$  for $\omega+d\alpha$. It turns out that the following diagram   commutes: 
\begin{equation}
	\label{intro:diag:main}
	\begin{tikzcd}[column sep=huge]
		&
		C^{\infty}(M)_{\omega}   \ar[r,"{\psimap_{\omega}}"] 
		&
		\Gamma(TM\oplus \RR)_{\omega}  \ar[dd,"\tau_\alpha"]\ar[dd,sloped,"\sim"']
		\\[-1em]
		\mathfrak{g}\ar[ru,"J^*"] \ar[dr,"J^*_{\alpha}"']
		\\[-1em]
		&
		C^{\infty}(M)_{\omega+d\alpha} \ar[r,"{\psimap_{\omega+d\alpha}}"] 
		&
		\Gamma(TM\oplus \RR)_{\omega+d\alpha}
	\end{tikzcd}
\end{equation}
Here $\tau_{\alpha}$ is the  gauge transformation of the Atiyah algebroids induced by $\alpha$, and is a Lie algebroid isomorphism.
We interpret this commutativity by saying that the two Lie algebra morphisms from $\g$ to the Atiyah algebroids agree, upon applying an isomorphism of the latter.
In a very loose sense, one could say that the Lie algebra morphism $\psimap_{\omega}$ only depends on the cohomology class $[\omega]\in H^2(M,\RR)$, once it is restricted to suitable final dimensional Lie subalgebras.
When ${\frac{1}{2\pi}}[\omega]$ is an integer cohomology class, the commutativity of the diagram
\eqref{intro:diag:main} can be shown with a geometric argument, which we provide in \S \ref{subsec:second}.

\bigskip
{\bf Main results.}
In this paper we show that the existence of the above Lie algebra embedding and commutative diagram extend to the setting of higher geometry, i.e. replacing $\omega$ by a  multisymplectic $(n{+}1)$-form (no integrality condition is required). In that case the Poisson algebra of functions $C^{\infty}(M)$ is replaced by a $L_{\infty}$-algebra \cite{LadaMarkl}\cite{LadaStasheff}, the Atiyah Lie algebroid   by a ``higher Courant algebroid'' $TM \oplus \wedge^{n{-}1} T^\ast M$, and $\alpha$ by an invariant $n$-form $B$ (see \S \ref{sec:multisympl}).

Our main results, which hold for arbitrary $n\ge 1$, are the following. 
\begin{itemize}
\item[A)] In Theorem   \ref{thm:iso}  we construct explicitly  an  embedding of $L_{\infty}$-algebras 
$$\psimap_{\omega}\colon L_{\infty}(M,\omega) \to		L_{\infty}(TM \oplus \wedge^{n{-}1} T^\ast M,\omega).$$
\item[B)] Building on this, in Theorem \ref{thm:comm}, we extablish that the higher version of the above diagram \eqref{intro:diag:main} commutes.
  \begin{equation*}
	\begin{tikzcd}
		&
		L_{\infty}(M,\omega)  \ar[r,"{\psimap_{\omega}}"] 		
		&
		L_{\infty}(TM \oplus \wedge^{n{-}1} T^\ast M,\omega) \ar[dd,"\tau_B"]\ar[dd,sloped,"\sim"']
		\\[-1em]
		\mathfrak{g}\ar[ru,"f"] \ar[dr,"f_B"']
		\\[-1em]
		&
		L_{\infty}(M,\widetilde{\omega})  \ar[r,"{\psimap_{\widetilde{\omega}}}"] 
		&
		L_{\infty}(TM \oplus \wedge^{n{-}1} T^\ast M,\widetilde{\omega})
	\end{tikzcd}
\end{equation*}
Here $\widetilde{\omega}:=\omega+dB$. The left arrows are homotopy moment maps in the sense of \cite{Callies2016},
and their twisting with respect to a closed invariant $n$-form was introduced in  \cite[\S 7.2]{Fregier2015}.
\end{itemize}
With minor variations, these results hold removing the non-degeneracy assumption, i.e.  for all closed $(n+1)$-forms.

The two $L_{\infty}$-algebras appearing in statement A)
have different origins: the codomain arises from a derived bracket construction, while the domain does not.
Our method to prove A) relies on the description of $L_{\infty}$-algebras as suitable coderivations. We make an Ansatz, with undetermined coefficients, and we check that a certain choice of coefficients (closely related to the Bernoulli numbers)  yields the desired $L_{\infty}$-embedding, see \S \ref{Section:ExtendedRogersEmbedding}. To do so, we need several computations in terms of the Nijenhuis-Richardson product (essentially, the composition of coderivations), which we carry out in Appendix \ref{app:Alg}. This method is constructive, guarantees that $\psimap_{\omega}$
is an $L_{\infty}$-embedding,  
 and yields explicit formulae. With this method, we need to carry out some   computations in Cartan calculus only for a few basic cases.
 We remark that, if one insisted in using  exclusively Cartan calculus,   checking that $\psimap_{\omega}$ is an $L_{\infty}$-algebra morphism would become an unmanageable task.

Our method for the proof of result B) is also based on coderivations and the Nijenhuis-Richardson product, see \S \ref{sec:gaugetrafos}. Curiously, the proof of result B) turns out to be very helpful in order to prove A), since it
provides us with a preferred choice of coefficients in the Ansatz for  A), which turns out to be the correct one.

{For both results, we apply known identities satisfied by the Bernoulli numbers:  the Elezovic summation formula,
the standard recursion formula, 
and the Euler product sum identity.}

\bigskip
{\bf Relation to the literature.}
We relate our first result above with the literature.
 In the special case $n=2$, the Atiyah algebroid is an instance of Courant algebroid, and the embedding was established by Rogers \cite[Theorem 7.1]{Rogers2013}. {For arbitrary $n$ the existence of the embedding is stated by S\"amann-Ritter in their preprint \cite[Theorem 4.10.]{Ritter2015a}. They provide a proof in which the embedding is constructed recursively, but not all steps are worked out explicitly, and they do not give a closed formula for the embedding. For a different approach in
the case of integral multisymplectic forms, involving a choice of open cover on the manifold $M$, see   Fiorenza-Rogers-Schreiber 
 \cite[\S 5]{FiorenzaRogersSchrLocObs}.} There, as a side-result, the authors include an approach for the construction of the embedding via the diagram in \cite[Proposition 5.10]{FiorenzaRogersSchrLocObs}, but work it out only for $n=2$.
 
 {Our second result above, to the best of our knowledge, has not been addressed in the literature yet}.

\bigskip
{\bf A geometric motivation. } 
 Our two main results, together with the fact that in the integral symplectic case these results can be obtained in a  geometric fashion from prequantization circle bundles (we do this in  \S \ref{sec:symplcase}),  support the hope that
 for arbitrary values of $n$    there might be a global geometric picture (``higher prequantization'') for integral forms that parallels the classical one.
 
In the integral case, ``higher Courant algebroids'' (also known as Vinogradov algebroids) can be obtained from $S^1$-gerbes  with connection
 and higher analogues, see \cite[\S 3.8]{Gualtieri2004}.
There are various notions of  higher prequantization in the literature, however none of the them seems to share all the good properties of classical prequantization.
In the integral case the analogue of the prequantization map  $\Preq_{\theta}$   of equation \eqref{eq:preq}  was already established  for all $n$ in Fiorenza-Rogers-Schreiber \cite[Thm. 4.6]{FiorenzaRogersSchrLocObs}; there however the higher prequantization bundle is described by means of a choice of open cover on the manifold $M$.  For $n=2$,  the analogue of   {the map $\Preq_{\theta}$} was established on a higher prequantization bundle that admits a global description (without choosing a cover)
in Krepski-Vaughan \cite[\S 5.1]{krepski2020multiplicative} -- but not as an $L_{\infty}$-algebra morphism -- and  in Sevestre-Wurzbacher
\cite[Thm. 3.5]{SevestreWurzbacherPreq} -- but using an infinite dimensional bundle --. See \cite[Remark 5.2]{krepski2020multiplicative}  
and \cite[Rem. 3.6]{SevestreWurzbacherPreq}  
for a comparison.

\bigskip 
{\bf Structure of the paper:}
For the special case of integral symplectic forms, in \S \ref{sec:symplcase} we use the prequantization scheme to give a geometric derivation of our main results. 
This section has a motivational purpose, and might be skipped on a first reading. 
In \S \ref{sec:background} we concisely review   $L_{\infty}$-algebras and their morphisms in terms of coalgebras, and in
\S \ref{sec:multisympl} we review the basics of multisymplectic geometry and the two $L_{\infty}$-algebras associated to a multisymplectic form.
 Our main contributions are in \S \ref{Section:ExtendedRogersEmbedding}  and  \S \ref{sec:gaugetrafos},  where we prove respectively Result A) and Result B) above. 
 The proofs rely heavily on computations which are best done in terms of the Nijenhuis-Richardson product of multilinear maps (it corresponds to the composition of coderivations), and which we carry out in Appendix \ref{app:Alg}. 
 In Appendix \ref{sec:system} we provide the proof of a proposition needed to obtain Result A).

\bigskip 
{\bf Acknowledgements:}  
This work has been partially supported by the long term structural funding -- Methusalem grant of the Flemish Government, the FWO research project G083118N,   the 
FWO-FNRS under EOS Project No. G0H4518N and under 
  EOS Project G0I2222N (Belgium) and the National Group for Algebraic and Geometric Structures and their Applications (GNSAGA – INdAM, Italy).
{A.M.M. thanks the Max Planck Institute for Mathematics in Bonn for its hospitality and financial support.
The authors thank Christian Blohmann, Chris Rogers, Leonid Ryvkin, Christian S\"amann, Gabriel Sevestre, Luca Vitagliano, and Tilmann Wurzbacher for helpful and motivating conversations.}

\section{Geometric motivation: the symplectic case}\label{sec:symplcase}

In the integral symplectic case, we provide geometric arguments for the existence of the Lie algebra embedding \eqref{eq:intropreqmap} and the commutativity of the diagram \eqref{intro:diag:main} appearing in the introduction.

Let $(M,\omega)$ be a connected symplectic manifold and  $$\pi\colon P\to M$$ be a principal $S^1$-bundle with connection 1-form $\theta\in \Omega^1(P)$ such that 
$d\theta=\pi^* \omega$.
That is, we assume that ${\frac{1}{2\pi}}[\omega]$ is an integral class and fix a prequantization circle bundle. (See {\cite[Ch. 6]{Cobordism} for a detailed discussion of prequantization.)}
We denote by $E\in \vX(P)$ the infinitesimal generator of the $S^1$-action and by $H_{\theta}:=\ker(\theta)$ the invariant Ehresmann connection on $P$ corresponding by $\theta$.

\begin{remark}
On the  symplectic manifold  $(M,\omega)$ we adopt the conventions that Hamiltonian vector fields are defined by $\iota_{v_f}\omega=-df$ and $\{f,g\}=\omega(v_f,v_g)$ (hence $f\mapsto v_f$ is a Lie algebra morphism). To shorten the notation, we denote by  $C^{\infty}(M)_{\omega}$ the Lie algebra $(C^{\infty}(M),\{\;,\;\})$.
\end{remark}

\subsection{Embedding of the observables in the Atiyah algebroid}\label{subsec:first}

In this subsection we give a geometric derivation of the map \eqref{eq:intropreqmap} in the integral case. {Part of the material reviewed here can be found also in \cite[\S 2]{Rogers2013}.}
 \subsubsection{Prequantization}

Denote by $$Q(P,\theta):=\{Y\in \vX(P):\cL_Y\theta=0\}$$ the Lie algebra of infinitesimal quantomorphisms, consisting of  vector fields on $P$ which preserve $\theta$.
(They are automatically $S^1$-invariant, since the generator $E$ of the action is determined by $\theta$ in virtue of $\theta(E)=1$, $\iota_Ed\theta=0$).

Prequantization  is a geometric procedure devised by Kostant and Souriau \cite{SouriauGeomQuant}\cite{Kost}, as a first step toward the quantization of a classical mechanical system. Prequantization provides a Lie algebra isomorphism
\begin{equation}\label{eq:preq}
	\morphism{\Preq_{\theta}}
	{C^{\infty}(M)_{\omega}}
	{Q(P,\theta)}
	{f}
	{v_f^{H_\theta}+\pi^*f\cdot E}~,
\end{equation}
where $X^{H_\theta}$ denotes the horizontal lift  of a vector field $X$ on $M$ using the Ehresmann connection $H_\theta$  (see \cite[Thm. 2.8]{VaughanJ} for a review). We refer to $  \Preq_{\theta}$ as \emph{prequantization map}; by letting the vector fields act on $S^1$-equivariant complex valued functions on $P$, it yields a faithful representation of the observables on $M$.

 \subsubsection{Atiyah algebroids}\label{subsec:At}
Consider again the principle circle bundle $\pi\colon P\to M$.
The {Atiyah} Lie algebroid $A_P$ is the transitive Lie algebroid over $M$ with space of sections given by $\vX(P)^{S^1}$, the invariant vector fields on $P$, and anchor given by $\pi_*$. {In other words, it is obtained taking the quotient of $TP\to P$ by the $S^1$-action.} It fits in a short exact sequence of Lie algebroids $$0\to\RR \to A_P\to TM\to 0,$$where 
{$\RR$ denotes the trivial rank-1 bundle over $M$} (a bundle of abelian Lie algebras, {with the constant section $1$ mapping to $E\in \Gamma(A_P)$}) and {the second map  is the anchor}.
A  connection $\theta$ on $P$ provides a linear splitting of this sequence,
i.e. an isomorphism of vector bundles 
\begin{equation}\label{eq:isoaty}
	\isomorphism{\sigma_\theta}
	{TM \oplus  \RR}
	{A_P}
	{\pair{v}{c}}
	{v^{H_\theta}+ c\cdot E}~.
\end{equation} 
Using this isomorphism to transfer the
Lie algebroid structure, we obtain the Lie algebroid  
$(TM\oplus \RR)_{{\omega}}$, with anchor map the first projection onto $TM$ and Lie bracket\footnote{To check this, use that $\theta([X^{H_\theta},Y^{H_\theta}])=-d\theta(X^{H_\theta},Y^{H_\theta})=-\pi^*(\omega(X,Y)).$}
 on sections 
 $$\left[\pair{X}{f},\pair{Y}{g\vphantom{f}}\right]_{{\omega}}:=
 \pair{[X,Y]}{X(g)-Y(f)+{\iota_X\iota_Y\omega}}~.$$	
Thus, $\sigma_\theta$ is a Lie algebroid isomorphism
$(TM\oplus \RR)_{{\omega}} \cong A_P$.

\subsubsection{Embedding}\label{subsec:emb}
Notice that $Q(P,\theta)\subset \vX(P)^{S^1} =\Gamma(A_P)$.
Composing the prequantization\footnote{In the following we will  sometimes view  the map
$\Preq_{\theta}$ of \eqref{eq:preq} as a map $C^{\infty}(M)_{\omega}\to \vX(P)^{S^1}$.
} 
map \eqref{eq:preq} \ with the inverse of the isomorphism \eqref{eq:isoaty} at the level of sections,
 we obtain a Lie algebra embedding
\begin{equation}\label{eq:chris}
	\morphism{\sigma_{\theta}^{-1}\circ \Preq_{\theta}}
	{C^{\infty}(M)_{\omega}}
	{\Gamma(TM\oplus \RR)_{{\omega}}}
	{f}{\pair{v_f}{f}}~,
\end{equation}
which does not depend on $\theta$. {This is exactly the map $\psimap_{\omega}$ appearing in \eqref{eq:intropreqmap}  in the introduction.}

  \begin{remark}
 We remind the reader that the above expression \eqref{eq:chris} is a Lie algebra embedding  
 even when $\omega$ does not satisfy the integrality condition.	
 \end{remark}

\subsection{Commutativity upon twisting}\label{subsec:second}
{In this subsection, assuming the integrality condition on $\omega$,  we give a geometric argument for the commutativity of the diagram \eqref{intro:diag:main}  in the introduction.}

 \subsubsection{Moment maps}\label{subsec:momaps}
Assume we have an action of a connected Lie group $G$ on $M$, and denote by $\rho\colon \g\to \vX(M)$ the corresponding infinitesimal action (a Lie algebra morphism).
Assume the existence of an equivariant 
 moment map  $$J\colon M\to \g^*.$$ This means that $J$ satisfies $\iota_{\rho(\xi)}\omega=-d(J^*(\xi))$ (i.e. $\rho(\xi)=v_{J^*(\xi)}$), for any $\xi\in\g$, 
and that $J^*\colon \g\to  C^{\infty}(M)_{\omega}$ is a Lie algebra morphism\footnote{This is equivalent to infinitesimal equivariance, i.e. 
$\cL_{\rho(\zeta)}J^*(\xi)=J^*([\zeta,\xi])$.}. 
Composing with the prequantization map \eqref{eq:preq} we obtain the Lie algebra morphism
$$L_0:=\Preq_{\theta}\circ J^*\colon \g \to Q(P,\theta)$$ lifting the infinitesimal action,  given by
 $L_0(\xi)=\rho(\xi)^{H_\theta}+\pi^*J^*(\xi)\cdot E$.

 \begin{remark}
The maps that appeared so far fit in the diagram of Lie algebras
 \begin{equation*}
	\begin{tikzcd}[column sep = huge]
		& C^{\infty}(M)_{\omega} \ar[rd,"{\Preq_{\theta}}"] \ar[d,dashed,"{v_{\bullet}}"]  
		&  
		\\
		\g \ar[r,"\rho"',dashed] \ar[ru,"{J^*}"] & 
		\vX(M) & 
		Q(P,\theta)  \ar[l,dashed,"{\pi_*}"]
	\end{tikzcd}
\end{equation*}

\end{remark}

 \subsubsection{Twisting by an invariant one-form}
Now we take  $\alpha\in \Omega^1(M)^G$ and use it to twist some of the above data, keeping the $G$-action fixed:
$\omega+d\alpha$ is an invariant symplectic form on $M$ (assuming it is non-degenerate), with moment map $J_{\alpha}$ determined by\footnote{Indeed it can be checked easily that $\iota_{\rho(\xi)}(\omega+d\alpha)=-d(J^*(\xi)+\iota_{\rho(\xi)}\alpha)$, using the $G$-invariance of $\alpha$  expressed as $\cL_{\rho(\xi)}\alpha=0$
for all $\xi\in \g$.} 
\begin{equation}\label{eq:jalpha}
	\morphism{J_{\alpha}^*}
	{\g}
	{C^{\infty}(M)_{\omega+d\alpha}}
	{\xi}
	{J^*(\xi)+\iota_{\rho(\xi)}\alpha}
	~.
\end{equation}
 Further, a prequatization of the symplectic manifold $(M,\omega+d\alpha)$ is given by the same circle bundle $P$ but with connection $\theta+\pi^*\alpha$.

We can repeat the  procedure outlined in \S \ref{subsec:momaps}, obtaining a  Lie algebra morphism 
$$L_\alpha:=\Preq_{\theta+\pi^*\alpha}\circ J^*_{\alpha}  ~\colon~ \g \to Q(P,\theta+\pi^*\alpha) $$ 
lifting the infinitesimal action. Since $\alpha$ is $G$-invariant,  any lift to $P$ of a generator $\rho(\xi)$ preserves $\pi^*\alpha$, hence we can view both $L_0$ and $L_{\alpha}$ as maps $$ \g \to Q(P,\theta)\cap Q(P,\theta+\pi^*\alpha)~.$$

\begin{lemma}\label{prop:L}
The   Lie algebra morphisms $L_0$ and $L_\alpha$ coincide.
\end{lemma}
\begin{proof}
Fix $\xi\in \g$  and write $f_\xi:=J^*(\xi)$. We have to show that $L_0(\xi)=L_{\alpha}(\xi)$, i.e.
$$\rho(\xi)^{H_\theta}+\pi^*f_\xi\cdot E=\rho(\xi)^{H_{\theta+\pi^*\alpha}}+\pi^*(f_\xi+\iota_{\rho(\xi)}\alpha)\cdot E.$$
We do so decomposing $TP$ as $H_\theta\oplus \RR E$. Since both the left hand side and the right hand side $\pi$-project to the same vector field (namely $\rho(\xi)$), it suffices to check  that   we obtain the same function applying $\theta$ to both vector fields. This is indeed the case, since
applying $\theta$ to the vector field on the right  we obtain
$$\pi^*(f_\xi+\iota_{\rho(\xi)}\alpha)-(\pi^*\alpha)(\rho(\xi)^{H_{\theta+\pi^*\alpha}})=\pi^*f_\xi ~.$$
\end{proof}

We can also repeat the construction of \Ss \ref{subsec:At}
using the connection $\theta+\pi^*\alpha$, {yielding a Lie algebroid isomorphism $$\sigma_{\theta+\pi^*\alpha}\colon (TM\oplus \RR)_{\omega+d\alpha} \cong A_P ~.$$}
The composition $(\sigma_{\theta+\pi^*\alpha})^{-1}\circ \sigma_{\theta}$ reads
\begin{equation}\label{eq:tau}
	\morphism{\tau_{\alpha}}
	{(TM\oplus \RR)_{\omega}}
	{(TM\oplus \RR)_{\omega+d\alpha}}
	{\pair{v}{c}}
	{\pair{v}{c+ \iota_{v}\alpha}}
\end{equation}
and is often referred to as  \emph{gauge transformation}.

\subsubsection{Commutativity}
We obtain a commutative diagram
\begin{displaymath}
	\begin{tikzcd}[column sep = huge]
		& 
		C^{\infty}(M)_{\omega} \ar[dr,"\Preq_{\theta}"] &
		&  
		\Gamma(TM\oplus \RR)_{\omega} \ar[dd,"\tau_{\alpha}"]
		\\
		\g \ar[rd,"J_{\alpha}^*"] \ar[ru,"J^*"] & 
		&
		\vX(P)^{S^1} \ar[ur,"\sigma_{\theta}^{-1}"] \ar[dr,"\sigma^{-1}_{\theta+\pi^*\alpha}"] 
		&
		\\
		& 
		C^{\infty}(M)_{\omega+d\alpha} \ar[ur,"\Preq_{\theta+\pi^*\alpha}"]   
		&
		&  
		\Gamma(TM\oplus \RR)_{\omega+d\alpha}
	\end{tikzcd}
\end{displaymath}
where the left portion commutes by  Lemma \ref{prop:L} and the right one by the paragraph following it.

The composition $\sigma_{\theta}^{-1}\circ \Preq_{\theta}\colon  C^{\infty}(M)_{\omega} \to
\Gamma(TM\oplus \RR)_{\omega}$   does not depend on $\theta$, as we saw in \S \ref{subsec:emb}. Hence removing $\vX(P)^{S^1}$ from the above diagram, we obtain a commutative diagram that makes no reference to the prequantization bundle $P$:

	\begin{equation}\label{diag:main}
		\begin{tikzcd}[column sep=huge]
			&
			  C^{\infty}(M)_{\omega}   \ar[r,"\psimap_{\omega}"] 
			&
			 \Gamma(TM\oplus \RR)_{\omega}  \ar[dd,"\tau_\alpha"]
			\\[-1em]
			\mathfrak{g}\ar[ru,"J^*"] \ar[dr,"J^*_{\alpha}"']
			\\[-1em]
			&
			 C^{\infty}(M)_{\omega+d\alpha} \ar[r,"\psimap_{\omega+d\alpha}"] 
			&
			\Gamma(TM\oplus \RR)_{\omega+d\alpha}
		\end{tikzcd}
	\end{equation}
This is precisely diagram \eqref{intro:diag:main} from the introduction.
This concludes the geometric proof of the commutativity of this diagram.

 \begin{remark}\label{rem:nosquare} For a given $\alpha\in \Omega^1(M)^G$, in general, there is no linear map  $C^{\infty}(M)_{\omega}\to C^{\infty}(M)_{\omega+d\alpha}$ making the right part of diagram \eqref{diag:main} commute. Indeed such a map exists if{f} for all $f\in C^{\infty}(M)$ we have $v_f^{\omega}=v_{f+\iota_{v^{\omega}_f}\alpha}^{\omega+d\alpha}$  (where the superscripts denote the symplectic form w.r.t. which we take the Hamiltonian vector field), or equivalently $\cL_{v_f^{\omega}}\alpha=0$, which is  a very strong condition. This explains why we are forced to consider moment maps.
\end{remark}
 
 \begin{remark}\label{rem:symcomm}
We remind the reader that diagram \eqref{diag:main} commutes for any symplectic form $\omega$, even for one that does not satisfy the integrality condition and therefore does not admit a prequantization bundle.
{This is immediate using the explicit expressions for the maps involved in equation \eqref{eq:chris}, 
\eqref{eq:jalpha} and \eqref{eq:tau}.}
\end{remark}

\section{Background on $L_{\infty}$-algebras}\label{sec:background}

We present some background material on $L_{\infty}$-algebras \cite[\S 3]{LadaStasheff}\cite{LadaMarkl}.
We follow  the``coalgebraic' presentation of \cite[Ch. 2, \S 1]{FloDiss}, see also \cite[\S 6]{MarklDoubek} and \cite[Ch. 1]{Miti2021a}. We conclude this section with the key Remark \ref{rem:exp}, describing   the pushforward of an $L_\infty[1]$-structure along the exponential of a degree $0$ coderivation.

\subsection{Coderivations and morphism of coalgebras}
Let $V$ be a $\ZZ$-graded vector space. 
We denote by $ V^{\odot k}$ the $k$-fold symmetric tensor product.
The graded symmetric algebra $S^{\ge 1}V := \oplus_{k\ge 1} V^{\odot k}$ carries a natural structure of (coassociative) coalgebra, with coproduct $\Delta\colon S^{\ge 1}V \to S^{\ge 1}V\otimes S^{\ge 1}V$ given by deconcatenation.
Such a coalgebra takes the name of \emph{free (reduced) symmetric tensor coalgebra}  \cite[\S 1.5]{Manetti2011}.

\subsubsection{Coderivations}
\begin{definition}
  A \emph{degree $k$ coderivation} of the coalgebra $S^{\ge 1}V$ is a degree $k$ linear map $C\colon S^{\ge 1}V \to S^{\ge 1}V$ such that $\Delta\circ C=(C\otimes Id+Id\otimes C)\circ \Delta$.
\end{definition}
To explain the terminology, notice that this equation is what one obtains dualizing the property of being a derivation of an algebra.

\begin{lemma}\cite[Lemma 2.4]{LadaMarkl}\label{lem:coder}
  There is a bijection between degree $k$  coderivations and degree $k$ linear maps $m\colon S^{\ge 1}V\to V$.
Given $m$, define $C$ (the unique extension of $m$ to a degree 1 coderivation) by 
\begin{equation}\label{eq:coder}
	C(x_1,\dots,x_n):= \sum_{i=1}^n \sum_{~\sigma \in \ush{i,n-i}}
	\epsilon(\sigma)~ m_{i}(x_{\sigma_1},\dots,x_{\sigma_{i}})\odot x_{\sigma_{i+1}}\odot\dots\odot x_{\sigma_n}
	~,
\end{equation}
where $\ush{i,j}$ denotes the subgroup of $(i,j)$-unshuffles in the permutation group $S_{i+j}$.
The inverse is given by the corestriction with respect to the canonical projection $\pr_V\colon S^{\ge 1}V\to V$. 
Namely, for any given a degree $k$  coderivation $C$, $m:=\pr_V\circ C$.
\end{lemma}
In the following, given a homogeneous linear map $m \colon \colon S^{\ge 1}V\to V$, we denote by $C_m$   the corresponding coderivation of $S^{\ge 1}V$.
\medskip 
 
The composition   of two   coderivations $C_1$ and $C_2$ is a linear map $C_1\circ C_2\colon S^{\ge 1}V \to S^{\ge 1}V$, which fails to be a derivation.
However the graded commutator
$$[C_1,C_2]:=C_1\circ C_2-(-1)^{|C_1||C_2|}C_2\circ C_1$$
is a coderivation of degree $|C_1|+|C_2|$. The space of coderivations, equipped with the graded commutator, is a graded Lie algebra. 
 
\begin{definition}
The \emph{Nijenhuis-Richardson product} of two linear maps $a,b:S^{\geq 1}V \to V$ is
\begin{equation}\label{eq:compsymm}
 a\cs b:= \pr_V(C_a\circ C_b).
\end{equation}  
\end{definition}
Explicitly, this is obtained by summing insertions of $b_j$ in $a_i$, where $a_i:=a \vert_{V^{\odot i}}$ (see equation \eqref{Eq:RNProducts} in the appendix for a more explicit formula).
Notice that 
$C_{a\cs b}\neq C_a\circ C_b$ (the latter is not even a coderivation in general). 
{The composition $\cs$  is not associative, and  makes the space of linear maps $ \colon S^{\ge 1}V\to V$ into a graded right pre-Lie algebra (see for instance \cite[\S 1.1]{Bandiera2016}). In particular, the graded commutator  $[\cdot,\cdot]$ w.r.t. $\cs$ is a graded Lie bracket.}
The correspondence of Lemma \ref{lem:coder} between linear maps $ S^{\ge 1}V\to V$ and coderivations preserves the commutator bracket: 
\begin{lemma}\label{lem:commutators}
  Given homogeneous  linear maps $a,b \colon \colon S^{\ge 1}V\to V$, we have
$$[C_a,C_b]=C_{[a,b]}.$$
\begin{proof}
 Since $[C_a,C_b]$ is a coderivation, it is the coderivation corresponding to $\pr_V([C_a,C_b])$. One computes
 $$\pr_V([C_a,C_b])=\pr_V( C_a\circ C_b)-(-1)^{|a||b|}\pr_V(C_b\circ C_a)=a\cs b-(-1)^{|a||b|}b\cs a = [a,b].$$
\end{proof}
 \end{lemma}

\subsubsection{Morphism of coalgebras}
\begin{definition} 
A \emph{morphism of coalgebras} $F\colon S^{\ge 1}V\to S^{\ge 1}W$ is a degree $0$ linear map satisfying $(F\otimes F)\circ \Delta=\Delta\circ F$.
\end{definition}

\begin{lemma}\label{lem:morcoalg}
   There is a bijection between 
morphisms of coalgebras  and  degree $0$ linear maps $f\colon S^{\ge 1}V\to W$.
\end{lemma}
Given $F$, define $f:=\pr_W\circ F$.   Given $f$, there is an explicit formula for   $F$ (the unique extension of $f$ to a morphism of coalgebras). We will not need the explicit formula, and refer the interested reader to \cite[Rem. 1.13]{FloDiss} or \cite[Ch. A, \S 3.3]{Miti2021a}.
   
\begin{remark}\label{rem:codermor}
 If $C$ is a degree $0$ coderivation of $S^{\ge 1}V$, then $e^C$ is a morphism of coalgebras, provided it converges.  \ Endow $S^{\ge 1}V$ with the filtration $\mathcal{F}_1\subset \mathcal{F}_2\subset\dots$ where $\mathcal{F}_k:=\oplus_{i=1}^k V^{\odot i}$. Whenever $C$ maps $\mathcal{F}_k$ to $\mathcal{F}_{k-1}$ for all $k$ (defining 
 $\mathcal{F}_0=\{0\}$), it follows that $e^C-Id$ has the same property. 
 Therefore $e^C|_{V^{\odot n}}$ is a finite sum for all 
$n$, and $e^C$ converges.

{The coderivation $C$ satisfies the above property whenever, writing $C=C_p$ for a degree $0$ linear map $p\colon S^{\ge 1}V\to V$, we have $p|_{V}=0$. This follows from equation \eqref{eq:coder}.}
\end{remark}

\subsection{$L_{\infty}[1]$-algebras}\label{subsec:Linfty1}
Let $V$ be a $\ZZ$-graded vector space.
\begin{definition}
  An \emph{$L_{\infty}[1]$-algebra structure} on $V$ is a degree $1$ linear map $m\colon S^{\ge 1}V\to V$ such that the corresponding degree $1$ coderivation $Q$ of $S^{\ge 1}V$, as in Lemma \ref{lem:coder}, 
   satisfies $[Q,Q]=0$.
\end{definition}

It is customary to encode the  $L_{\infty}[1]$-algebra structure
 $m$ by the family of ``multibrackets'' $\{m_k\}_{k\ge 1}$, where $m_k:=m|_{V^{\odot k}}\colon V^{\odot k} \to V$. 
 The condition $[Q,Q]=0$ then translates into a hierarchy of quadratic relations for the multibrackets. 

Recall that an  $L_{\infty}$-algebra structure on a graded vector space $U$ consists of degree $2-k$ linear maps $ U^{\wedge n}\to U$, for all $k\ge 1$, satisfying certain quadratic relations (usually called ``higher Jacobi identities'').
Given a graded vector space $U$,
an $L_{\infty}[1]$-algebra structure on $U[1]$ is equivalent to a $L_{\infty}$-algebra structure on $U$, via the d\'ecalage isomorphism.
The \emph{d\'ecalage isomorphism} is
\begin{equation}\label{deca}
	\isomorphism{dec}
	{\Big(U^{\wedge n}\Big)[n]}
	{\Big(U[1]\Big)^{\odot n}}
 	{u_1\cdots u_n}
 	{u_1\cdots u_n\cdot (-1)^{(n{-}1)|u_{1}|+\dots+2|u_{n-2}|+|u_{n{-}1}|}}~,
 \end{equation}
 where $u_1,\dots,u_n\in U$ are homogeneous and $|u_i|$ denote the degrees of $u_i\in U$.

\begin{definition}
Given $L_{\infty}[1]$-algebra structures on $V$ and $W$, denote the corresponding coderivations by $Q_V$ and $Q_W$. A \emph{morphism of $L_{\infty}[1]$-algebras} from $V$ to $W$ is degree $0$ linear map $f\colon S^{\ge 1}V\to W$ such that the corresponding
 morphism of coalgebras $F\colon S^{\ge 1}V\to S^{\ge 1}W$,
 as in Lemma \ref{lem:morcoalg}, satisfies
$F\circ Q_V=Q_W\circ F$.
\end{definition}

The following remark will be important in the sequel.
\begin{remark}[{Pushforward of $L_{\infty}[1]$-algebra structures}]\label{rem:exp}
Let $(V,m)$ be an $L_{\infty}[1]$-algebra, and denote the corresponding degree 1 coderivation   by $Q:=C_m$. 
 A degree $0$ linear map
$p\colon S^{\ge 1}V\to V$ gives rise to a degree $0$ coderivation $C_p$ of $S^{\ge 1}V$, by Lemma \ref{lem:coder}. 
In turn the latter gives rise to a morphism of coalgebras
$e^{C_p}$, {assuming it converges}, by Remark \ref{rem:codermor}.
Define $$Q':=e^{C_p}\circ Q\circ e^{-C_p}.$$
One checks easily that 
\begin{itemize}
\item $Q'$ is also a coderivation on $S^{\ge 1}V$, and thus corresponds to a new $L_{\infty}[1]$-algebra structure $m'$ on $V$;
\item $e^{C_p}$ intertwines $Q$ and $Q'$, hence it corresponds to
an $L_{\infty}[1]$-isomorphism $\Phi$ from $(V,m)$ to $(V,m')$.
\end{itemize}
One can view $Q'$ as the push-forward of the $L_\infty[1]$-structure $m$ along the $L_{\infty}[1]$-isomorphism $\phimapshift$, given in terms of $p$.

Explicitly, one has
 \begin{align}\label{eq:pushbrackets}
	m'
	=&~\pr_V(Q')
	=\pr_V\left(Q+[C_p,Q]+\frac{1}{2!}[C_p, [C_p,Q]]+\dots \right)\\
	=&~ m+[p,m]+\frac{1}{2!}[p, [p,m]]+\dots
	\nonumber
\end{align}
{using Lemma \ref{lem:commutators},} and
 $$
	\phimapshift=\pr_V(e^{C_p})=
	{\pr}_V + \sum_{n\geq 1} \frac{1}{n!}p^{\cs n}=
	\pr_V+p+\frac{1}{2!}(p\cs p)+\dots 
 $$
where 
$[\cdot,\cdot]$ denotes the graded commutator w.r.t. $\cs$.
 In components 
 \begin{equation}\label{eq:pushforwardLinftymorph}
 	\begin{split}
	\phimapshift_1 
	=&~
	\id_V+\sum_{n\geq 1}  \frac{1}{n!} ~ (p\vert_V)^{ n}~,
	\\
 	\phimapshift_k 
 	=&~
 	\sum_{n\geq 1} \frac{1}{n!} ~ p^{\cs n} \vert_{V^{\odot k}}~
 	\qquad,~ \forall k\geq 2
 	~. 	
 	\end{split}
 \end{equation}
 In particular $\phimapshift_1 = \id_V$ when $p\vert_V : V \to V$ is equal to zero.
\end{remark}

\section{Two $L_{\infty}$-algebras associated to multisymplectic forms}
\label{sec:multisympl}
{There are two distinct  $L_\infty$-algebras associated to any multisymplectic manifold. One, due to Rogers \cite{Rogers2010}, is the higher analogue of the Poisson algebra of observables associated to a  symplectic manifold. The other, which was introduced in \cite{Zambon2012} relying on work of Fiorenza-Manetti and Getzler, is the higher analogue of the  Atiyah algebroid mentioned in the introduction. }

In this section, after reviewing the basics of multisymplectic geometry, we recall the construction of the two above-mentioned $L_\infty$-algebras.

\subsection{Multisymplectic manifolds}

We recall few classical notions in multisymplectic geometry.
\begin{definition} 
An \emph{$n$-plectic form} on a manifold $M$ is an $(n{+}1)$-form $\omega$ which is closed (i.e. $d\omega=0$) and non-degenerate, in the sense that the bundle map $TM\to \wedge^n T^*M, v\mapsto \iota_v\omega$ is injective.
\end{definition}
Notice that for $n=1$ one recovers the notion of symplectic manifold.

\begin{definition} \label{Hamiltonian}
A $(n{-}1)$-form $\alpha$ is \emph{Hamiltonian} if  there exists a (necessarily unique) vector field $v_\alpha \in \mathfrak{X}(M)$ such that $$d \alpha= -\iota_{v_\alpha} \omega\, .$$  \end{definition}

We denote the vector space of Hamiltonian $(n{-}1)$-forms by $\Omega^{n{-}1}_{\mathrm{Ham}}(M,\omega)$. 
Mimicking the construction of the Poisson bracket in symplectic geometry, one notices that there is a well-defined bilinear bracket $\left\{\cdot , \cdot\right\} : \Omega^{n{-}1}_{\mathrm{Ham}}(M)\times \Omega^{n{-}1}_{\mathrm{Ham}}(M)\to \Omega^{n{-}1}_{\mathrm{Ham}}(M)$  given by
\begin{equation*}
\left\{\alpha, \beta\right\}  = \iota_{v_{\beta}}\iota_{v_{\alpha}}\omega\,.\end{equation*} 
The bracket is skew-symmetric, but fails to satisfy the Jacobi identity, hence it is not  a Lie bracket.  In the next subsection we recall a way to make up for this failure.

\subsection{Rogers' $L_{\infty}$-algebra}
Let $(M,\omega)$ be an $n$-plectic manifold.  
Rogers  \cite[Thm. 5.2]{Rogers2010} associated an $L_{\infty}$-algebra to $(M,\omega)$. which we denote by $L_{\infty}(M,\omega)$.

\begin{definition}  
Given an $n$-plectic manifold $(M,\omega)$, the \emph{observables} form  an $L_{\infty}$-algebra, denoted
$L_{\infty}(M,\omega):=(\Lspace,\{l_{k} \})$. The underlying graded vector space  is given by 
\[
\Lspace^i =
\begin{cases}
\Omega^{n{-}1}_{\mathrm{Ham}}(M,\omega)  & i=0,\\
\Omega^{n{-}1+i}(M)  & -n{+}1 \leq i < 0.
\end{cases}
\]
The only non-vanishing multibrackets, up to permutations of the entries,
are the following: 
\begin{itemize}
\item {Unary bracket:}  whenever $\deg({\alpha})<0$,
\[ 
l_{1}(\alpha)=d\alpha.
\]
\item Higher brackets: for all $k>1$ 
\[
l_{k}(\alpha_{1},\hdots,\alpha_{k}) =
\begin{cases}
0 & \text{if $\deg(\alpha_{1} \otimes \cdots \otimes \alpha_{k}) < 0$}, \\
\varsigma(k) \iota_{v_{\alpha_{k}}}\cdots\iota_{v_{\alpha_{1}}}\omega 
 & \text{if
  $\deg({\alpha_{1} \otimes \cdots \otimes \alpha_{k}})=0$}, 
  \end{cases}
\]
 where $v_{\alpha_{i}}$ is any Hamiltonian vector field
associated to $\alpha_{i} \in \Omega^{n{-}1}_{\mathrm{Ham}}(M,\omega)$ and $\varsigma(k)=-(-1)^{k(k+1)/2}$ is a sign prefactor. 
 \end{itemize}
 \end{definition}

{A \emph{Lie-$n$ algebra} is just an $L_{\infty}$-algebra whose underlying  
graded vector space is concentrated in degrees $-n+1,\dots,0$. (In this case, by degree reasons, only the first $n+1$ multibrackets can be non-trivial). Notice that  
$L_{\infty}(M,\omega)$ is a  Lie-$n$ algebra.}

\subsection{Vinogradov's $L_{\infty}$-algebra}\label{subsec:Vinogradov}
Let $M$ be a manifold and consider the vector bundle
\begin{displaymath}
	E^n := TM \oplus \wedge^{n{-}1} T^\ast M 
\end{displaymath}
for a fixed $n \geq 1$. Let us denote elements of $E^n$ as $e = \pair{X}{\alpha}$.
This vector bundle is endowed with the following structures: 

\begin{itemize}
\item A bundle map $\rho \colon E^n \to  TM$ given by the first projection,
\item {A map
$\pairing_{+} \colon E^n \otimes E^n \to \wedge^{n-2} T^\ast M $ given by 
$$\left(\pair{X_1}{\alpha_1},\pair{X_2}{\alpha_2}\right) \mapsto 
	 \frac{1}{2}\left( \iota_{X_1}\alpha_2 + \iota_{X_2}\alpha_1 \right)$$
	 }
\item a skew-symmetric bracket $[\cdot,\cdot]_C$ on  the sections  $\Gamma(E^n)$, called Higher Courant bracket, given by 
$$ \left(\pair{X_1}{\alpha_1},\pair{X_2}{\alpha_2}\right) \mapsto 	
		\pair{[X_1,X_2]}{\mathcal{L}_{X_1}\alpha_2 - \mathcal{L}_{X_2}\alpha_1 - \dd 
		\left\langle \pair{X_1}{\alpha_1},\pair{X_2}{\alpha_2}\right\rangle_-}	 
$$	 
{where $\pairing_{-}$ is defined in equation \eqref{eq:pairing-}.}
\end{itemize}

The bracket can be ``twisted'' by any closed form $\omega\in \Omega^{n{+}1}(M)$, by defining 
$$[e_1,e_2]_\omega =  [e_1,e_2]_C  + \pair{0}{\iota_{X_1}\iota_{X_2} \omega}.$$	 

Borrowing the terminology proposed by Ritter and S\"amann in \cite{Ritter2015a}, we introduce the following geometric structure:
\begin{definition}\label{def:Vinalgoid}
The  \emph{Vinogradov Algebroid twisted by $\omega$} consists of the  data\\
 $(TM\oplus \wedge^{n{-}1}T^\ast M ,  \rho, {\pairing_{+}}, [\cdot,\cdot]_\omega)$. 
 \end{definition}
We think of twisted Vinogradov Algebroids as ``higher Courant algebroid'', since in the case $n=2$ they coincide with the notion of (split) exact Courant algebroid.

The following anti-symmetric pairing is not part of the definition of Vinogradov algebroid, but it will play a fundamental role in this paper:
\begin{equation}
	\label{eq:pairing-}
	\morphism{\pairing_{-}}
	{E^n \otimes E^n}{\wedge^{n-2} T^\ast M}
	{\left(\pair{X_1}{\alpha_1},\pair{X_2}{\alpha_2}\right)}
	{\frac{1}{2}\left( \iota_{X_1}\alpha_2 -\iota_{X_2}\alpha_1 \right)}~.
\end{equation}

For any  twisted Vinogradov Algebroid there is an associated Lie $n$-algebra, worked out in  \cite[Prop. 8.1 and 8.4]{Zambon2012} relying on a result by Fiorenza-Manetti \cite{MapCone} and Getzler \cite{Getzler}.
Its higher multibrackets 
involve the   Bernoulli numbers. Recall that the \emph{Bernoulli numbers}  $B_{n}$ are defined by $\frac{x}{e^x-1}=\sum_{n=0}^{\infty}B_n\frac{x^n}{n!}.$ In particular $B_0=1, B_1=-\frac{1}{2}, B_2=\frac{1}{6}$, and $B_n=0$ for odd $n\neq 1$. See \cite{Weisstein} for a brief survey.

 \begin{definition}\label{def:vinolinfty}
Given a twisted Vinogradov Algebroid $(TM\oplus \wedge^{n{-}1}T^\ast M , \rho, {\pairing_{+}}, [\cdot,\cdot]_\omega)$, the \emph{associated Lie $n$-algebra} structure 
$L_{\infty}(E^n,\omega) = ({\Vspace}, \{\mu_k\})$
{has underlying graded vector space
\begin{displaymath}
	{\Vspace^i} =
	\begin{cases}
		\mathfrak{X}(M)\oplus \Omega^{n{-}1}(M)  &\quad i=0,\\
		\Omega^{n{-}1+i}(M) &\quad -n{+}1 \leq i < 0.
	\end{cases}
\end{displaymath}
} 
The actions of non vanishing multi-brackets (up to permutations of the entries) on arbitrary vectors $\Velem_i=f_i\oplus e_i \in {\Vspace}$,
with $e_i = \pair{X_i}{\alpha_i} \in \mathfrak{X}(M)\oplus \Omega^{n{-}1}(M) = \Gamma(E^n)$ and 
$f_i \in \bigoplus_{k=0}^{n-2}\Omega^k(M)$,
are given as follows:

\begin{itemize}
	\item unary bracket:
		$$\mu_1 \left(f\right) =   \dd f ~;$$
	\item binary bracket: 
		\begin{displaymath}
			\begin{split}
				\mu_2 \left(e_1,e_2\right) 	=&~ [e_1,e_2]_\omega 
				~;
				\\
				\mu_2 \left(e_1,f_2\right) =&~ 
				\frac{1}{2} \mathcal{L}_{X_1} f_2 
				~;
			\end{split}
		\end{displaymath}
	\item ternary bracket:
		\begin{displaymath}
			\begin{split}
				\mu_3 (e_1, e_2, e_3) =&~ -T_\omega(e_1,e_2,e_3) 
				:=~ 
				-\frac{1}{3} \langle[e_1,e_2]_\omega,e_3 \rangle_+ + \cyc
				~;
				 \\
				\mu_3 (f_1, e_2, e_3) =&~ 
				-\frac{1}{6}
				\left(
				\frac{1}{2}(\iota_{X_1}\mathcal{L}_{X_2} 
				- \iota_{X_2}\mathcal{L}_{X_1}) + \iota_{[X_1,X_2]}
				\right)f
				~;				
			\end{split}
		\end{displaymath}	
	\item $k$-ary bracket for $k \ge 3$ an \emph{odd} integer:

		\begin{equation}\label{eq:VinoMultibrakAllaZambon_1}
			\begin{split}
				\mu_k(\Velem_0,\cdots,\Velem_{k-1})
				=&
				\left(\sum_{i=0}^{k-1} {(-1)^{i-1}\mu_k(f_i+\alpha_i,X_0,\dots,\widehat{X_i},\dots,X_{k-1})}\right)
				+\\
				&+(-1)^{\frac{k+1}{2}} \cdot k \cdot B_{k-1} \cdot 
				\iota_{X_{k-1}}	\dots \iota_{X_{0}} \omega			
				~;
			\end{split}
		\end{equation}
		where 
		
		\begin{equation}\label{eq:VinoMultibrakAllaZambon_2}
			\begin{split}
			\mu_k&(f_0+\alpha_0,X_1,\dots,X_{n{-}1}) =
			\\
			&=
			~c_k
			\sum_{1\le i<j\le k-1}(-1)^{i+j+1}\iota_{X_{k-1}}\dots   
  			\widehat{\iota_{X_{j}}}\dots \widehat{\iota_{X_{i}}}\dots
				\iota_{X_{1}} ~ [f_0+\alpha_0,X_i,X_j]_3~.
			\end{split}
		\end{equation}			
In the above formula,	
$[\cdot,\cdot,\cdot]_3 = -T_0$ denotes the ternary bracket 
associated to the untwisted ($\omega=0$) Vinogradov Algebroid, and $c_k$ is a numerical constant
		\begin{equation}\label{eq:UglyCoefficient}
			c_k= (-1)^{\frac{k+1}{2}}\frac{12~B_{k-1}}{(k-1)(k-2)}.
		\end{equation}
				\end{itemize}
\end{definition}

\begin{remark}\label{Remark:semplifica-conti}
	Notice that for $k\ge 2$,
Vinogradov's brackets $\mu_k$ vanish unless $k-1$ entries are elements of degree zero.
{For $k\ge 2$, Rogers' brackets $\pi_k$ vanish unless all entries are elements in degree zero.}
\end{remark}

We introduce the notation (see \cite[\S 4]{Callies2016})
\begin{displaymath}
\Ham_{\infty}^{n{-}1}(M,\omega):=\left\lbrace\left.
		\pair{v_{\alpha}}{\alpha} \in \mathfrak{X}(M)\otimes \Omega^{n{-}1}_{\mathrm{Ham}}(M,\omega)
		~\right\vert~
		\dd \alpha = -\iota_{v_{\alpha}} \omega
\right\rbrace
\end{displaymath}
for the vector subspace of $\Gamma(E^n)$ whose elements are pairs consisting of a Hamiltonian form and its Hamiltonian vector field. 
 One can check that   $\Ham_{\infty}^{n{-}1}(M,\omega)$
is closed under the twisted higher Courant bracket $[\cdot,\cdot]_\omega$. 
 This implies that the $L_{\infty}$-algebra structure $L_{\infty}(E^n,\omega)$ restricts
	\begin{equation}\label{eq:A}
\Aspace:=	 
	\left(\bigoplus_{k=0}^{n-2}\Omega^k(M)\right)\oplus 
\Ham_{\infty}^{n{-}1}(M,\omega),
\end{equation}
yielding an $L_{\infty}$-algebra $(\Aspace, \mu_k)$. The latter is sometimes denoted by $\Ham_\infty(M,\omega)$ in the literature (see \cite[\S 4.1]{Callies2016}).

\section{Extending Rogers' embedding}\label{Section:ExtendedRogersEmbedding}

Let $(M,\omega)$ be a $n$-plectic manifold. 
In this section we state the main result of this paper, Theorem \ref{thm:iso}. 
This theorem provides  an explicit $L_{\infty}$-embedding from the $L_{\infty}$-algebra of observables on $(M,\omega)$ into the $L_{\infty}$-algebra associated to the   $\omega$-twisted Vinogradov algebroid $TM \oplus \wedge^{n{-}1} T^\ast M$.

\subsection{Strategy}\label{sec:Preparation}
 In this subsection we outline the strategy we will follow, {using freely the notation introduced in \S\ref{sec:multisympl}.}
By the non-degeneracy of $\omega$ we have a linear isomorphism
\begin{equation*}
  \Omega^{n{-}1}_{\mathrm{Ham}}(M,\omega)\cong \Ham_{\infty}^{n{-}1}(M,\omega),
\end{equation*}
given by $\alpha\mapsto\pair{v_{\alpha}}{\alpha}$. Using this in degree zero 
and the identity in negative degrees, we obtain an isomorphism of graded vector spaces
\begin{equation}\label{eq:LMomegaA}
\begin{tikzcd}[]
	\Lspace \ar[r,equal,"\sim"] &
	\Aspace \ar[r,hook] &
	\Vspace
\end{tikzcd}
~.
\end{equation}
We therefore have two distinct  $L_{\infty}$-algebra structures on $\Aspace$ (whose underlying  cochain complexes agree):
\begin{itemize}
\item[i)] $(\Aspace,\{\pi_k\})$, the one obtained transferring the multibrackets of Rogers' $L_{\infty}$-algebra  $L_{\infty}(M,\omega)$ via the linear isomorphism $\Lspace \cong \Aspace$. 
\item[ii)] $(\Aspace,\{\mu_k\})$, the one obtained restricting the multibrackets of the $L_{\infty}$-algebra
$L_{\infty}(E^n,\omega)$ associated to the $\omega$-twisted Vinogradov algebroid, see the text after Remark
\ref{Remark:semplifica-conti}. 
\end{itemize}

It is convenient to work with $L_{\infty}[1]$-algebras, by applying the d\'ecalage isomorphism \eqref{deca}.
Denote  by $(\Aspace[1],\{\bpi_k\})$ the $L_{\infty}[1]$-algebra corresponding to $(\Aspace,\{\pi_k\})$.  
Notice that applying the d\'ecalage isomorphism  to obtain the $\bpi_k$'s does not introduce any extra signs, since the higher multibrackets in Rogers' $L_{\infty}$-algebra vanish unless all entries have degree $0$.
Write $\bpi \colon  S^{\ge 1}(\Aspace[1])\to \Aspace[1]$ for the map with components $\bpi_k$ ($k\ge 1)$.
Similarly, denote  by $(\Aspace[1],\{\bmu_k\})$ the $L_{\infty}[1]$-algebra corresponding to $(\Aspace,\{\mu_k\})$,
and write $\bmu$ for the map with components $\bmu_k$ ($k\ge 1)$.  

{We want to construct an $L_{\infty}[1]$-isomorphism from $(\Aspace[1],\bpi)$ to $(\Aspace[1],\bmu)$.}
The idea is to apply the key Remark \ref{rem:exp}. Denote by $Q_{\bpi}$ the coderivation on $S^{\ge 1}(\Aspace[1])$ corresponding to the $L_{\infty}[1]$-algebra structure $(\Aspace[1],\bpi)$. 
For any   degree $0$ linear map $$p\colon S^{\ge 1}(\Aspace[1])\to \Aspace[1],$$
denoting by $C_p$ the corresponding degree $0$ coderivation  of $S^{\ge 1}\Aspace[1]$
and assuming that $e^{C_p}$ converges, we know that
\begin{itemize}
\item $ e^{C_p}\circ Q_{\bpi}\circ e^{-C_p}$ is a new coderivation, which corresponds to a new 
 $L_{\infty}[1]$-algebra structure $\bpi'$ on $\Aspace[1]$,   
 \item $e^{C_p}$  corresponds to
an $L_{\infty}[1]$-isomorphism $\Phi$ from $(\Aspace[1],\bpi)$ to $(\Aspace[1],\bpi')$.
\end{itemize}
The explicit formulae for $\bpi'$ and $\Phi$ were given in Remark \ref{rem:exp}.
We will show that $p$ can be chosen in such a way that $\bpi'=\bmu$.

\subsection{{An Ansatz}}\label{sec:ansatz}
In this subsection we  make a specific choice of degree $0$ linear map $p\colon S^{\ge 1}(\Aspace[1])\to \Aspace[1]$, namely the one given by equation \eqref{eq:PhikWithC} below with the choice of coefficients given in equation \eqref{eq:ExplicitCk}.

\begin{remark}\label{rem:pair-}
 In the following we will  employ the straightforward extension of the skew-symmetric pairing operator  $ \pairing_{-}$ of \Ss \ref{subsec:Vinogradov} from $\mathfrak{X}(M)\oplus \Omega^{n{-}1}$(M) to the whole graded vector space $\Vspace$, namely, the {graded}
 skew-symmetric bilinear map
 $\pairing_- : \Vspace \otimes \Vspace \to \Vspace$
given by
\begin{equation}\label{Eq:PairingExtensions}
	\left\langle f_1 \oplus \pair{X_1}{\alpha_1}, f_2 \oplus \pair{X_2}{\alpha_2} \right\rangle_- = \frac{1}{2}\left(
	\iota_{X_1}( \alpha_2 + f_2) -\iota_{X_2}( \alpha_1 + f_1)
	\right)
	~.
\end{equation} 
 Restricting to $\Aspace$ we obtain a graded skew-symmetric bilinear map $\Aspace\otimes \Aspace\to \Aspace$ of degree $-1$.

In turn the latter, by d\'ecalage, defines a degree zero graded symmetric bilinear map    $(\Aspace[1])^{\odot 2}\to \Aspace[1]$,
which vanishes if both entries of $\Aspace[1]$ lie in degrees $\le -2$. 
By extending trivially we obtain a degree zero map which we denote by $\bS$:
\begin{equation}\label{eq:defbs}
\bS\colon S^{\ge 1}(\Aspace[1])\to \Aspace[1]
\end{equation} 
\end{remark}

\begin{lemma}\label{lem:pansatazPhi}
Assume that 
\begin{equation*}
	p = \sum_{i=1}^\infty~ c_i ~ \bS^{i}
\end{equation*}
{where the $c_i$ are real numbers.}
Then $e^{C_p}$ is convergent and  the corresponding $L_\infty[1]$-morphism, given by 
equation \eqref{eq:pushforwardLinftymorph},  {has the following components: $ \phimapshift_{1}=\id_{\Aspace[1]}$ and, for $k\ge 1$,}
\begin{equation}\label{eq:PhikWithC}
\phimapshift_{k+1} = \sum_{n=1}^k \sum_{\substack{k_1+\dots+k_n=k \\ k_i \geq 1}} \frac{c_{k_1}\dots c_{k_n}}{n!}\bS^k
~.
\end{equation}
\end{lemma}
\begin{proof}
Since $p|_{\Aspace[1]}=0$, the convergence of  $e^{C_p}$  is guaranteed by Remark \ref{rem:codermor}.
The iterated power of $p$ reads as follows:
\begin{displaymath}
	p^{\cs n} = \sum_{k_1,\dots,k_n =1}^\infty c_{k_1}\cdots c_{k_n} \bS^{k_1+\dots+k_n}~,
\end{displaymath}
{in particular, any restriction to $\Aspace[1]^{\odot(k+1)}$ {vanishes for $k<n$, and} with $k\geq n \geq 1$ reads}
\begin{displaymath}
p^{\cs n} \Big\vert_{\Aspace[1]^{\odot(k+1)}} = \sum_{k_1+\dots+k_n=k}c_{k_1}\dots c_{k_n}~ \bS^k~.
\end{displaymath}
\end{proof}

We would like to choose the coefficients $c_i$ so that the resulting  $L_{\infty}[1]$ morphism $\Phi$  
has components ($k\ge 0$)
\begin{equation}
	\begin{split}
	\phimapshift_{k+1} 
	=&~ \left(\frac{2^k}{k!} B_k \right)\cdot \bS^k
	~.	
	\end{split}
	\label{eq:PhiAnsatz}
\end{equation}
One reason for doing so is the proof of our later Theorem \ref{thm:comm}, which  -- relying on the standard recursion formula for the Bernoulli numbers \eqref{eq:weisstein} --  shows that with this choice of components the natural diagram \eqref{eq:pentagonDiagram} commutes (use equation \eqref{eq:nicesum} to see this).
Another reason is that these components can be written as suitable contractions times a Bernoulli number (see Theorem \ref{thm:iso}), thus extending in  a manifest way the Rogers' embedding for $n=2$ \cite[Theorem 7.1]{Rogers2013}, in which the coefficient of the second component is the Bernoulli number $B_1=-\frac{1}{2}$.

Hence, according  {to Lemma \ref{lem:pansatazPhi}}, we would like to choose the 
coefficients $c_k$ so that they obey the following recurrence formula:
\begin{equation}\label{eq:recurrenceformulawithAnsa}
	\frac{2^k}{k!} B_k = \sum_{n=1}^k \sum_{\substack{\mkern20mu	 k_1+\dots+k_n=k \\ k_i \geq 1}}	 \frac{c_{k_1}\dots c_{k_n}}{n!}~.
\end{equation}

{A priori, the existence of coefficients $c_k$ with the above property is not clear at all. However, the following lemma shows explicitly their existence:} 
\begin{lemma}\label{lem:elezovic}
Equation \eqref{eq:recurrenceformulawithAnsa} is satisfied by the coefficients {$(k\ge 1)$}
 \begin{equation}\label{eq:ExplicitCk}
 	c_k :=
	(-1)^{k+1} \frac{B_k}{k\cdot k!} 2^k~.
 \end{equation}
\end{lemma}
\begin{proof}
The Bernoulli numbers satisfy the Elezovic summation formula (\cite[Cor. 2.10]{Elezovic})
{
 \begin{align*}
		\frac{B_k}{k!}
		&=~
		\sum_{n=1}^k (-1)^{n+k}
		\sum_{\substack{k_1+\dots+k_n=k \\ k_i \geq 1}} 
		\frac{1}{n!}\Big(\frac{B_{k_1}}{k_1\cdot k_1!}\Big)
		\dots
		\Big(\frac{B_{k_n}}{k_n\cdot k_n!}\Big)
		\\
		&=~
		\sum_{n=1}^k
		\sum_{\substack{k_1+\dots+k_n=k \\ k_i \geq 1}} 
		\frac{1}{n!}\Big((-1)^{k_1+1}\frac{B_{k_1}}{k_1\cdot k_1!}\Big)
		\dots
		\Big((-1)^{k_n+1}\frac{B_{k_n}}{k_n\cdot k_n!}\Big)~.	
	\end{align*}
}
\end{proof}

\begingroup
\setlength{\tabcolsep}{10pt} 
\renewcommand{\arraystretch}{2.2} 
\begin{table}[h!]
	\label{table:coeffi}
	\centering
	\begin{tabular}{|c|c|c|c|c|c|c|c|c|c|c|c|}
		\hline
		$\mathbf k$ & $0$             & $1$            & $2$            & $3$ & $4$             & $5$ & $6$               & $7$ & $8$               & $9$ & $10$                                     
		\\ \hline
		$\mathbf B_k$ & $1$ & $-\frac{1}{2}$ & $\frac{1}{6}$  & $0$ & $-\frac{1}{30}$ & $0$ & $\frac{1}{42}$    & $0$ & $-\frac{1}{30}$   & $0$ & $\frac{5}{66}$                           
		\\ \hline
		$\mathbf c_k$ &          & $-1$           & $-\frac{1}{6}$ & $0$ & $\frac{1}{180}$ & $0$ & $-\frac{1}{2835}$ & $0$ & $\frac{1}{37800}$ & $0$ & \multicolumn{1}{c|}{$-\frac{1}{467775}$} \\ \hline
	\end{tabular}
	\smallskip
	\caption{Sampling of the numerical coefficients. $B_k$ is the $k$-th Bernoulli number, and $c_k$
	is defined in equation \eqref{eq:ExplicitCk}.}
\end{table}
\endgroup

\subsection{Expressing  $\pi'$ in terms of $\pi$.}\label{Sec:piPrime}

The purpose of this subsection is to prove Proposition \ref{prop:MarcoClaimHigherN},
which provides an expression for $\bpi_n'$ which we will need in the sequel.

By construction {(see equation \eqref{eq:pushbrackets})}, 
the component $\bpi'_n=\bpi'|_{\Aspace[1]^{\odot n}}$ 
is a finite sum:
\begin{equation}\label{eq:pin'exp}
	\bpi'_n = \bpi_n + 
	\sum_{m=1}^{n{-}1} \frac{1}{m!}~
	\left.
	\left[\underbrace{p,\dots[p}_{m\text{ times}},\bpi]\right] 
	\right\vert_{\Aspace[1]^{\odot n}}
	~.
\end{equation}
Assuming the Ansatz introduced in Lemma \ref{lem:pansatazPhi}, $$p= \sum_{k=1}^\infty c_k ~\bS^{ k},$$
where at first we take the $c_i$ to be arbitrary real numbers, one has 
\begin{align} 
	\Big[\underbrace{p,\dots[p}_{m\text{ times}}&,\bpi]\Big]
	\Big\vert_{\Aspace[1]^{\odot n}} 
	=
	\notag
	\\
	=&~
	\sum_{\substack{k_1+\dots+k_m+k_{m+1}=n \\ k_i \geq 1}}
	c_{k_1}\dots c_{k_m}
	\Big[\underbrace{\bS^{ k_1},\dots[\bS^{ k_m}}_{m\text{ times}},
	\bpi_{k_{m+1}}]\Big]
	\notag
	\\
	=&~
	\sum_{q=1}^{n-m}
	\sum_{\substack{k_1+\dots+k_m = n-q \\ k_i \geq 1}}
	c_{k_1}\dots c_{k_m}
	\Big[\underbrace{\bS^{ k_1},\dots[\bS^{ k_m}}_{m\text{ times}},
	\bpi_{q}]\Big]
	~.
	\label{eq:summandsIteratedCommu}
\end{align}

We now make some considerations about the vanishing of $\bpi'_n$ 
when applied to {homogeneous} elements of $\Aspace[1]$.
\begin{remark}\label{rem:degrees}
	Recall that the elements of $\Aspace[1]$ of maximal degree are those of degree $-1$, and are those lying in $\Ham_{\infty}^{n{-}1}(M,\omega)[1]$. 
	According to equation \eqref{eq:summandsIteratedCommu}, $\bpi'_n$ is a linear combination of operators  $\bS^i \cs \bpi_q \cs \bS^j$ of arity $n=i+j+q$, with $i,j\geq 0$ and $1\leq q \leq n$.
	Observe that operator $\bS^i \cs \bpi_q \cs \bS^j$ applied to  {homogeneous} elements of $\Aspace[1]$ might be non-vanishing only when 
	\begin{displaymath}
		\begin{cases}
			q \leq 2 &, \text{\emph{all but possibly one} elements are in degree $-1$;}\\
			q\ge 3 &,\text{\emph{all} elements are in degree $-1$}.
		\end{cases}
	\end{displaymath}
	The conclusion we draw is that the $L_{\infty}[1]$-algebra structure $\bpi'$ on $\Aspace[1]$, defined as in equation \eqref{eq:pin'exp}, has the following property:  the evaluation of multibrackets with arity $k\ge 2$ on  homogeneous elements
   might be non-vanishing only when \emph{all but possibly one} elements are of top  degree (i.e. degree $-1$). Notice that the $L_{\infty}[1]$-algebra associated to the $\omega$-twisted Vinogradov algebroid has the same property, as follows from Remark \ref{Remark:semplifica-conti}.
\end{remark}
\begin{remark}\label{rem:bpi'small}
Given equations \eqref{eq:pin'exp} and \eqref{eq:summandsIteratedCommu}, it is possible to directly compute $\bpi'_n$ for low values of $n$:
\begin{align*}
	\bpi_1' =&~ \bpi_1
	\\
	\bpi_2' =&~ -c_1 \bpi_2 +c_{1}[\bS,\bpi_1]
	\\
	\bpi_3' =&~ \bpi_3
	+ \left( \frac{c_1 c_1}{2} + c_1 \right) [\bS,[\bS,\bpi_1]]
	+c_2 [\bS^{ 2},\bpi_1]
	~,
\end{align*}
 where to compute $\bpi_3'$ we used Lemma \ref{Prop:TernaryCommutator}.
\end{remark}

The general expression for $n$ greater than $3$ is given as follows:
\begin{lemma}\label{lem:pin'first}
 Given $n\ge 3$, one has:
\begin{align*}
	\bpi_n'
	=&~\phantom{+}
	\bpi_n 
	~+
	\\
	&+
	\left\lbrace
		\sum_{m=1}^{n{-}1}
		\sum_{q=1}^{n-m}
		\left(
			\sum_{k_1+\dots+ k_m = n-q}
			\frac{c_{k_1}\cdots c_{k_m}}{m!}
		\right)
		\left(\frac{n!}{2^{n-q}q!}\right)
	\right\rbrace
	\bpi_n
	~+
	\\
	&-
	\left\lbrace
		\left(
			c_{n-2}+c_{n{-}1}+\sum_{k_1+k_2=n{-}1}\frac{c_{k_1}c_{k_2}}{2}
		\right)
		\left(\frac{n!}{2^{n{-}1}}\right)
	\right\rbrace
	\bpi_n
	~+
	\\
	&+
	\sum_{k_1+k_2=n{-}1}\frac{c_{k_1}c_{k_2}}{2} ~ [\bS^{ k_1},[\bS^{ k_2},\bpi_1]]
	~+
	\\
	&+
	c_{n-2}[\bS^{(n-2)},[\bS,\bpi_1]]
	~+
	\\
	&+
	c_{n{-}1}[\bS^{(n{-}1)},\bpi_1]
	~.	
\end{align*}
\end{lemma}
\begin{proof}
Observe that the summands in equation \eqref{eq:summandsIteratedCommu} can be expressed in terms of higher $\bpi$ multibrackets employing Proposition \ref{prop:marcoestate123}.
Namely, the terms in the summation expressing $\bpi_n'$ are subsumed by the following three cases.
\\
When $m\geq 3$
\begin{displaymath}
	\Big[\underbrace{p,\dots[p}_{m\text{ times}},\bpi]\Big] \Big\vert_{\Aspace[1]^{\odot n}} 
	=
	\sum_{q=1}^{n-m}\left(\sum_{k_1+\dots+k_m=n-q} c_{k_1}\dots c_{k_m}\right)
	\left(\frac{n!}{2^{n-q}q!}\right)
	\bpi_n	
	~.
\end{displaymath}
When $m=2$
\begin{align*}
	\Big[p,[p,\bpi]\Big] \Big\vert_{\Aspace[1]^{\odot n}} 
	=&~\phantom{+}
	\sum_{q=2}^{n-2}
	\left(\sum_{k_1+k_2=n-q}c_{k_1}c_{k_2}\right)
	\left(\frac{n!}{2^{n-q}q!}\right) \bpi_n
	~+
	\\
	&~+
	\sum_{k_1+k_2=n{-}1}c_{k_1}c_{k_2} ~ [\bS^{ k_1},[\bS^{ k_2},\bpi_1]]
	\\[1em]
	=&~\phantom{+}
	\sum_{q=1}^{n-2}
	\left(\sum_{k_1+k_2=n-q}c_{k_1}c_{k_2}\right)
	\left(\frac{n!}{2^{n-q}q!}\right) \bpi_n
	~+
	\\
	&~-
	\left(\sum_{k_1+k_2=n{-}1}c_{k_1}c_{k_2}\right)
	\left(\frac{n!}{2^{n{-}1}}\right) \bpi_n ~+
	\\
	&~+
	\sum_{k_1+k_2=n{-}1}c_{k_1}c_{k_2} ~ [\bS^{ k_1},[\bS^{ k_2},\bpi_1]]
	~.
\end{align*}
When $m=1$, one has that
\begin{displaymath}
	[p,\bpi] \Big\vert_{\Aspace[1]^{\odot n}} 
	=
	\sum_{q=1}^{n{-}1} c_{n-q} [\bS^{n-q},\bpi_q]~,
\end{displaymath}
which, in the case that $n\geq 3$, reads as follows
\begin{align*}
	[p,\bpi] \Big\vert_{\Aspace[1]^{\odot n}} 
	=&~\phantom{+}
	\sum_{q=3}^{n{-}1}
	\left(c_{n-q}\right)
	\left(\frac{n!}{2^{n-q}q!}\right) \bpi_n
	~+
	\\
	&~+
	c_{n-2}[\bS^{(n-2)},\bpi_2]
	+ c_{n{-}1}[\bS^{(n{-}1)},\bpi_1]
	\\[1em]
	=&~\phantom{+}
	\sum_{q=1}^{n{-}1}
	\left(c_{n-q}\right)
	\left(\frac{n!}{2^{n-q}q!}\right) \bpi_n
	~+
	\\
	&~
	- c_{n-2}\left(\frac{n!}{2^{n{-}1}}\right)\bpi_n 
	- c_{n{-}1}\left(\frac{n!}{2^{n{-}1}}\right)\bpi_n
	\\
	&~+
	c_{n-2}[\bS^{(n-2)},[\bS,\bpi_1]]
	+ c_{n{-}1}[\bS^{(n{-}1)},\bpi_1],
\end{align*}
rewriting the penultimate terms according to Corollary \ref{lemma:pairpi2-as-pairpairPi1}.
Adding all the terms finishes the proof of the lemma.
\end{proof}

Assume now that the coefficients $c_k$ are given by equation \eqref{eq:ExplicitCk}.
\begin{remark}[$\pi'_n$ for $n\le 3$]\label{rem:pin'small}
Plugging in the values of $c_k$ prescribed by equation \eqref{eq:ExplicitCk} in the expressions given in Remark \ref{rem:bpi'small}, we conclude  that
\begin{align*}
	\bpi_1' =&~\bpi_1
	\\
	\bpi_2' =&~\bpi_2 -[\bS,\bpi_1]
	\\
	\bpi_3' =&~
	   \bpi_3
	- \frac{1}{2} {[\bS,[\bS,\bpi_1]]}
	-\frac{1}{6} [\bS^{ 2},\bpi_1] ~.
	\end{align*}
Incidentally, this shows that Proposition \ref{prop:MarcoClaimHigherN} below does not hold for $n=3$.
\end{remark}
 
\begin{proposition}\label{prop:MarcoClaimHigherN}
When $n\geq 4$, {and with $c_k$ defined as in \eqref{eq:ExplicitCk},} we have
\begin{align}
	\bpi_n'
	=&
	-
	\left\lbrace
		\left(
			c_{n{-}1}+\sum_{\substack{k_1+k_2=n{-}1\\k_i\geq 2}}\frac{c_{k_1}c_{k_2}}{2}
		\right)
		\left(\frac{n!}{2^{n{-}1}}\right)
	\right\rbrace	
	\bpi_n
	~+
	\notag
	\\
	&+
	\sum_{\substack{k_1+k_2=n{-}1\\k_i\geq 2}}\frac{c_{k_1}c_{k_2}}{2}
	[\bS^{ k_1},[\bS^{ k_2},\bpi_1]]
	\label{eq:MarcoClaim}
	\\
	&+
	c_{n{-}1}[\bS^{(n{-}1)},\bpi_1]	\notag
\end{align}
When $n\geq 4$ is even, it follows that ${\bpi_n'=0}$.
\end{proposition}
\begin{proof}
Consider the expression for $\pi_n'$ obtained in Lemma \ref{lem:pin'first}, taking
as $c_k$ the coefficients given  in equation \eqref{eq:ExplicitCk}. This leads to several simplifications.

{The first term in curly braces reads} 
	\begin{align*}
		\lbrace\cdots \rbrace
		&=~
		\sum_{m=1}^{n{-}1}
		\sum_{q=1}^{n{-}m}
		\left(\frac{n!}{2^{n-q}q!}\right)	
		\left(
			\sum_{k_1+\dots+ k_m = n-q}
			\frac{c_{k_1}\cdots c_{k_m}}{m!}
		\right)
		\\
		&=~
		\sum_{q=1}^{n{-}1}
		\left(\frac{n!}{2^{n-q}q!}\right)			
		\left(
			\sum_{m=1}^{n-q}
			\sum_{k_1+\dots+ k_m = n-q}
			\frac{c_{k_1}\cdots c_{k_m}}{m!}
		\right)
		\\
		&=
		\sum_{q=1}^{n{-}1}
		\left(\frac{n!}{2^{n-q}q!}\right)			
		\left(\frac{2^{n-q}}{(n-q)!}\right)
		B_{n-q}
		\\
		&=
		\sum_{q=1}^{n{-}1}
		\binom{n}{q}
		B_{n-q}
		\\
		&=
		\sum_{s=1}^{n{-}1}
		\binom{n}{s}
		B_{s}
		\\
		&=
		- B_0 = -1.							
	\end{align*}
Here the third equality holds by the Elezovic summation formula
\eqref{eq:recurrenceformulawithAnsa}, and in the penultimate line, one can recognize the standard recursion formula for Bernoulli numbers {(see equation \eqref{eq:weisstein})}.

We now look at the second term in curly braces.
Noting that $c_1 c_{n-2} =- c_{n-2}$, when $n\geq 4$ one has that\footnote{This {equality} is false in the case that $n=3$ and ill-defined when $n\geq 2$.}
\begin{displaymath}
	\left\lbrace
		\left(
			c_{n-2}+c_{n{-}1}+\sum_{\substack{k_1+k_2=n{-}1\\k_i\geq 1}}\frac{c_{k_1}c_{k_2}}{2}
		\right)
		\left(\frac{n!}{2^{n{-}1}}\right)
	\right\rbrace
	=
	\left\lbrace
		\left(
			c_{n{-}1}+\sum_{\substack{k_1+k_2=n{-}1\\k_i\geq 2}}\frac{c_{k_1}c_{k_2}}{2}
		\right)
		\left(\frac{n!}{2^{n{-}1}}\right)
	\right\rbrace~.	
\end{displaymath}
Finally we turn to the third and fourth term in the expression for $\pi_n'$ obtained in Lemma \ref{lem:pin'first}. Employing again the previous observation, one has that
\begin{align*}
	\sum_{\substack{k_1+k_2=n{-}1\\k_i\geq 1}}\frac{c_{k_1}c_{k_2}}{2} &~ [\bS^{ k_1},[\bS^{ k_2},\bpi_1]]+
	c_{n-2}[\bS^{(n-2)},[\bS,\bpi_1]]
	\\
	&=~
	\sum_{\substack{k_1+k_2=n{-}1\\k_i\geq 2}}\frac{c_{k_1}c_{k_2}}{2} ~ [\bS^{ k_1},[\bS^{ k_2},\bpi_1]]	
\end{align*}
In conclusion, the expression for $\pi_n'$ obtained in Lemma \ref{lem:pin'first}  reduces to equation \eqref{eq:MarcoClaim}.

Finally, assume now $n\geq 4$ to be even.  
	Since 
$c_j=0$ for any odd $j\ge 3$, 	
	one has in particular that $c_{n{-}1}=0$ and
	\begin{displaymath}
		\sum_{\substack{k_1+k_2=n{-}1\\k_i\geq 2}}\frac{c_{k_1}c_{k_2}}{2}\Big(\dots \Big) =0	
	\end{displaymath}
since each summand is a product of a $c_{i}$ with $i$ even and a $c_j$ with $j\ge 3$ odd.
\end{proof}

\subsection{Comparing $\mu_n$ and $\pi_n'$}\label{Sec:muvspiprime}

 In this subsection we will show that $\bpi_n'=\bmu_n$ for all $n\ge 1$, allowing us in Proposition \ref{prop:Phi} to display the $L_{\infty}[1]$-morphism sought for.

 Proposition \ref{prop:MarcoClaimHigherN} and the definition of $\bmu_n$ immediately imply:
\begin{proposition}\label{prop:even}
For all {\it even} $n\ge 4$, we have $\bpi_n'=0=\bmu_n$.
\end{proposition}

{ 
\begin{remark}[$\bpi_n'=\bmu_n$ for $n\le 3$]\label{rem:bpi3}
Recall the computation of $\bpi_n'$ for small values of $n$ given in Remark \ref{rem:pin'small}.
	One can directly check that $\bpi_n'=\bmu_n$ comparing those results with Lemma \ref{Prop:mu2} for $n=2$  and Lemma \ref{Prop:mu3} for $n=3$.
\end{remark}
 }

Due to the above, in the following we will focus on the case that $n\ge 5$ is an odd integer.
We start expressing $\bmu_n$  in terms of $\bpi_n$ and certain commutators with $\bpi_1$, just as we did for $\bpi'_n$ in Proposition \ref{prop:MarcoClaimHigherN}.

\begin{proposition}\label{lem:express-mun}
{For all $n\ge 3$:}
	\begin{displaymath}
		\bmu_n = \big(2n B_{n{-}1}\big)\bpi_n
		- \left[\frac{2^{n-1}}{(n-1)!} B_{n-1}\right]
		\big(
			2 \bS^{n{-}1}\cs\bpi_1 - 3 \bS^{n-2}\cs\bpi_1\cs \bS + \bS^{n-2}\cs\bpi_1\cs \bS^2
		\big)
	\end{displaymath}
\end{proposition}
\begin{proof}  
	This follows from some direct computations carried out in the appendix.	
{More precisely, Proposition \ref{lem:mun} expresses 
	$\bmu_{{n}}$ as a multiple of $\bS^{n{-}3} \cs \bmu_3$, so we can plug into this the statement of Lemma \ref{Prop:mu3} (which expresses $\bmu_3$ in terms of commutators of $\bS, \bpi_3, \bpi_1$). Then employ Proposition \ref{cor:higherpi} to obtain a multiple of $\bpi_n$, and expand the commutators.}
\end{proof}

Now assume $n\geq {5}$ is odd. {We define the integer $N:=\frac{n{-}1}{2}$.}
Motivated by Proposition \ref{prop:MarcoClaimHigherN}, we make the following assumption:

\begin{assumption}\label{assumption}
There exist real numbers $a,d,b_2,b_4,\dots, b_{2\floor{N/2}}$ such that
\footnote{Using the floor function $\floor{\;\;}$, the sum in equation \eqref{eq:ugly-term} effectively runs over even integers $k$ between $2$ and 
$${2\floor{N/2}}=\begin{cases}N \text{ if $N$ is even}\\ 
N-1 \text{ if $N$ is odd}.
\end{cases}
$$
}
\begin{equation}\label{eq:ugly-term}
		\big(
			2 \bS^{n{-}1}\cs\bpi_1 - 3 \bS^{n-2}\cs\bpi_1 \cs\bS + \bS^{n-2}\cs\bpi_1\cs \bS^2
		\big)
		=
		a \, \bpi_n +
		d\, [\bS^{n{-}1},\bpi_1] +
		\sum_{\substack{k \text{ even}\\ 2\le k \le N}} b_k \,[ \bS^k,[\bS^{n{-}1-k},\bpi_1]].
\end{equation}
 \end{assumption}

\begin{lemma}\label{lem:techMarco1}
Let $n\geq {5}$ be odd.	The following identities imply $\bpi_n' = \bmu_n$:
	\begin{align*}
		2 n B_{n{-}1} - 
		\left[\frac{2^{n-1}}{(n-1)!} B_{n-1}\right]
		&a =
		-\left(
			c_{n-1} + \sum_{\substack{k_1+k_2=n{-}1 \\ k_i\geq2}}
			\frac{c_{k_1}c_{k_2}}{2}
		\right) \frac{n!}{2^{n{-}1}}
		\\
		-
		\left[\frac{2^{n-1}}{(n-1)!} B_{n-1}\right]
		& d =~ c_{n{-}1}
		\\
		-
		\left[\frac{2^{n-1}}{(n-1)!} B_{n-1}\right]
		 &b_k =
		\begin{cases}~ \;\;\,c_k \, c_{n{-}1-k} \qquad \text{for } k \text{ even},\; 2\le k < N\\
~ \frac{1}{2}c_k \, c_{n{-}1-k} \qquad \text{for } k=N \text{ even}		\end{cases}
	\end{align*}
	\end{lemma}
\begin{proof}
	Proposition \ref{prop:MarcoClaimHigherN} gives an expression for $\bpi'_n$ in terms of $\bpi_n$ and certain commutators.
	Proposition \ref{lem:express-mun} together with equation \eqref{eq:ugly-term} does the same for $\bmu_n$. The three equations in the statement of this lemma are obtained equating the coefficients of $\bpi_n$, of $[\bS^{n{-}1},\bpi_1]$, and of $[\bS^k,[\bS^{n{-}1-k},\bpi_1]]$ respectively.
\end{proof}

{Rewriting the right hand sides in Lemma \ref{lem:techMarco1} in terms of the Bernoulli numbers, we get:}
\begin{lemma}\label{lem:techMarco2}
Let $n\geq {5}$ be odd. The three identities in Lemma \ref{lem:techMarco1} are equivalent to the following three identities:
	\begin{align*}
 \frac{2^{n{-}1}}{n!}	a =&~
		\left(2-  \frac{1}{n{-}1}\right)
		+\frac{1}{2}\frac{1}{B_{n{-}1}}\left(
			 \sum_{\substack{k_1+k_2=n{-}1 \\ k_i\geq2}}
				\binom{n{-}1}{k_1}
			\frac{1}{k_1k_2}B_{k_1} B_{k_2}
		\right)			
		\\
		d =&~ \frac{1}{n{-}1}
		\\
		b_k =&\begin{cases}~ \;\,\,-\frac{1}{B_{n{-}1}}\binom{n{-}1}{k}
			\frac{1}{k (n{-}1-k)}B_k B_{n{-}1-k}
			\qquad \text{for }  k \text{ even},\; 2\le k < N\\
			~ -\frac{1}{2}\frac{1}{B_{n{-}1}}\binom{n{-}1}{k}
			\frac{1}{k (n{-}1-k)}B_k B_{n{-}1-k}
			\qquad \text{for }  k=N \text{ even}\\
			\end{cases}
	\end{align*}
\end{lemma}
\begin{proof}
{Recall that by definition $c_k= (-1)^{k+1}\cfrac{B_k}{k\cdot k!} 2^k$, see equation \eqref{eq:ExplicitCk}.} This implies immediately the expression for $d$. 
It also implies that for any positive integers $k_1,k_2$ with $k_1+k_2=n{-}1$ we have
\begin{equation}\label{eq:ck1ck2}
\frac{(n{-}1)!}{2^{n{-}1}}c_{k_1}c_{k_2}=
\binom{n{-}1}{k_1} \frac{B_{k_1}}{k_1}\frac{B_{k_2}}{k_2}.  
\end{equation}
From this the expression for the $b_k$ readily follows.
Using $-c_{n{-}1}\frac{n!}{2^{n{-}1}}=\frac{n}{n{-}1}B_{n{-}1}$ and
equation \eqref{eq:ck1ck2} we can rewrite the r.h.s. of the first equation in Lemma \ref{lem:techMarco1}, and dividing by $-nB_{n{-}1}$ we obtain exactly the first equation of the present lemma.
\end{proof}

  The proof of the following proposition is given in  Appendix \ref{sec:system}.
\begin{proposition}\label{prop:sys}
 Let $n\geq {5}$ be odd. Let $d$ and $b_k$ as in Lemma \ref{lem:techMarco2}. Then there exist real numbers $\{a_{k_3k_2}\}$ and $\{a_{k_4k_3}\}$  such
that  
\begin{equation}\label{eq:231}
2\bS^{n-1}\cs \bpi_1 - 3\bS^{n-2}\cs\bpi_1 \cs\bS + \bS^{n-3}\cs\bpi_1\cs \bS^2  
\end{equation} 
equals
\begin{equation}\label{eq:polydevelop}
	\begin{split}
	d[\bS^{n{-}1} ,\bpi_1]
 	&+
	\sum_{\substack{k \text{ even}\\ 2\le k \le N}}
	b_k[\bS^k,[\bS^{n{-}1-k},\bpi_1]]+
	\\
	&+
	\sum_{\substack{k_3\ge k_2\ge 1\\ k_3+k_2 =n{-}2}}
	a_{k_3k_2}  [\bS^{k_3},[\bS^{k_2},[\bS ,\bpi_1]]]
	\\
	&+
	\sum_{\substack{k_4\ge   k_3\ge 1 \\ k_4+k_3=n{-}3}}
	a_{k_4k_3}  [\bS^{k_4},[\bS^{k_3},[\bS,[\bS,\bpi_1]]]]
	~.
	\end{split}
\end{equation}
\end{proposition}

Now let $d,b_k,a_J$ as in Proposition \ref{prop:sys}, where
$2\le k \le N$ is even and 		
	 $J$ denotes pairs of indices as in that proposition.
		Thanks to 
		Proposition \ref{prop:marcoestate123} (with $q=1$) we know that the last two sums in  \eqref{eq:polydevelop} combine to 
	\begin{displaymath}
		a \bpi_n \qquad \text{where} \quad a:= \frac{n!}{2^{n{-}1}}\left(\sum a_J \right) 
		~.
	\end{displaymath}	
This implies that Assumption \ref{assumption} is satisfied.

Further, the coefficient $a$ can be written as 
	\begin{equation}\label{eq:a}
		 a = \frac{n!}{2^{n{-}1}}\left(2 -d - \sum_{\substack{k \text{ even}\\ 2\le k \le N}}
 b_k\right) ~.
	\end{equation}
Indeed, even though we do not know the individual coefficients $a_J$,
equating the coefficients of $\bS^{n{-}1}\cs\bpi_1$ in Proposition \ref{prop:sys} one sees that
$d + \sum b_k + \sum a_J = 2$.

Now, equation \eqref{eq:a} is equivalent  to the first equation in Lemma \ref{lem:techMarco2}. The remaining equations in Lemma \ref{lem:techMarco2} are automatically satisfied due to our choice of $d$ and $b_k$. Using Lemma \ref{lem:techMarco1}, we thus proved the following theorem:

\begin{theorem}\label{thm:bpi5}
For all odd {$n\ge 5$}, we have $\bpi_n'=\bmu_n$.
\end{theorem}

Combining Proposition \ref{prop:even}, Remark \ref{rem:bpi3} and Theorem \ref{thm:bpi5} we obtain
\begin{corollary}\label{cor:bpi1}
 For all integers  $n\ge 1$, we have $\bpi_n'=\bmu_n$.
\end{corollary}

Together with this corollary, the discussion in  \Ss \ref{sec:Preparation} yields
 that the two $L_\infty[1]$-algebra structures $\bpi$ and $\bmu$ are isomorphic. More precisely:
\begin{proposition}\label{prop:Phi}
There exists an $L_\infty[1]$-morphism ${\phimapshift}\colon (\Aspace[1],\bpi)\to(\Aspace[1],\bmu)$ whose first component is given by $Id_{\Aspace[1]}$ and higher component given as in equation \eqref{eq:PhiAnsatz}.
More precisely, 
\begin{equation}\label{eq:recallPhik} 
{\phimapshift}_k = { \left(\frac{2^{k-1}}{(k-1)!}B_{k-1}\right)} ~\bS^{k-1}~,  
\end{equation}
for any $k\geq 1$.
\end{proposition}
 
\begin{remark}\label{rem:PhikIndependentOmega}
The expressions given by equation \eqref{eq:recallPhik}, being proportional to powers of $\bS$, are well-defined on the graded vector space $\Vspace[1]$ and there \emph{do not} depend on the choice of  multisymplectic form $\omega$ on $M$.	On the other hand, the $L_\infty[1]$-morphism $\phimapshift$ does depend  on the choice of $\omega$, since it is obtained   restricting the above expressions to $\Aspace[1]$, which in degree zero depends on $\omega$ due to required specification of Hamiltonian forms and vector fields.
\end{remark}

\subsection{ The $L_{\infty}$-embedding of Rogers' $L_\infty$-algebra into Vinogradov'}\label{sec:HigherRogEmbeddings}

We can finally state the main result of this section, which extends Rogers' \cite[Theorem 7.1]{Rogers2013} to {the setting of $n$-plectic forms for} arbitrary values of $n$.

\begin{theorem}\label{thm:iso}
Given an $n$-plectic manifold $(M,\omega)$, consider the corresponding Rogers and Vinogradov $L_\infty$-algebras, denoted by $L_\infty(M,\omega)$ and $L_\infty(E^n,\omega)$ respectively.
	There exists an $L_{\infty}$-embedding $${\psimap}: L_{\infty}(M,\omega)\hookrightarrow L_\infty(E^n,\omega),$$
with components 
\begin{align*}
{\psimap}_1(f\oplus \alpha) 
			=& 
			f\oplus \pair{v_\alpha}{\alpha}\\
		{\psimap}_k(f_1\oplus \alpha_1,\dots,f_k\oplus \alpha_k)
		=&
		B_{k-1}\sum_{j=1}^k (-1)^{k-j}\iota_{v_{\alpha_{1}}}\dots\widehat{\iota_{v_{\alpha_{j}}}}\dots \iota_{v_{\alpha_k}} (f_{j}\oplus \alpha_{j})~ \quad\quad  \text{for $k\ge 2$ } 
\end{align*}	
 for any 
 $f_i \in \bigoplus_{k=0}^{n-2}\Omega^k(M)$ and $\alpha_i \in \Omega^{n{-}1}_{Ham}(M,\omega)$.
	\end{theorem}
 
\begin{remark}
 The expression for ${\psimap}_1$ and the one for  ${\psimap}_2$, namely
$$	{\psimap}_2(f_1\oplus \alpha_1,f_2\oplus \alpha_2)
			=
			-\dfrac{1}{2}\left(
				\iota_{v_{\alpha_1}}(f_2\oplus \alpha_2) - \iota_{v_{\alpha_2}}(f_1\oplus \alpha_1)
			\right),
$$
are due to Rogers \cite[Theorem 7.1]{Rogers2013} and
are 
sufficient to elucidate the $2$-plectic case.
		\end{remark}


The proof of Theorem \ref{thm:iso} is a  direct reformulation of Proposition \ref{prop:Phi} in the skew-symmetric multibrackets framework.

\begin{proof}
Applying the inverse d\'ecalage of the $L_{\infty}[1]$-isomorphism $\phimapshift:(\Aspace[1],\bpi)\cong(\Aspace[1],\bmu)$ exhibited in Proposition \ref{prop:Phi}, we obtain an $L_\infty$-isomorphism 
${\phimap}\colon (\Aspace,\pi)\to(\Aspace,\mu)$.  
Its components are given by \footnote{
The first higher components of $\phimap$ read as follows: 	${\phimap}_2 =~ -\pairing_-$,
	${\phimap}_3 =~ \frac{1}{3}\pairing_-\ca \pairing_-$, and  
	${\phimap}_4 =~ 0$.
}  
 \begin{align*}
	{\phimap}_1 &=~ Id_{\Aspace}\\ 
		\phimap_{k} &=~ \left(\frac{2^{k-1}}{(k-1)!}B_{k-1}\right)\pairing_-^{\ca (k-1)}  \qquad, \forall  k \geq 2 
\end{align*}
where $\ca$ is the skew-symmetric operators version of $\cs$ defined in equation \eqref{Eq:RNProducts}.
 This is a consequence of the fact that the operation $\ca$ (used here) and $\cs$ (used implicitly in equation \eqref{eq:recallPhik}) correspond under the d\'ecalage isomorphism, see Remark \ref{rem:NJiso}.
Further, the expressions\footnote{These expressions  no longer involve powers of $2$ and factorials.}  we obtain applying $\phimap_k$ to strings of elements are worked out in Corollary \ref{cor:EvaluatedExpression}.

 Consider the following commutative diagram in the category of $L_\infty$-algebras: 
	\begin{displaymath}
		\begin{tikzcd}[column sep= large]
			L_\infty(M,\omega) \ar[r,dashed,hook,"\psimap"] \ar[d,"\sim",sloped]&
			L_\infty(E^n,\omega)
			\\
			(\Aspace,\pi) \ar[r,"\sim","\phimap"']&
			(\Aspace, \mu) \ar[u,hook]
		\end{tikzcd}
		~
	\end{displaymath}
 where the left vertical arrow is given by \eqref{eq:LMomegaA}.

By the discussion at the beginning of \Ss \ref{sec:Preparation}, this diagram determines the $L_\infty$-embedding $\psimap$ we are after. The discussion above implies that the components of $\psimap$ are as given in the statement of Theorem \ref{thm:iso}.
\end{proof}

\section{Gauge transformations}\label{sec:gaugetrafos}

Given two $n$-plectic forms $\omega$ and $\widetilde{\omega}$ on the same manifold $M$, we say that they are \emph{gauge related} if there exists a $n$-form $B$ such that 
$$
	\widetilde{\omega} = \omega  {+}\dd B ~.
$$
 The differential form $B$ determines an isomorphism between the Vinogradov algebroid twisted by $\omega$ and the one twisted by $\widetilde{\omega}$. The corresponding map is usually dubbed $B$-transformation or \emph{gauge transformation} (see  \cite[\S 3]{SW} for the  Courant algebroid  case).

{The aim of this section is to show that the $L_{\infty}$-embedding constructed in 
Theorem \ref{thm:iso}
is compatible with gauge transformations, see Theorem \ref{thm:comm} below.}

 \subsection{Vinogradov algebroids and gauge transformations}
 \label{subsec:Vinogauge}
 {Let $\omega$ and $\widetilde{\omega} = \omega {+} \dd B$ be gauge-related closed $(n{+}1)$-forms on $M$.}
The two corresponding twisted Vinogradov algebroids $(E^n, \omega)$ and $(E^n,\widetilde{\omega})$, {see Definition \ref{def:Vinalgoid},} are isomorphic. Indeed, the vector bundle isomorphism 
\begin{displaymath}
	\begin{tikzcd}[column sep= small,row sep=0ex,
				/tikz/column 1/.append style={anchor=base east}]
		\tau_B \colon &
		E^n = TM \oplus \wedge^{n{-}1} T^\ast M \ar[r]& E^n ~, \\
	 &\pair{X}{\alpha} \ar[r, mapsto] & 
	 \pair{X}{\alpha {+} \iota_X B} 
	\end{tikzcd}		
	 ~,		
\end{displaymath}
{preserves the anchor $\rho$, the pairing  $\pairing_{+}$, and maps the bracket $[\cdot,\cdot]_{\omega}$ to $[\cdot,\cdot]_{\widetilde{\omega}}$.}

Hence this bundle isomorphism induces a strict
{\footnote{``Strict'' means that all  components of arity greater than one vanish.}}
 $L_\infty$-isomorphism at the level of the corresponding Vinogradov's $L_\infty$-algebras 
 (cfr. \cite[Prop. 8.5]{Zambon2012}),
given by 
\begin{displaymath}
	(\tau_B)_1 =  \id_{\Vspace} {+} \iota_{\rho(\cdot)} B	
	~.
\end{displaymath}

Notice that there is a natural diagram in the $L_\infty$ algebra category 
\begin{equation}\label{diag:3sides}
	\begin{tikzcd}
		L_{\infty}(M,\omega) \ar[r,"\psimap"] 
		&
		L_{\infty}(E^n,\omega) \ar[d,"\tau_B"]
		\\
		L_{\infty}(M,\widetilde{\omega}) \ar[r,"\widetilde{\psimap}"] 
		&
		L_{\infty}(E^n,\widetilde{\omega})
	\end{tikzcd}
\end{equation}
where the horizontal arrows are the $L_{\infty}$-embeddings we constructed in Theorem \ref{thm:iso}.

When considering two gauge-related multisymplectic manifolds, it is not possible  to define a   canonical $L_\infty$ morphism between the two corresponding
observables $L_\infty$ algebras.
In particular there is no canonical way to close this diagram on the left to give a commutative square (this is already apparent in the symplectic case, see Remark \ref{rem:nosquare}).
In what follows, we will
look for a suitable pullback $\mathfrak{g}$ in the category of $L_\infty$ algebras:

\begin{displaymath}
	\begin{tikzcd}
			   \mathfrak{g}\ar[r] \ar[d] & L_\infty(M,\omega) \ar[d,"\psimap"]
		\\
		 L_\infty(M,\tilde{\omega}) \ar[r,"\widetilde{\psimap}"] & L_\infty(E^n,\omega)\cong L_\infty(E^n,\tilde{\omega})		
	\end{tikzcd}
	~.
\end{displaymath}

 \subsection{Homotopy comoment maps and gauge transformations}

{An infinitesimal action preserving an  $n$-plectic form} is called \emph{Hamiltonian} if admits a   homotopy comoment map: 
\begin{definition}[{\cite[Def./Prop. 5.1]{Callies2016}}]
	Let $\rho\colon \mathfrak g\to \mathfrak X(M)$ be a Lie algebra morphism 
 which preserves\footnote{{That is, $\mathcal{L}_{\rho(\xi)}\omega=0$ for all $\xi \in \mathfrak{g}$.}}
 the {$n$-plectic} form $\omega \in \Omega^{n{+}1}(M)$.
	A \emph{homotopy comoment map}   pertaining to $\rho$ is 
	an $L_\infty$-morphism 
	$$
		(f)=
		\left\{f_k:{\wedge}^k\mathfrak{g} \to \Lspace^{1-k}\subseteq \Omega^{n-k}(M)
		\right\}_{k=1,...,n}
	$$
	from $\mathfrak g$ to $L_\infty(M,\omega)$ 
	satisfying $$d(f_1(\xi))=-\iota_{\rho(\xi)}\omega \qquad \forall \xi\in\mathfrak{g}.$$ 
\end{definition}
{Any gauge-related multisymplectic structure inherits a homotopy comoment map,}
 see   \cite[Beginning of \S7.2]{Fregier2015}.
\begin{lemma}\label{lem:momaps}
Let the infinitesimal action $\rho \colon \mathfrak{g}\to \mathfrak{X}(M)$ 
	{preserve the $n$-symplectic form $\omega$ and admit a homotopy comoment map $(f):\mathfrak{g}\to L_\infty(M,\omega)$. 
	Suppose that $B\in \Omega^n(M)$ is preserved by the action. Then $\widetilde{\omega}=\omega + dB$, which we assume to be $n$-plectic, is also preserved and}
  admits a homotopy comoment map $(\widetilde{f}):\mathfrak{g}\to L_\infty(M,\widetilde{\omega})$, with components 
 
	\begin{displaymath}
		\widetilde{f}_k = (f_k +\mathsf{b}_k)
		~: {\wedge}^k \mathfrak{g}\to \Lspace^{1-k}
		~.
	\end{displaymath}
	where $\mathsf{b}_k =  \varsigma(k+1)\iota_{\rho(\cdot)} B~:\wedge^k\mathfrak{g}\to \Omega^{n-k}(M)$. 
	Explicitly, $\mathsf{b}_k(\xi_1,\dots,\xi_k)=\varsigma(k+1) \iota_{\rho(\xi_k)}\dots
\iota_{\rho(\xi_1)}B$.
\end{lemma}

\subsection{Commutativity}\label{sec:pentagonCommutativity}
 
{Consider an infinitesimal action $\rho:\mathfrak{g}\to \mathfrak{X}(M)$ 
	preserving the $n$-plectic form $\omega$. 
	Suppose that $B\in \Omega^n(M)$ is also preserved, and   that $\widetilde{\omega}:=\omega + dB$ is non-degenerate.}
Further,  assume there exists a  homotopy comoment map $(f):\mathfrak{g}\to L_\infty(M,\omega)$ for $\omega$.

 {Then we obtain a homotopy comoment map $(\widetilde{f})$ for $\widetilde{\omega}$, by Lemma \ref{lem:momaps}.
 }
We will show that the diagram \eqref{diag:3sides} can be completed
 to a commutative pentagon; {this is the main statement of this section.}
 
 \begin{theorem}\label{thm:comm}
The following diagram of $L_{\infty}$-algebra morphisms strictly\footnote{The commutativity is strict, in the sense that it is $1$-commutative in the $\infty$-category of $L_\infty$-algebra, i.e. it is not "commutative up to homotopy".} commutes.
Here $\psimap$ is the morphism introduced in Theorem \ref{thm:iso}.

  \begin{equation}\label{eq:pentagonDiagram}
	\begin{tikzcd}[column sep = huge]
		&
		L_{\infty}(M,\omega) \ar[r,"\psimap"] 
		&
		L_{\infty}(E^n,\omega) \ar[dd,"\tau_B"]
		\\[-1em]
		\mathfrak{g}\ar[ru,"f"] \ar[dr,"\widetilde{f}"']
		\\[-1em]
		&
		L_{\infty}(M,\widetilde{\omega}) \ar[r,"\widetilde{\psimap}"] 
		&
		L_{\infty}(E^n,\widetilde{\omega})
	\end{tikzcd}
\end{equation}
\end{theorem}
{One way to interpret this commutativity is by saying that the twisting of the homotopy comoment moment map by $B$ is compatible with the twisting of the Vinogradov algebroid.}

\bigskip
We prepare the ground for the proof of Theorem \ref{thm:comm}. 
We will {make use of the strict $L_{\infty}$-isomorphism $L_{\infty}(M,\omega)\cong (\Aspace,\pi_k)$ of Equation   \eqref{eq:LMomegaA}, which is  
given by  the identity in negative degrees, and  $\alpha\mapsto \pair{v_{\alpha}}{\alpha}$ in degree zero, {and similarly we make use of $L_{\infty}(M,\widetilde{\omega})\cong (\widetilde{\Aspace},\widetilde{\pi}_k )$.}
The commutativity of diagram \eqref{eq:pentagonDiagram} is then equivalent\footnote{{This is clear from the diagram in the  proof of Theorem \ref{thm:iso}}.} to the commutativity of the following diagram:
	\begin{equation}\label{diag:HamiltonianpentagonDiagram}
		\begin{tikzcd}[column sep=huge]
			&
			\big(\Aspace,\{\pi_k\}\big) \ar[r,"\phimap"] 
			&
			\big(\Aspace,\{\mu_k\}\big)
			\ar[dd,"\tau_B"]\ar[dd,sloped,"\sim"']
			\\[-1em]
			\mathfrak{g}\ar[ru,"f"] \ar[dr,"\widetilde{f}"']
			\\[-1em]
			&
			\big(\widetilde{\Aspace},\{\widetilde{\pi}_k \}\big)
			 \ar[r,"\widetilde{\phimap}"] 
			&
			\big(\widetilde{\Aspace},\{\widetilde{\mu}_k \}\big)
		\end{tikzcd}
	\end{equation}
	
\noindent
Here we denote the composition of $f$ with the above $L_{\infty}$-isomorphism by the same letter $f$ used in diagram \eqref{eq:pentagonDiagram}, and $\phimap$ is the $L_{\infty}$-morphism constructed in the proof of Theorem \ref{thm:iso}
  (it is obtained as the d\'ecalage of the $L_{\infty}[1]$-morphism $\phimapshift$ given in Proposition \ref{prop:Phi}).

\begin{remark}\label{Rem:DiagramSituation}
All the arrows involved in diagram \eqref{diag:HamiltonianpentagonDiagram} can be expressed in term of the pairing $\pairing_-$ {and the operation $\ca$  defined in equation \eqref{Eq:RNProducts}}:
	\begin{align}\label{eq:Phim}
		\phimap_m =&
		\begin{cases}
			\id & m=1
			\\
			\varphi_m ~ \pairing_-^{\ca (m-1)} & m\geq 2
		\end{cases}	
		\\ \nonumber
		(\tau_B)_m =&
		\begin{cases}
			\id_{\Aspace} {-} 2 ~\langle B, \cdot \rangle_- & m=1
			\\
			0 & m\geq 2
		\end{cases}
		\\		\nonumber
		(f)_m =&
		\begin{cases}
			f_1: \xi \mapsto \pair{\rho(\xi)}{f_1(\xi)} \in \Aspace_0 & m=1
			\\
			f_m & m\geq 2
		\end{cases}
		\\ \nonumber
		(\widetilde{f})_m =&
		\begin{cases}
			f_1  {-} 2~ \langle B, \cdot \rangle_- \circ f_1
			& m=1 
			\\
			f_m {-} d_m ~\left(\pairing_-^{\ca (m-1)} \ca \langle B,\cdot \rangle_-\right) \circ f_1^{\otimes m} & m\geq 2
			~.
		\end{cases}		
	\end{align}
For the coefficients of $\phimap_m$ in \eqref{eq:Phim}   we use the short-hand notation (cf. equation \eqref{eq:PhiAnsatz})
\begin{equation}\label{eq:varphim}
\varphi_m:=\frac{2^{m-1}}{(m-1)!} B_{m-1}
~. 
\end{equation}

Observe that, due to Remark \ref{rem:PhikIndependentOmega}, the components of $\widetilde{\phimap}$ and $\phimap$ here agree, and are given in equation \eqref{eq:Phim}.

The coefficients $d_m$ are given by
$$d_m= \left( \frac{2^m}{m!}\right),$$ as follows {from Lemma \ref{lem:momaps}}, Lemma \ref{lemma:InsertionsAsPairing} and Corollary \ref{Cor:piccoloerroredisegnoneldraft} noting that $\mathsf{b}_m$ (defined in Lemma \ref{lem:momaps}) can be written as 
	\begin{equation}\label{eq:bm}
		\begin{split}
		\mathsf{b}_m(\xi_1,\dots \xi_m) &= 
		\varsigma(m+1)
		\iota_{\rho({\xi_n})}\dots \iota_{\rho({\xi_1})} B
		\\
		&=
			{-}\left(\varsigma(m+1)\varsigma(m) \frac{2^m}{m!}\right)
			~\left(\pairing_-^{\ca (m)} \right) \circ {\left(B\otimes \left(f_1^{\otimes m}(\xi_1,\dots \xi_m)\right)\right)}
			\\
			&=
			{- \dfrac{2^m}{m!} \left(\pairing_-^{\ca(m-1)}\ca \langle B, \cdot \rangle_- \right) \circ f_1^{\otimes m}}(\xi_1,\dots \xi_m)
	~.
	\end{split}
	\end{equation}
\end{remark}
	
\subsection{Proof of Theorem \ref{thm:comm}}\label{sec:proofthmcommut}
In this subsection we provide the proof of Theorem \ref{thm:comm}.
	
	To ascertain the strict commutativity of diagram \eqref{diag:HamiltonianpentagonDiagram} one has to make sure that 
	\begin{displaymath}
		(\tau_B \circ \phimap \circ f)_m -(\phimap\circ \widetilde{f})_m 
		= 0 \qquad \forall m\geq 1
		~.
	\end{displaymath}
	The case $m=1$ is straightforward:
	\begin{displaymath}
		\begin{split}
		(\tau_B \circ \phimap \circ f)_1 -(\phimap\circ \widetilde{f})_1 
		=&~
		(\tau_B)_1 \circ \phimap_1 \circ f_1 -\phimap_1 \circ \widetilde{f}_1
		\\
		=&~
		(\tau_B)_1  \circ f_1 - f_1 + 2 \langle B, \cdot\rangle_- \circ f_1 
		\\ 
		=&~ 0
		~.
		\end{split}
	\end{displaymath}
	{Alternatively, one can adapt the argument given in the symplectic case $n=1$ in Remark \ref{rem:symcomm}.}

{Now we let $m\ge 2$, i.e. we consider higher components.}	Since $\tau_B$ is a strict morphism and $(\tau_B)_1$ acts as the identity on any element of $\Aspace$ in degree different than $0$, 
	the higher cases requires to check that
	\begin{displaymath}
		{\big(\phimap \circ f\big)_m- \big(\phimap \circ \widetilde{f} \big)_m = 0}.
	\end{displaymath}

Thanks to 
 the expression in Theorem \ref{thm:iso},
 equation \eqref{eq:LinfinityMorphismsComposition} below {for  the composition of   $L_{\infty}$-morphisms    takes the simpler form
		\begin{displaymath}
			(\phimap \circ f )_m =
			\phimap_m \circ f_1^{\otimes m} + \left[
				\sum_{\ell=1}^{m-1} \phimap_\ell \circ \left( f_1^{\otimes (\ell -1)} \otimes f_{m-\ell+1} \right) 
				\circ P_{\ell -1, m-\ell +1}
			\right]
			~,
		\end{displaymath}
 and similarly for 		$(\phimap \circ \widetilde{f} )_m $.
 Recalling that, according to Theorem \ref{thm:iso}, $\phimap_m$ is proportional to $\pairing_-^{\ca (m{-}1)}$, 
 that the latter operator vanishes when evaluated on more than one element with null vector field component, and  that $\widetilde{f}_k = f_k + \mathsf{b}_k$ where $\mathsf{b}_k$ has no component in $\mathfrak{X}(M)$ for any $k\geq 1$, one gets
 	\begin{align*}
			\phimap_m \circ f_1^{\otimes m}
			-
			\phimap_m \circ \widetilde{f}_1^{\otimes m} 			
			=&~  - \sum_{\ell=0}^{m-1}
			\phimap_m \circ \left(f_1^{\otimes(m-1-\ell)}\otimes \mathsf{b}_1 \otimes f_1^{\otimes \ell}\right) 
			\\
			=&~
			- \phimap_m \circ \left( f_1^{\otimes (m-1)}\otimes \mathsf{b}_1 \right) \circ P_{m-1, 1}~, 	
 	\end{align*}
 	where the last equality follows from being $\phimap_m$ skew-symmetric.
 	For the same reasons, one gets 
	 		\begin{equation}\label{eq:square}
		(\phimap \circ f)_m
		-(\phimap\circ \widetilde{f})_m = 
		- \left[
				\sum_{\ell=1}^{m} 
				\phimap_\ell \circ 
				\left( f_1^{\otimes(\ell -1)} \otimes \mathsf{b}_{m-\ell+1} \right) 
				\circ P_{\ell -1, m-\ell +1} 			\right]
		~.
	\end{equation}

The following lemma allows to compute the summands on the right hand side of equation \eqref{eq:square}, since we have
$\phimap_{\ell}=\varphi_\ell \cdot \pairing_-^{\ca (\ell-1)}$  for $\ell \ge 2$ by  equation \eqref{eq:Phim}.
 
	\begin{lemma}\label{lem:rhsbm}
{For all $l\ge 1$ we have}
	 \begin{displaymath}
	 	\pairing_-^{\ca \ell -1} \circ 
	 	\left( f_1^{\otimes(\ell -1)} \otimes \mathsf{b}_{m-\ell+1} \right) \circ 
	 	P_{\ell -1, m-\ell +1} 
	 	= \binom{m}{\ell-1} \left[\frac{(\ell-1)!}{2^{\ell-1}}\right] \cdot \mathsf{b}_m
	 \end{displaymath}
	 where $\binom{m}{\ell-1}$ is the Newton binomial.
	\end{lemma}
	\begin{proof}
Recall that $\mathsf{b}_m ={-}d_m ~\left(\pairing_-^{\ca (m-1)} \ca \langle B,\cdot \rangle_-\right) \circ f_1^{\otimes m}$ {(see equation \eqref{eq:bm})}.
{We use this in the first and last equalities below,} to write
 the left hand side in the statement of the lemma as 
	\begin{align*}
			(l.h.s.) 
			=&~
			{-}d_{m-\ell+1} \cdot \pairing_-^{\ca (\ell -1)} \circ
			\left( 
				\Unit_{\ell-1} \otimes 
				\left(
					\pairing_-^{\ca (m -\ell)}\ca \langle B,\cdot \rangle_-
				\right)
			\right) \circ
			P_{\ell-1, m-\ell +1} \circ {f_1}^{\otimes m}		
			\\
			=&~
			{-}(-1)^{(\ell-1)(m-\ell)} ~			
			d_{m-\ell+1} \cdot 
			\pairing_-^{\ca (\ell -1)} \circ			
			\left( 
				\left(
					\pairing_-^{\ca (m -\ell)}\ca \langle B,\cdot \rangle_-
				\right) \otimes			
				\Unit_{\ell-1} 
			\right) \circ
			P_{m-\ell +1, \ell-1 } \circ {f_1}^{\otimes m}
			\\
			=&~
			{-}(-1)^{(\ell-1)(m-\ell)} ~			
			(-1)^{(\ell-1)({|\langle B,\cdot \rangle_-| -m+\ell)}} ~		
			d_{m-\ell+1} \cdot 
			\left( 
				\pairing_-^{\ca (m -1)}\ca \langle B,\cdot \rangle_-
			\right) \circ
			{f_1}^{\otimes m}
			\\
			=&~	
			\frac{d_{m-\ell+1}}{d_m} \cdot 
			\mathsf{b}_m
			~.
	\end{align*}
 The sign term in the second equality comes from noting that, for any graded $b$-multilinear map $\nu_b$ {on a} graded vector space $V$, one has
	\begin{displaymath}
		\Unit_a \otimes \nu_b =
		(-1)^{a(b+1)}~ \mathsf{C}_{(a+1)} \circ
		\left(\nu_b\otimes\Unit_a\right) \circ 
		\mathsf{C}^{-1}_{(a+b)}
		~,
	\end{displaymath}
	where $\Unit_a$ denotes the identity isomorphism on $V^{\otimes a}$ and
	 $\mathsf{C}_{(i)}$ denotes the odd action of the cyclic permutation on $V^{\otimes i}$.
	The sign term in the third equality comes from the sign convention in the definition of $\ca$.
	The final cancellation of  the  sign prefactor comes from noticing that $\langle B,\cdot \rangle_-$ is a degree $0$ operator from $\Lspace$ into itself.
 
		The claim now follows from an explicit computation of the coefficients:
$$\dfrac{d_{m-\ell+1}}{d_m}=\dfrac{1}{2^{\ell-1}} \dfrac{m!}{(m-\ell+1)!}
=	
		\binom{m}{\ell-1} \dfrac{(\ell-1)!}{2^{\ell-1}}.
$$			
	\end{proof}
 
{Thanks to Lemma \ref{lem:rhsbm}, 	 we can write equation \eqref{eq:square} as follows, for any $m\geq 2$:}
	\begin{equation}\label{eq:nicesum}
		(\phimap \circ f)_m 
		-(\phimap\circ \widetilde{f})_m = 
		- \left[
				\sum_{\ell=1}^{m}
				\binom{m}{\ell-1}
				~
				\dfrac{(\ell-1)!}{2^{\ell-1}}\varphi_\ell
			\right]
			\cdot 	\mathsf{b}_m
			~.
	\end{equation}	
Notice that this simplifies since $\dfrac{(\ell-1)!}{2^{\ell-1}}\varphi_{\ell}=B_{\ell-1}$  for $\ell=1,\dots,m$,
by the definition given in equation \eqref{eq:varphim}.}	
	We now make use of a {standard recursion formula}  for the    Bernoulli numbers, given by $B_0=1$ and
the following summation formula for all $m\ge 2$  (see for instance \cite{Agoh}\cite{Weisstein}):
	\begin{equation}\label{eq:weisstein}
  \sum_{j=0}^{m-1}
				\binom{m}{j}B_{j}=0.
\end{equation}	
It follows that the r.h.s. of equation \eqref{eq:nicesum} vanishes,
for all $m\ge 2$. We  conclude that diagram \eqref{diag:HamiltonianpentagonDiagram} commutes, finishing the proof of Theorem \ref{thm:comm}.

 \begin{remark}
	We recall the explicit formula for the composition of two $L_{\infty}$-morphisms, {used above}.
	Given two $L_\infty$-morphisms $f\colon V \to W$ and $g\colon W \to Z$	,
	the components of their  {composition}   $f\circ g$ are
	\begin{equation}\label{eq:LinfinityMorphismsComposition}
		(g \circ f)_m = \sum_{\ell=1}^m g_\ell \circ \mathcal{S}_{\ell,m} (f).
	\end{equation}
	(One can retrieve this formula decalaging the analog formula in 
\cite[\S 1]{Vitagliano2013}.)

Here the operator $\mathcal{S}_{\ell,m} (f)$ is the component $ V^{\otimes m} \to  W^{\otimes \ell}$ of the lift of $f$ 
	to a coalgebra morphism  $T^{\ge 1}V\to T^{\ge 1}W$ between the free tensor coalgebras of $V$ and $W$ (see \cite[Cor. 1.3.3 and Prop 1.5.3]{Manetti2011}). 
	Explicitly,
	\begin{equation}\label{Eq:Soperator}
		\mathcal{S}_{\ell,m} (f) =
		\left( 
			\sum_{\substack{k_{1}+\cdots+k_{\ell}=m\\1\leq k_{1}\leq\cdots\leq k_{\ell}}}
			(-1)^{\sum_{i=1}^{\ell-1}(|f_{k_i}|)(\ell-i)}
			(f_{k_1}\otimes\cdots\otimes f_{k_\ell})\circ P_{k_1,\ldots,k_\ell}^<		
		\right)~,
		\end{equation}
		with $P_{k_1,\ldots,k_\ell}^<$ denoting the sum over the (odd) action of  all permutations 
		$\sigma$ in $(k_1,\cdots,k_\ell)$-unshuffles on $V^{\otimes (k_1+\dots+k_\ell)}$ satisfying the extra condition
	\begin{displaymath}
		\sigma(k_1+\dots+k_{j-1}+1)<\sigma(k_1+\dots+k_{j}+1) 
		\quad \text{if}~k_{j-1}=k_j ~.
	\end{displaymath}	
\end{remark}

\appendix  
\section{Algebraic structure of multibrackets}
\label{app:Alg}
{Let $(M,\omega)$ be an $n$-plectic manifold, for a positive integer $n$. In this appendix we establish certain relations between the multibrackets
 of the $L_{\infty}$-algebras
 of Rogers and of Vinogradov introduced in \S 
 \ref{sec:multisympl} (or, rather, of the corresponding 
$L_{\infty}[1]$-algebras).
{We do so by means of concise computations}
 using the operation $\cs$   introduced in equation \eqref{eq:compsymm}. 
These relations are instrumental in making explicit the expressions of $\bpi'_n$ and $\bmu_n$ in the body of the paper (\S \ref{Sec:piPrime} and \S\ref{Sec:muvspiprime}, respectively).

\subsection{{The Nijenhuis-Richardson products $\cs$ and $\ca$}}\label{sec:BasicsNR}
Consider a graded vector space $V$, we denote by   $M_{n,k}^{sym}(V)$ and $M_{n,k}^{skew}(V)$ respectively 
 the spaces of degree $k$  symmetric and skew-symmetric  $n$-multilinear homogeneous maps on $V$ {with values in $V$}, {where $n\ge 1$.} \
 Hence, in particular, an element of $M_{n,k}^{sym}(V)$ is a degree $k$ linear map $S^k=V^{\odot k}\to V$.
 Considering all the possible arities and possible degrees collectively, one obtains two graded vector spaces 

 \begin{displaymath}
  M^{sym}(V):= \Big( k \mapsto \bigoplus_{n} M_{n, k}^{sym}(V)\Big)
  \qquad
  M^{skew}(V):= \Big( k \mapsto \bigoplus_{n+i=k+1} M_{n, i}^{skew}(V)\Big) . 
 \end{displaymath}
We endow them with the operation introduced in equation \eqref{eq:compsymm} and the {skew-symmetric} Nijenhuis-Richardson product \cite{Nijenhuis1967} (see also \cite[Ch. B]{Miti2021a}) respectively. 
They then constitute graded right pre-Lie algebras $(M^{sym}(V),\cs)$ and $(M^{skew}(V),\ca)$, see Remark \ref{rem:rightpreLie} for a characterization of this notion.
 
\begin{remark}\label{rem:NJiso}
We remark that $\ca$ is obtained from the product $\cs$  on graded symmetric multilinear maps on the shifted vector space $V[1]$, by precomposition with the  d\'ecalage introduced in equation \eqref{deca}. 
In other words, we have an isomorphism of graded algebras
 		\begin{equation}\label{eq:Dec}
			Dec: (M^{skew}(V), \ca) \xrightarrow{\quad \sim \quad} (M^{sym}(V[1]),\cs)
			~,
		\end{equation}  
which we call \emph{d\'ecalage of multilinear maps}.
The choice of two different gradings when defining $M^{skew}(V)$ and $M^{sym}(V)$ is justified, a posteriori,
by the fact that it keeps track of  
how the d\'ecalage of multilinear maps mixes arity (form degree) and the degree as a homogeneous map (weight degree).
Given a $k$-multilinear map $\mu_k$, we will always denoted by $|\mu_k|$ its (weight degree). 
When $\mu_k$ is skew-symmetric, its degree inside of the Nijenhuis-Richardson algebra $(M^{skew}(V), \ca)$ is given by $|\mu_k|+k-1$.
\end{remark}

Before proceeding, we establish some notation. Recall that a permutation $\sigma \in S_{n}$ is a $(i,n-i)$-unshuffle if $\sigma_{k}<\sigma_{k+1}$ for any $k\neq i$.
 We denote by $\ush{i_1,\dots, i_\ell}$ the subgroup of $(i_1,\dots,i_\ell)$-\emph{unshuffles permutations}.  
 We denote by  $B_{i_1,\dots, i_\ell}$ {(resp. $P_{i_1,\dots,i_\ell}$)} the operator summing over all {unsigned (resp. signed)}  permutations of the $(i_1,\dots,i_\ell)$-unshuffles subgroup $\ush{i_1,\dots,i_\ell}$. 
 Namely, denoting by $\epsilon(\sigma) $   the Koszul sign,
	\begin{displaymath}
		\begin{split}
		B_{i_1,\dots,i_\ell} ~\left( x_1\otimes x_2 \otimes \dots \right)
		=& \sum_{\sigma \in \ush{i_1,\dots,i_\ell}} 
		\epsilon(\sigma) x_{\sigma_1}\otimes x_{\sigma_2} \otimes \dots	
		\\
		P_{i_1,\dots,i_\ell} ~\left( x_1\otimes x_2 \otimes \dots \right) 
		=& \sum_{\sigma \in \ush{i_1,\dots,i_\ell}} 
		\epsilon(\sigma)(-1)^\sigma x_{\sigma_1}\otimes x_{\sigma_2} \otimes \dots	
		~.
		\end{split}
	\end{displaymath}
 Evaluating them on homogeneous elements $x_i\in V$, the two non-associative products read as follows:	\begin{equation}\label{Eq:RNProducts}
		\begin{split}
		 \mu_n \cs & \mu_m ~(x_1,\dots,x_{m+{n}-1}) =
		 \\
		 =&~
		 \sum_{\sigma \in \ush{m,n{-}1}}
		 \mkern-20mu		 
		  \epsilon(\sigma) 
		 \mu_n\Big(\mu_m(x_{\sigma_1},\dots,x_{\sigma_m}),x_{\sigma_{m+1}}\dots,x_{\sigma_{m+{n}-1}}	\Big)
		 \\
		 \\
		 \mu_n \ca & \mu_m ~(x_1,\dots,x_{m+{n}-1}) =
		 \\
		 =&~
		 (-1)^{|\mu_m|(n{-}1)}\mkern-20mu
		 \sum_{\sigma \in \ush{m,n{-}1}}\mkern-20mu
		  (-1)^\sigma \epsilon(\sigma) 
		 \mu_n\Big(\mu_m(x_{\sigma_1},\dots,x_{\sigma_m}),x_{\sigma_{m+1}}\dots,x_{\sigma_{m+{n}-1}}\Big)
		\end{split}
	\end{equation}
 where the sums run over all the $(m,n{-}1)$-unshuffles.

The Nijenhuis-Richardson product can be succinctly written as
	\begin{displaymath}
		\begin{split}
			\mu_n \cs \mu_m	
			=& 
			\mu_n \circ (\mu_m \otimes \Unit_{n{-}1}) \circ B_{m,n{-}1}	
			\\
			\mu_n \ca \mu_m	
			=& 
			(-1)^{|\mu_m|(n{-}1)}
			~\mu_n \circ (\mu_m \otimes \Unit_{n{-}1}) \circ P_{m,n{-}1}
			~,
		\end{split}
	\end{displaymath}
	denoting by $\Unit_k$ the identity isomorphism on $V^{\otimes k}$.
 		Note that such composition operators are well-defined for arbitrary multilinear maps, regardless of whether they are symmetric or skew-symmetric. Composing two arbitrary multilinear maps via $\cs$ or $\ca$ will not exhibit any symmetry in general.

 The non-associativity of $\cs$ is measured by the \emph{associators} 
 \begin{displaymath}
 	 \alpha(\cs ;\mu_\ell,\mu_m,\mu_n) = (\mu_\ell \cs \mu_m) \cs \mu_n - \mu_\ell \cs (\mu_m \cs \mu_n)~.
 \end{displaymath}
It can be proved, see e.g. \cite[Prop. B.1.27]{Miti2021a}, that the associators are multilinear operators given by the following equation:
 \begin{displaymath}
 	\alpha(\cs ;\mu_\ell,\mu_m,\mu_n)  = \mu_\ell \circ \left(\mu_m\otimes\mu_n\otimes \Unit_{\ell-2} \right) \circ B_{m,n,\ell-2}
 	\qquad (\ell\geq 2)	
 	~,
 \end{displaymath}
 and they vanish when $\mu_\ell$ is of arity $\ell=1$, due to $\mu_1\cs\mu_m=\mu_1\circ\mu_m$ and the associativity of $\circ$.
On arbitrary elements $x_k$, this reads as:
  \begin{align}
  	 \alpha&({\cs};\mu_\ell,\mu_m,\mu_n)(x_1,\dots,x_{m+n+\ell-2}) = 
  	 \notag
  	 \\
  	 =&
  	 \mkern-20mu\sum_{\sigma \in \ush{m,n,\ell-2}}\mkern-30mu 
  	 (-1)^{|\mu_n|(|x_1|+\dots+|x_m|)}\epsilon(\sigma)
  	  \mu_\ell\Big(\mu_m(x_{\sigma_1},\dots,x_{\sigma_m}),\mu_n(x_{\sigma_{m+1}},\dots,x_{\sigma_{m+n}}),x_{\sigma_{m+n{+}1}},\dots,x_{\sigma_{m+n+\ell-2}}\Big)~.
  	  \label{Eq:explicitassociators}
  \end{align}
  A similar expression hold for the skew-symmetric Nijenhuis-Richardson product, noting that $$Dec\big(\alpha(\ca ;\mu_\ell,\mu_m,\mu_n)\big)=\alpha\big(\cs ;Dec(\mu_\ell),Dec(\mu_m),Dec(\mu_n)\big)~.$$

\begin{remark}\label{rem:rightpreLie}
The right pre-Lie property of $\cs$ can be stated as the following symmetry property of the associator (see \cite[Ch. D]{Miti2021a} for further details):
    	\begin{equation*}
		\alpha(\cs;\mu,\nu,\pi) := (-1)^{|\nu||\pi|}~ \alpha(\cs;\mu,\pi,\nu) ~.
	\end{equation*}	 
\end{remark}

\begin{remark}
The commutator $$[\mu,\nu]_{ \cs} := \mu \cs \nu -(-1)^{|\mu||\nu|}\nu\cs \mu$$ 
satisfies the distributive properties
	\begin{equation}\label{eq:preLieDistrubutive}
	\begin{aligned}
		\big[\mu, (\nu \cs \pi)\big]_{ \cs}
		=&~ 
		\big[\mu,\nu\big]_{ \cs}\cs \pi + (-1)^{|\mu||\nu|}\nu \cs \big[\mu,\pi\big]_{ \cs} + \alpha\big(\cs; \mu,\nu,\pi\big)
		\\
		\big[(\nu\cs \pi),\mu\big]_{ \cs}
		 =&~
		 \nu\cs \big[\pi,\mu\big]_{ \cs} + (-1)^{|\mu||\pi|}\big[\nu,\mu\big]_{ \cs}\cs \pi - (-1)^{|\mu|(|\nu|+|\pi|)}\alpha\big(\cs; \mu,\nu,\pi\big)		
		~.
	\end{aligned}
	\end{equation}
	(Similar equalities apply to $\ca$.)
This is a consequence of the right pre-Lie property of $\cs$,  {as stated in Remark \ref{rem:rightpreLie}.}
\end{remark}

In the following,  we will omit the subscript $\cs$ or $\ca$ when indicating the commutator and the associator in the Nijenhuis-Richardson algebras. Everything should be clear from the context.

\subsection{Evaluation of the iterated pairing on elements.}\label{subsec:pairingevaluatins}
Consider the graded vector space $\Aspace$ of equation \eqref{eq:A}. We defined  the degree zero map  $\bS\colon S^{\ge 1}(\Aspace[1])\to \Aspace[1]$ as the extension of the pairing  on $\Aspace[1]$ which under  d\'ecalage corresponds to $\pairing_-$, see equation \eqref{eq:defbs} in Remark \ref{rem:pair-}.
The iterated products $\bS^k:= \bS\cs\dots\cs \bS$ are well-defined by Remark \ref{rem:nullassociators} below. 
They are important for our purposes since,  up to a numerical prefactor,  they are exactly the components of the $L_{\infty}[1]$-morphisms $\Phi$ constructed in Lemma \ref{lem:pansatazPhi}. 
In this subsection we   work out the evaluation of $\bS^k$ on strings of elements of $\Aspace[1]$.}

\begin{remark}[Vanishing of associators]\label{rem:nullassociators}
Although the $\cs$ operator is not associative,
when taking powers of the operator $\bS$, we do not need to to pay attention to the order in which the various copies of $\bS$ are composed. 
That is because  $\bS$ vanishes when evaluated on two elements of $\Aspace[1]$ with null vector field component, implying that the associators $\alpha(\bS,\bS,\bS)$ vanish.

More generally, the associator $\alpha(\bS^i,\eta_j,\chi_k)$ vanishes if, each time we
evaluate $\eta_j$ and $\chi_k$ on elements of $\Aspace[1]$, the output  has null vector field component. 
 {(In this case, according to equation \eqref{Eq:explicitassociators}, the associator involves plugging into $\bS^i$ at least two elements with null vector field component.)}
 This happens in particular when 
 $\eta_j$ has degree  $<j-1$ or equals $\bpi_1$, and  $\chi_k$ has degree  $<k-1$ or equals $\bpi_1$.
We deduce that 
$$
	\alpha(\cs;\bS^i,\eta,\chi)=0
	\quad\text{ for }\quad
	\eta,\chi\in \{\bS^j\}_{j \ge 1}\cup \{\bpi_k\}_{k \neq 2}
	~.
$$ 
We will use this repeatedly in \S\ref{subsec:PropRogersAlg}.
\end{remark}

Instead of working directly with {$\bS|_{\Aspace[1]^{\odot 2}}$}, we will work with the  graded skew-symmetric pairing $\pairing_-$ corresponding to it under d\'ecalage.
Actually, we will consider the larger graded vector space $\widehat{\Vspace}:=\mathfrak{X}(M)\oplus\Omega(M){[n{-}1]}$
 obtained by extending $\Vspace$ and $\Aspace$ with differential forms of all degrees and view $\pairing_-$ as a pairing defined there (by the same formula as in equation \eqref{Eq:PairingExtensions}).
Further we denote by $\rho$ the standard projection $\rho:\widehat{\Vspace} \to\mathfrak{X}(M)$.

\color{black}

{The following technical lemma is used in}
{Remark \ref{Rem:DiagramSituation} to express the action of the gauge transformation of a homotopy comoment map in terms of the pairing.
The key idea is that}   one can employ the operator $\pairing_-$ to express the insertion of  several vector fields in a given differential form as a ``power'' of the pairing.
	\begin{lemma}[Insertions as pairing]\label{lemma:InsertionsAsPairing}
		Given an arbitrary differential form  $B$, i.e. a homogeneous element in $\ker(\rho)\subset \widehat{\Vspace}$, and given vector fields  
		the following equation holds  for all  {$m\geq 0$}:
		\begin{displaymath}
			\left(
			\pairing_-^{\ca m}
			\right)(B, X_1,\dots, X_m)
			=
			\left(-\varsigma(m) \cdot \frac{m!}{2^m}
			\right)~
			\iota_{X_m}\dots \iota_{X_1} B
			~.
		\end{displaymath}
Here the left hand side denotes the evaluation of operator $\pairing_-^{\ca m}$
on the element $ B\otimes X_1\otimes \dots \otimes X_m \in \widehat{\Vspace}^{\otimes m+1}$.
	\end{lemma}
	\begin{proof}
		By induction over $m$. 
		Observe first that 
		$$ \pairing_- ~ (B,X_1)= -\frac{1}{2}\iota_{X_1} B~. $$		
		Assume now that the statement holds for $m$.
Then the statement for $m+1$ follows:
		\begin{displaymath}
			\begin{split}
				{\pairing_-\ca \pairing_-^{\ca(m)}}
				&~(B,X_1,\dots,X_{m+1})
				=
				\\
				=&
						(-1)^{|\pairing_-^{\ca (m)}|}
				\sum_{\sigma \in \ush{m,1}} \chi(\sigma) 
				\left\langle 
					\left( 
						\pairing_-^{\ca(m)} ~(B,X_{\sigma_1},\dots,X_{\sigma_m})
					\right)
					, X_{\sigma_{m+1}}				
				\right \rangle_-
				\\
				=&
				(-1)^m
				\sum_{\sigma \in \ush{m,1}} \chi(\sigma)
				\left(-\varsigma(m) \frac{m!}{2^m}\right)\left(-\frac{1}{2}\right) 
				\iota_{X_{\sigma_{m+1}}}\iota_{X_{\sigma_m}}\dots \iota_{X_{\sigma_1}} B
				\\
				=&
				\left(
					(-1)^m \varsigma(m) \frac{(m+1)!}{2^{m+1}}
				\right)
				\iota_{X_{m+1}}\dots \iota_{X_1} B
				~,
			\end{split}
		\end{displaymath}
		where $\chi(\sigma)=(-1)^\sigma \epsilon(\sigma)$ denotes the odd Koszul sign.
		The claim follows by noticing that $\varsigma(k-1)\varsigma(k)=(-1)^k$.
		Remark \ref{rem:nullassociators} ensures associativity in the first equality.
	\end{proof}
 
We re-express the above statement by singling out the contraction with the differential from $B$.
\begin{corollary}\label{Cor:piccoloerroredisegnoneldraft}
For any given differential form $B$ and integer $m\ge 1$ one has:
		\begin{displaymath}
			\pairing_-^{\ca (m-1)} \ca \langle B, \cdot \rangle_- = (-1)^{m(|B|-n+1)} \pairing_-^{\ca (m)} B \otimes \Unit_{m}
			~.
		\end{displaymath}
\end{corollary}
\begin{proof}
	Observe first that any given differential form $B$ can seen as a degree $|B|-n+1$ homogeneous element in $\widehat{\Vspace}$.
	Hence, one can consider the degree $(|B|-n)$ unary operator $\langle B, \cdot \rangle_-$ given by
	$$
		\langle B, \cdot \rangle_-~(X_1) = - \frac{1}{2}\iota_{X_1} B~,
	$$
	for any $X_1\in \mathfrak X(M)$. By further inspection on vector fields, one has
	\begin{align*}
		\pairing_-^{\ca (m-1)} \ca \langle B, \cdot \rangle_- & (X_1,\dots,X_{m}) =
		\\
				&=
				(-1)^{m(|\langle B, \cdot \rangle_-|)}
				\sum_{\sigma \in S_{m}}\chi(\sigma)
				\pairing_-^{\ca(m-1)} \Big( \langle B, X_{\sigma_1}\rangle_-,X_{\sigma_2},\dots, X_{\sigma_m}\Big)
				\\
				&=
				(-1)^{m(|B| -n)}
				\left( - \frac{m}{2}\right)
				\pairing_-^{\ca (m-1)}
				\Big( \iota_{X_1}B ,X_{\sigma_2},\dots, X_{\sigma_m}\Big)
				\\
				&=
				(-1)^{m(|B| -n)}
				\left( \frac{m!}{2^m}\right)
				\varsigma(m{-}1)
				\iota_{X_m}\dots\iota_{X_1} B
				\\
				&=
				(-1)^{m(|B|-n+1)}
				\left[
					-\varsigma(m)
					\left(\frac{m!}{2^m}\right)
					\iota_{X_m}\dots\iota_{X_1} B
				\right]
				\\
				&=
				(-1)^{m(|B|-n+1)}
				\pairing_-^{\ca (m)} (B,X_1,\dots X_m) ~,
	\end{align*}
	employing lemma \ref{lemma:InsertionsAsPairing} in the third and last equalities.
\end{proof}

 The following corollary is used in the proof of Theorem \ref{thm:iso} to spell out the $L_{\infty}$-embedding obtained there.

\begin{corollary}\label{cor:EvaluatedExpression}
Consider  $\Velem_i=X_i+\beta_i$
with $X_i\in\mathfrak X(M)$ and ${\beta_i\in \Omega(M)}$. 
For $m\ge 1$ one has that
\begin{displaymath}
	\pairing_-^{\ca m}(\Velem_1,\dots,\Velem_{m+1}) =
	\left[\frac{m!}{2^m} \right]
	\sum_{\j=1}^{m+1} 
	(-1)^{j+m+1}
	\iota_{X_1}\dots\widehat{\iota_{X_j}}\dots\iota_{X_{m+1}}\beta_{j}~.
\end{displaymath}
\end{corollary}

\begin{proof}
	Observe that 
	$$
	\pairing_-=  \pairing_-\circ (\Unit\otimes \rho) \circ P_{1,1}~,
	$$
since $\pairing_-\circ (\Unit\otimes \rho)(\Velem_1,\Velem_2) = -\iota_{X_2}\beta_1$ and $P_{1,1}$ coincides with the graded skew-symmetrization operator.
	Iterating the previous equation, one gets
	$$
		\pairing_-^{\ca m}=  \pairing_-^{\ca m}\circ (\Unit\otimes \rho^{\otimes m}) \circ P_{1,m}
		~.
	$$
	It follows from Lemma \ref{lemma:InsertionsAsPairing} that
	\begin{align*}
		\pairing_-^{\ca m}(\Velem_1,\dots,\Velem_{m+1}) 
		=&~
		 {-} \varsigma(m) 
		\left[		 
		 \frac{m!}{2^m}
		\right]
		 \sum_{\sigma \in \ush{1,m}} 
		 {\chi(\sigma)~}
		 \iota_{X_{\sigma_{m+1}}}\dots \iota_{X_{\sigma_{2}}} \beta_{\sigma_1}
		\\
		=&~
		 (-1)^{m+1}
		\left[
		 \frac{m!}{2^m}
		\right]
		 \sum_{\sigma \in \ush{1,m}} {\chi(\sigma)~}
		 \iota_{X_{\sigma_{2}}}\dots \iota_{X_{\sigma_{m+1}}} \beta_{\sigma_1}
		\\
		=&~
		\left[
		\frac{m!}{2^m}
		\right]
		 \sum_{\sigma \in \ush{m,1}} 
		 {\chi(\sigma)~}
		 \iota_{X_{\sigma_{1}}}\dots \iota_{X_{\sigma_{m}}} \beta_{\sigma_{m+1}}~.
	\end{align*}
\end{proof}

\subsection{{Properties} of Rogers' $L_{\infty}[1]$-algebra}\label{subsec:PropRogersAlg}
Consider again the graded vector space $\Aspace$ introduced in equation \eqref{eq:A}.
Denote by $\pi_k$ the $k$-th multibracket induced on $\Aspace$ from Rogers' Lie infinity algebra $L_\infty(M,\omega)$ via the isomorphism $\Aspace \cong \Lspace$ of equation \eqref{eq:LMomegaA}. 
Denote by  $(\Aspace[1],\{\bpi_k\})$ the corresponding $L_{\infty}[1]$-algebra.

In this subsection we work out the iterated commutators (w.r.t. $\cs$) of powers 
$\bS^{l}$ with $\bpi_k$, in some of the relevant cases, yielding a general result in Proposition \ref{prop:marcoestate123}. We need to do this because
the pushforward  $\{\bpi'_k\}$ of Roger's $L_{\infty}[1]$-brackets $\{\bpi_k\}$ by $\Phi$ -- see Equations \eqref{eq:pin'exp} and \eqref{eq:summandsIteratedCommu} -- is expressed exactly as such an iterated commutator.
Along the way, we express all higher multibrackets $\bpi_k$ in terms of  powers 
$\bS^{l}$  and $\bpi_3$ (Proposition \ref{cor:higherpi}). 
 
\medskip

 We carry out some of the proofs in terms of the multilinear maps $\pi$ on $\Aspace$, rather than using the corresponding graded-symmetric maps $\bpi$ on $\Aspace[1]$. This is possible thanks to the graded algebra isomorphism $Dec$ given in equation \eqref{eq:Dec}, {and it is convenient because it allows us to apply the identities of Cartan calculus.}
In the following, we will denote by $\Aelem_i \in \Aspace$ a generic non-homogeneous vector of $\Aspace$.
Such an element can be decomposed as $\Aelem_i= f_i {+} e_i$ where
$f_i \in \bigoplus_{k=0}^{n-2}\Omega^k(M)$,
$e_i = \pair{X_i}{\alpha_i} \in \mathfrak{X}(M)\oplus \Omega_{\mathrm{Ham}}^{n{-}1}(M)$, and $X_i = v_{\alpha_i}$ is the Hamiltonian vector field of $\alpha_i$.
We define the contraction of an arbitrary element of $\Aspace$ with a vector field $Y$ as $\iota_Y \Aelem_i = \iota_Y (f_i +\alpha_i)$.

\begin{lemma}

\end{lemma}
 
\subsubsection{Higher multibrackets $\bpi_k$ in terms of $\bpi_3$}
We start expressing all higher multibrackets $\bpi_k$ in terms of  powers  $\bS^{l}$  and $\bpi_3$.

\begin{remark}\label{rem:two}
 We have
 $$ [\bS, \bpi_k ] = \bS \cs \bpi_k ~,$$
since $\bpi_k$ with $k\geq 2$ yields a  non-zero output only when evaluated on top degree elements.

\end{remark}

\begin{lemma}[Higher Rogers multibrackets recursive formula]\label{lem:rogersRecurFormula}
 \begin{displaymath}
	[\bS,\bpi_{k-1}]
	= 
	\frac{k}{2} \bpi_k \qquad \forall k \geq 4
 \end{displaymath}
\end{lemma}

\begin{proof}
	Inspecting on arbitrary elements $\Aelem_i=f_i + \pair{X_i}{\alpha_i}\in \Aspace$ one gets
	\begin{displaymath}
		\begin{split}
			\pi_k(\Aelem_1,\dots,\Aelem_k) 
			&= \varsigma(k)\omega(X_1,\dots,X_k) 
			\\
			&=
			\varsigma(k)\sum_{\sigma\in \ush{k-1,1}} \frac{1}{k}\iota_{X_{\sigma_k}}\omega(X_{\sigma_1},\dots,X_{\sigma_{k-1}})
			\\
			&=
			\left(\frac{\varsigma(k)\varsigma(k-1)}{k}\right)
			\sum_{\sigma\in \ush{k-1,1}}
			\iota_{X_{\sigma_k}}~\pi_{k-1}(\Aelem_{\sigma_1},\dots,\Aelem_{\sigma_{k-1}})
			\\
			&= -\frac{2}{k} (-1)^k 
			\sum_{\sigma\in \ush{k-1,1}}
			\Big\langle \pi_{k-1}(\Aelem_{\sigma_1},\dots,\Aelem_{\sigma_{k-1}}), \Aelem_{\sigma_k} \Big\rangle_-
			\\
			&= -\frac{2}{k} (-1)^k (-1)^{|\pi_{k-1}|}
			\pairing_- \ca \pi_{k-1} (\Aelem_1,\dots \Aelem_k)
			\\
			&=
			\frac{2}{k}			\pairing_- \ca \pi_{k-1} (\Aelem_1,\dots \Aelem_k)		
		\end{split}
	\end{displaymath}
	The claim follows after d\'ecalage.
\end{proof}

According to the next corollary, the multibrackets of Rogers' $L_\infty[1]$ structure on $\Aspace[1]$ can be obtained by suitable compositions of the four elements $\{\bS, \bpi_1, \bpi_2, \bpi_3\}$, regarding each of them as a map $S^{\ge 1}(\Aspace[1])\to \Aspace[1]$, via the product $\cs$.

	\begin{proposition}\label{cor:higherpi}
	For all $n\ge 3$ we 	have\footnote{Note again that the order of composition in the above equation is irrelevant, as Remarks \ref{rem:nullassociators} implies that $\alpha(\bS,\bS,\bpi_3)=0$.}
			\begin{displaymath} 
			\bpi_n = \left( \frac{2^{n-3}}{n!} 3!\right) \bS^{n-3} \cs \bpi_3~.
		\end{displaymath}
	\end{proposition}
Note that the order of composition in the above equation is irrelevant since, according again to equation \eqref{Eq:explicitassociators}, $\alpha(\bS,\bS,\bpi_3)=0$.
	\begin{proof}
			This statement can be proved iterating Lemma \ref{lem:rogersRecurFormula}. Alternatively, once can employ   Lemma \ref{lemma:InsertionsAsPairing}.  
	\end{proof}

\subsubsection{Iterated commutators with powers of the pairing}\label{Sec:IteratedCommutators}

 Here we compute iterated commutators of powers with $\bS^l$ with
 Rogers' $L_{\infty}[1]$-multibrackets $\bpi_k$.  In some basic cases we do this by evaluation on elements;
for this purpose we use the symmetric pairing 
$\pairing_+ : \Vspace \otimes \Vspace \to \Vspace$,
which is defined  analogously to $\pairing_-$
 in equation \eqref{Eq:PairingExtensions} but replacing the minus sign there with a plus sign.

	\begin{remark}\label{rem:cawithSymBra}
		Observe that Equation \ref{Eq:RNProducts} defining $\cs$ and $\ca$ naturally extend to well-defined operations on arbitrary multilinear maps.
In particular, for any given bilinear operator $\eta_2$ and skew-symmetric operator $\mu_2$, the composition $\eta_2 \ca \mu_n$ yields again a skewsymmetric operator.
Furthermore, one can see by direct inspection on homogeonous elements  that $\pairing_- \ca \pairing_+ =0$ due to the anti-commutation of the contraction operator. 
	\end{remark}

\begin{lemma}[Iterated commutator of  $\bS$ with $\bpi_1$ of arity $3$]
\label{Prop:TernaryCommutator}
 \begin{displaymath}
  [\bS,[\bS, \bpi_1]]=	[\bS,\bpi_2]~.
 \end{displaymath}
\end{lemma}
We point out that the
		  computations in the proof below are similar to -- but more concise than --  those found in \cite[Lemmas 7.2, 7.3 7.4]{Rogers2013}.
		  
\begin{proof}
	First note by inspecting on elements 
	 $\Aelem_i=f_i + \pair{X_i}{\alpha_i}\in \Aspace$ that:
	\begin{displaymath}
		2 \big(\pairing_- \ca \pi_2\big)~(\Aelem_1,\Aelem_2,\Aelem_3)
		=
		\iota_{[X_1,X_2]} \Aelem_3 - \omega(X_1,X_2,X_3) + \cyc		
		~.	
	\end{displaymath}
	Using Cartan's magic formula twice we obtain
	\begin{align}
		\iota_{[X_1,X_2]} \Aelem_3 &+\cyc
		=
		\notag
		\\
		&=
		\mathcal{L}_{X_1} \iota_{X_2} \Aelem_3 - \iota_{X_2} \mathcal{L}_{X_1} \Aelem_3 +\cyc
		\notag
		\\
		&=
		(\iota_{X_1} \dd + \dd \iota_{X_1}) \iota_{X_2} \Aelem_3 
		- \iota_{X_2} (\dd \iota_{X_1} + \iota_{X_1} \dd ) \Aelem_3 +\cyc
		\label{eq:bruttocoso}
		\\
		&=
		\dd \iota_{X_1} \iota_{X_2} \Aelem_3 
		+ \iota_{X_1} \dd \iota_{X_2} \Aelem_3 
		- \iota_{X_2} \dd \iota_{X_1} \Aelem_3 
		+ \iota_{X_2}\iota_{X_1}\iota_{X_3} \omega 
		- \iota_{X_2} \iota_{X_1} \mu_1 \Aelem_3 +\cyc
		\notag
		\\
		&=
		(\mu_1 \iota_{X_1} \iota_{X_2} \Aelem_3) 
		+ (\iota_{X_1} \mu_1 \iota_{X_2} \Aelem_3 - \iota_{X_2} \mu_1 \iota_{X_1} \Aelem_3) 
		- (\iota_{X_2} \iota_{X_1} \mu_1 \Aelem_3) 
		+ \omega(X_1,X_2,X_3)	+\cyc  \notag
	\end{align}
	where, in the penultimate equation, it is employed that:
	\begin{displaymath}
		\dd~ \Aelem_3 = \dd (\alpha_3 + f_3) = -\iota_{X_3} \omega + \mu_1 \Aelem_3
		~.
	\end{displaymath}
	{The first three terms on the r.h.s. of} equation \eqref{eq:bruttocoso} can be recast as follows:
	\begin{align*}
		\mu_1 \iota_{X_1} \iota_{X_2} \Aelem_3 &+ \cyc
		=
		\\
		&=
		\mu_1 \iota_{X_3} \iota_{X_1} \Aelem_2 + \cyc 
		\\
		&= 
		\mu_1 \iota_{X_3} (\langle \Aelem_1 , \Aelem_2\rangle_+ + \langle \Aelem_1 , \Aelem_2\rangle_- ) + \cyc
		\\
		&=
		- 2 \mu_1 \langle (\langle \Aelem_1 , \Aelem_2\rangle_+ + \langle \Aelem_1 , \Aelem_2\rangle_- ), \Aelem_3 \rangle_- + \cyc
		\\
		&=
		 2 \mu_1 \pairing_- \ca \big(\pairing_+ + \pairing_- \big) (\Aelem_1,\Aelem_2,\Aelem_3)
		\\
		&=
		2 \Big(\mu_1 \pairing_- \ca \pairing_-\Big)~(\Aelem_1,\Aelem_2,\Aelem_3)		
	\end{align*}
	\begin{align*}
		\iota_{X_1} \mu_1 \iota_{X_2} \Aelem_3 &- \iota_{X_2} \mu_1 \iota_{X_1} \Aelem_3 + \cyc 
		=
		\\
		&=
		\iota_{X_3} \mu_1 \iota_{X_1} \Aelem_2 
		- \iota_{X_3} \mu_1 \iota_{X_2} \Aelem_1 + \cyc
		\\
		&= 
		2 \iota_{X_3} \mu_1 \langle \Aelem_1,\Aelem_2 \rangle_- + \cyc 
		\\
		&= 
		-4 \langle \mu_1 \langle \Aelem_1,\Aelem_2 \rangle_-, \Aelem_3 \rangle_- + \cyc 
		\\
		&=
		-4 \Big(\pairing_- \ca \mu_1 \ca \pairing_-\Big) (\Aelem_1,\Aelem_2,\Aelem_3)
	\end{align*}
	\begin{align*}
		-\iota_{X_2} \iota_{X_1} \mu_1 \Aelem_3 &+ \cyc
		=
		\\
		&=
		-\iota_{X_3} \iota_{X_2} \mu_1 \Aelem_1 + \cyc 
		\\
		&= 
		- \dfrac{1}{2}\iota_{X_3}
		\big(\iota_{X_2} \mu_1 \Aelem_1 - \iota_{X_1} \mu_1 \Aelem_2) + \cyc
		\\
		&=
		+ \iota_{X_3}
		\big(\langle \mu_1 \Aelem_1, \Aelem_2 \rangle_- - \langle \mu_1 \Aelem_2,\Aelem_1 \rangle_-\big) + \cyc		
		\\
		&=
		- 2 \langle
		\big(\langle \mu_1 \Aelem_1, \Aelem_2 \rangle_- - \langle \mu_1 \Aelem_2,\Aelem_1 \rangle_-\big), \Aelem_3 \rangle_- + \cyc			
		\\
		&= 2 \Big(\pairing_- \ca (\pairing_- \ca \mu_1)\Big)~(\Aelem_1,\Aelem_2,\Aelem_3)~.
	\end{align*}
	Note that in expressing the first term we used that $\pairing_- \ca \pairing_+=0$ (see Remark \ref{rem:cawithSymBra}).
	
	Hence, after d\'ecalage, one gets:
	\begin{align}
		[\bS,\bpi_2]
		=&~\bS \cs \bpi_2 
		\notag
		\\
		=&~
		\bpi_1 \cs \bS\cs \bS
		-2 \bS \cs \bpi_1 \cs \bS 
		+ \bS \cs \bS \cs \bpi_1
		\label{Eq:pairing-pi2}
		\\
		=&~ [\bS,[\bS, {\bpi_1}]]
		\notag
		~.	
	\end{align}
	Note that it is not necessary to pay attention to the order of associativity in the above equations since $\alpha(\bpi_1,\bS,\bS)=\alpha(\bS,\bpi_1,\bS)=0$, see Remark \ref{rem:nullassociators}.
\end{proof}

\begin{lemma}[Iterated commutator of  $\bS$ with $\bpi_1$ of arity $4$]
\label{Prop:QuaternaryCommutator}
	\begin{equation*}
			[\bS,[\bS,[\bS,\bpi_1]]] = 3 \bpi_4 ~.
	\end{equation*}
\end{lemma}	
\begin{proof}
	Observe first that Lemma \ref{Prop:TernaryCommutator} implies that
	$[\bS,[\bS,[\bS,\bpi_1]]] = [\bS,[\bS,\bpi_2]]$.
	Expanding the r.h.s. in terms of the product $\cs$, and keeping track of the non-associativity, one gets that 
	$$
		[\bS,[\bS,\bpi_2]] 
		=
		[\bS^{ 2}, \bpi_2 ] -	2 \alpha (\bS,\bS,\bpi_2)~.
	$$
	The claim follows from the next Lemma \ref{Lemma:BoringAssociator}, which provides an explicit computation of the  above associator. \end{proof}

\begin{lemma}[{A} recurrent associator]\label{Lemma:BoringAssociator}
	\begin{displaymath}
		\alpha(\cs; \bS,\bS, \bpi_2) = \frac{1}{2}\Big( [\bS^{ 2},\bpi_2] -3 \bpi_4\Big)
	\end{displaymath}
\end{lemma}
\begin{proof}
	Observe that
	\begin{equation}\label{eq:obscure}
		2 \pairing_- \ca \pi_2 = K + 3 \pi_3
	\end{equation}
	where the auxiliary operator $K$ is given by the following equation 
	\begin{equation*}		
		\begin{split}	
		K(\Aelem_1,\Aelem_2,\Aelem_3)=&~
		(\pairing_+ + \pairing_-)\ca \pi_2 (\Aelem_1,\Aelem_2,\Aelem_3)
		\\
		=&~
		\iota_{[X_1,X_2]}\Aelem_3 + \iota_{[X_2,X_3]}\Aelem_1 + \iota_{[X_3,X_1]}\Aelem_2
		~,
		\end{split}
	\end{equation*}	
	for any $\Aelem_i=f_i + \pair{X_i}{\alpha_i}\in \Aspace$.
		Note that, as   stated in Remark \ref{rem:cawithSymBra}, $\pairing_+\cs \pi_2$ gives  a well-defined skewsymmetric operators.
		
	According to equation \eqref{Eq:explicitassociators},  the following  holds: 
		\begin{equation}\label{eq:Kalpha}
		\begin{split}
			\alpha(\ca;\pairing_-,\pairing_-,\pi_2) &(\Aelem_1,\Aelem_2,\Aelem_3,\Aelem_4)
			= 
			\frac{1}{2}\iota_{[X_3,X_4]}\langle \Aelem_1,\Aelem_2 \rangle_- + \unsh{(2,2)}
			\\
			=&~
			\frac{1}{2}\biggr(			
			+ \iota_{[X_3,X_4]}\langle \Aelem_1,\Aelem_2\rangle_- 
			+ \iota_{[X_1,X_2]}\langle \Aelem_3,\Aelem_4\rangle_-
			- \iota_{[X_2,X_4]}\langle \Aelem_1,\Aelem_3\rangle_-
			+
			\\
			&\phantom{\frac{1}{2}\biggr(}- \iota_{[X_1,X_3]}\langle \Aelem_2,\Aelem_4\rangle_-
			+ \iota_{[X_2,X_3]}\langle \Aelem_1,\Aelem_4\rangle_-
			+ \iota_{[X_1,X_4]}\langle \Aelem_2,\Aelem_3\rangle_-
			\biggr)~,
		\end{split}
	\end{equation}	
	{where $\unsh{(2,2)}$ denotes sum over all $(2,2)$-unshuffles.}
	\\
	{On the other hand, one finds that}
	\begin{displaymath}
	\begin{split}
		\big(\pairing_-&\ca K\big) (\Aelem_1,\Aelem_2,\Aelem_3,\Aelem_4)
		=
		(-1)^{|K|} \langle K(\Aelem_1,\Aelem_2,\Aelem_3),\Aelem_4 \rangle_- + \unsh{(3,1)}
		\\
		=&
		-(\langle K(\Aelem_1,\Aelem_2,\Aelem_3),\Aelem_4\rangle_-
		- \langle K(\Aelem_1,\Aelem_2,\Aelem_4),\Aelem_3\rangle_-
		+ \langle K(\Aelem_1,\Aelem_3,\Aelem_4),\Aelem_2\rangle_-
		- \langle K(\Aelem_2,\Aelem_3,\Aelem_4),\Aelem_1\rangle_-)
		\\
		=&\frac{1}{2}\Big\{
			+\iota_{X_4}(\iota_{[X_1,X_2]}\Aelem_3+\iota_{[X_2,X_3]}\Aelem_1+\iota_{[X_3,X_1]}\Aelem_2)+
		\\
		&\phantom{\frac{1}{2}\big\{}
			-\iota_{X_3}(\iota_{[X_1,X_2]}\Aelem_4+\iota_{[X_2,X_4]}\Aelem_1+\iota_{[X_4,X_1]}\Aelem_2)+
		\\
		&\phantom{\frac{1}{2}\big\{}
			+\iota_{X_2}(\iota_{[X_1,X_3]}\Aelem_4+\iota_{[X_3,X_4]}\Aelem_1+\iota_{[X_4,X_1]}\Aelem_3)+
		\\
		&\phantom{\frac{1}{2}\big\{}
			-\iota_{X_1}(\iota_{[X_2,X_3]}\Aelem_4+\iota_{[X_3,X_4]}\Aelem_2+\iota_{[X_4,X_2]}\Aelem_3)
		\Big\}
		\\
		=&+ \iota_{[X_1,X_2]}\langle \Aelem_3, \Aelem_4\rangle_- 
			- \iota_{[X_1,X_3]}\langle \Aelem_2, \Aelem_4\rangle_-
			+ \iota_{[X_1,X_4]}\langle \Aelem_2, \Aelem_3\rangle_-
		+\\&
			+ \iota_{[X_2,X_3]}\langle \Aelem_1, \Aelem_4\rangle_-
			- \iota_{[X_2,X_4]}\langle \Aelem_1, \Aelem_3\rangle_-
			+ \iota_{[X_3,X_4]}\langle \Aelem_1, \Aelem_2\rangle_-
		\\=&
			2~ \alpha(\ca;\pairing_-,\pairing_-,\pi_2) (\Aelem_1,\Aelem_2,\Aelem_3,\Aelem_4)
			~,
	\end{split}
	\end{displaymath}
{using equation \eqref{eq:Kalpha} in the last equality. In other words,}
	\begin{displaymath}
		\pairing_- \ca K = 2~ \alpha(\ca;\pairing_-,\pairing_-,\pi_2)
		~.
	\end{displaymath}
	Plugging this last result into the equation obtained by {composing} equation \eqref{eq:obscure} with $\pairing_-$ from the left one gets:
	\begin{displaymath}
		2~\pairing_- \ca (\pairing_-\ca \pi_2 ) =
		2~\alpha(\ca;\pairing_-,\pairing_-,\pi_2) +
		3~ \pairing_- \ca \pi_3
		~.
	\end{displaymath}			
	
Applying on the l.h.s. the definition of the associator and on the r.h.s. the equality $\pairing_-\ca \pi_3 = 2 \pi_4$ coming from Proposition \ref{cor:higherpi}, 
one gets
	\begin{displaymath}
		[ \pairing_-^{\ca 2}, \pi_2]  = 2~\alpha(\ca; \pairing_-, \pairing_-, \pi_2 ) + 3\,\pi_{4}
		~.
	\end{displaymath}		
	The claim follows after d\'ecalage.
\end{proof}

The identities we obtained so far  in this subsection were proven by evaluation on strings of elements of $\Aspace$.
Next, using properties of the product $\cs$, we can easily obtain more general identities.

 \begin{corollary}\label{lemma:pairpi2-as-pairpairPi1}
 	Given $n\geq 1$,
 	\begin{displaymath}
 		[\bS^{ n}, [\bS,\bpi_1]] = [\bS^{ n}, \bpi_2] 
 		~.
 	\end{displaymath}
 \end{corollary}
 \begin{proof}
 	By induction.
 	The basis of induction (the case $n=1$) is given by Lemma \ref{Prop:TernaryCommutator}.
 	From the distributive property of pre-Lie algebras commutators (equation \eqref{eq:preLieDistrubutive}) one has that
 	\begin{displaymath}
		[\bS^n, \bpi_2] = \bS\cs [\bS^{n{-}1},\bpi_2] + [\bS,\bpi_2] \cs\bS^{n{-}1} + \alpha(\bpi_2,\bS,\bS^{n{-}1})~, 	
 	\end{displaymath}
 	where the last term is the associator,  
	which vanishes by 
equation \eqref{Eq:explicitassociators} and Remark \ref{rem:two}.
	Plugging the induction hypothesis in the r.h.s. of the above equation implies that
	\begin{align*}
		[\bS^n, \bpi_2] =&~
		\bS \cs[\bS^{n{-}1},[\bS,\bpi_1]] + [\bS,[\bS,\bpi_1]]\cs \bS^{n{-}1}
		\\
		=&~
		\bS^{n{+}1} -\bS^n \cs\bpi_1 \cs\bS - \bS^2\cs\bpi_1\cs\bS^{n{-}1} + \bS \cs\bpi_1\cs \bS^n +
		\\
		&~+ \bS \cs\bpi_1 \cs\bS^n -2 \bS \cs\bpi_1 \cs\bS^n + \bpi_1 \cs\bS^{n{+}1}
		\\
		=&~ [\bS^n,[\bS,\bpi_1]]~.
	\end{align*}
 \end{proof}

\begin{corollary}\label{prop:marcoestate1}
	For any $k_i \geq 1$ with $k_1+k_2+k_3=n{-}1$, we have 
	\begin{displaymath}
		\Big[ \bS^{k_1},~\Big[\bS^{k_2},~\Big[\bS^{k_3},\bpi_1\Big]\Big]\Big] = 
		\left(\dfrac{n!}{2^{n{-}1}}\right)
		\bpi_{n}
		~.
	\end{displaymath}
\end{corollary}
 \begin{proof}
  The proof goes by induction as before. Here the basis is given by 
Lemma \ref{Prop:QuaternaryCommutator}.
 \end{proof}

\begin{corollary}\label{prop:marcoestate2}
	For any $k_i \geq 1$ with $k_1+k_2=n-2$, we have
	\begin{displaymath}
		[ \bS^{k_1},[\bS^{k_2},\bpi_2]] = 
		\left(\dfrac{n!}{2^{n{-}1}}\right)
		\bpi_{n}
		~.
	\end{displaymath}
\end{corollary}
\begin{proof}
	The proof follows from Corollary \ref{lemma:pairpi2-as-pairpairPi1} and \ref{prop:marcoestate1}.
\end{proof}

\begin{corollary}\label{prop:marcoestate3}
{
For $q\geq 3$ and $k=n-q \ge 1$, we have
 \begin{displaymath}
 	\Big[ \bS^{k},	\bpi_{q}\Big]
	=
	\left( \frac{n!}{2^{n-q}q!}\right)\bpi_n 
 \end{displaymath}
}
\end{corollary}
\begin{proof}
	According to Remark \ref{rem:two},
	the l.h.s. equals $ \bS^{k} \cs \bpi_q$.
	The claim follows from Proposition \ref{cor:higherpi}.
\end{proof}

 The last three lemmas lead to -- and are subsumed by -- the following proposition, which is used in 
the proof of Lemma \ref{lem:pin'first} and in \Ss \ref{Sec:muvspiprime}.
\begin{proposition}\label{prop:marcoestate123}
	Consider positive integers $q,m$ and $k_1,\dots,k_m$ with $k_1+\dots+k_m =n-q$. 
	Then  one has
 \begin{displaymath}
 	\Big[\underbrace{\bS^{ k_1},\dots[\bS^{ k_m}}_{m\text{ times}},
	\bpi_{q}]\Big]
	=
	\left( \frac{n!}{2^{n-q}q!}\right)\bpi_n 
 \end{displaymath}
whenever $q+m\geq 4$.
\end{proposition}
\begin{proof}
{In the cases $q=1$, $q=2$ and $q\ge 3$ apply  Corollary 	\ref{prop:marcoestate1}, \ref{prop:marcoestate2} and \ref{prop:marcoestate3} respectively. 
Then apply Proposition \ref{cor:higherpi}, 
this is possible since necessarily $n\ge q+m\ge 4$. }
\end{proof}

As an aside, notice that Proposition \ref{prop:marcoestate123} addresses all possible iterated commutators of powers of $\bS$ with   multibrackets $\bpi_q$, with the exception of
the combinations $[\bS^{k}, \bpi_1]$, $[\bS^{k_1},[\bS^{k_2},\bpi_1]]$ and $[\bS^{k}, \bpi_2]$. A relation between the latter is provided by Corollary \ref{lemma:pairpi2-as-pairpairPi1}.

\subsection{{Properties} of Vinogradov's  $L_{\infty}[1]$-algebra}\label{subsec:PropVinoAlg}

Consider the Vinogradov's Lie-$n$ algebra $L_\infty(E^n,\omega)$ introduced in Definition \ref{def:vinolinfty}.
We denote by $\Vspace$ the underlying graded vector space and by $\mu_k$ the $k$-th multibracket.
Let  $\mu_k\vert_\Aspace$ be the restriction of the $k$-th multibracket to the subspace $\Aspace \subset \Vspace$ and $(\Aspace[1],\{\bmu_k\})$ the $L_{\infty}[1]$-algebra corresponding to the latter restriction, after d\'ecalage.

In this subsection we establish some relations between the multibrackets $\{\bmu_k\}$ and Rogers' multibrackets $\{\bpi_k\}$. Using these relations, for all $k\ge 3$ we can easily express  $\bmu_k$ in terms of   $\bS, \bpi_k$ and $\bpi_1$, in Proposition \ref{lem:express-mun} in the body of the paper. (In turn, this and Proposition \ref{prop:sys} allow us to express $\bmu_k$ as iterated brackets of powers of $\bS$ with  $\bpi_k$ and $\bpi_1$.)

Analogously to \Ss \ref{subsec:PropRogersAlg}, we denote by $\Velem_i \in \Vspace$ a generic  
vector of $\Vspace$.
Such an element can be decomposed as $\Velem_i= f_i {+} e_i$ where
$f_i \in \bigoplus_{k=0}^{n-2}\Omega^k(M)$, and
$e_i = \pair{X_i}{\alpha_i} \in \mathfrak{X}(M)\oplus \Omega^{n{-}1}(M)$. As earlier, we interpret the contraction of an arbitrary element of 
 $\Vspace$ with a vector field $Y$ as $\iota_Y \Velem_i = \iota_Y (f_i +\alpha_i)$.

\begin{lemma}[{Binary bracket}]\label{Prop:mu2}
	\begin{displaymath}
		{\bmu_2} = \bpi_2 - [\bS,\bpi_1]
	\end{displaymath}
\end{lemma}
\begin{proof}
 First, we introduce the following {auxiliary} binary bracket on $\Vspace$:
\begin{displaymath}
\begin{split}
	\eta_2(e_1,e_2) 
	=& 
	\pair{[X_1,X_2]}{\iota_{X_1}\dd \alpha_2 - \iota_{X_2}\dd \alpha_1 + \iota_{X_1}\iota_{X_2}\omega}
	\\
	\eta_2(e,f) 
	=& 
	\eta_2(f,e)= 0
	~.
\end{split}
\end{displaymath}
 Here  we write $e_i=\pair{X_i}{\alpha_i}$ and $f_i\in \bigoplus_{k=0}^{n-2}\Omega^{k}(M)$.
Notice that this auxiliary bracket is related to the $\omega$-twisted Courant bracket via 
\begin{align*}
	\mu_2(e_1,e_2) =&~
	\dd \langle	 e_1,e_2\rangle_- + \eta_2(e_1,e_2)~,
	\\
	\mu_2(e_1,f_2) =&~ \langle e_1, \dd f_2\rangle_-
~.	
\end{align*}
 
 Recalling the natural extension of the pairing operator (see Remark \ref{rem:pair-}), the binary bracket in definition \ref{def:vinolinfty} can then be written as
$$ 	\mu_2 = \eta_2 + \mu_1 \ca \pairing_- - \pairing_- \ca \mu_1 ~.$$
 This can be checked by inspection on arbitrary elements $\Velem_i=f_i+e_i \in\Vspace$:
 \begin{displaymath}
 \begin{split}
	\mu_2(\Velem_1 ,\Velem_2) &=~ 
	\mu_2(e_1,e_2) + \mu_2(f_1,e_2)-\mu_2(e_1,f_2)
	\\
	&=~d \langle	 e_1,e_2\rangle_- + \eta_2(e_1,e_2)-\frac{1}{2}\mathcal{L}_{X_2}f_1 + \frac{1}{2}\mathcal{L}_{X_1}f_2  
	\\
	&=~ \frac{1}{2}\big[
		\dd (\iota_{X_1}\alpha_{2} - \iota_{X_2}\alpha_1	)	
		- (\dd \iota_{X_2}f_1 + \iota_{X_2} \dd f_1)
		+ (\dd \iota_{X_1}f_2 + \iota_{X_1} \dd f_2)	\big]
		+  \eta_2(e_1,e_2) 
	\\
	&=~ \frac{1}{2}\big[
		\dd \iota_{X_1}(\alpha_{2} + f_2)
		-\dd \iota_{X_2}(\alpha_{1} + f_1)
		+ \iota_{X_1} \dd f_2 -\iota_{X_2}\dd f_1		
		\big]
		+ \eta_2(f_1\oplus e_1 , f_2\oplus e_2)
	\\
	&=~\big[
		\mu_1 \ca \pairing_- - \pairing_- \ca \mu_1 + \eta_2
	\big](\Velem_1,\Velem_2)
 \end{split}
 \end{displaymath}
Restricting to $\Aspace{\subset\Vspace}$,  
one finds that
 $\eta_2  = \pi_2$, since {$\eta_2(e_1,e_2) 	= \pair{[X_1,X_2]}{\iota_{X_2}\iota_{X_1}\omega} 
	= \pi_2 (e_1,e_2)$.} 
 Hence {on $\Aspace$ we have}
\begin{equation}\label{Eq:mu2}
	\mu_2 \vert_\Aspace = \pi_2 + \mu_1 \ca \pairing_- - \pairing_- \ca \mu_1 
\end{equation} 	
and thus on $\Aspace[1]$:
\begin{equation}\label{Eq:mu2-shifted}
 {\bmu_2 = \bpi_2 + \bmu_1 \cs\bS - \bS \cs \bmu_1 ~.}
\end{equation} 
\end{proof}	

As an aside, we mention that  equation \eqref{Eq:mu2-shifted}  and Lemma  \ref{Prop:TernaryCommutator} imply that the Vinogradov binary bracket commutes with the pairing, i.e $[\bS,\bmu_2] = 0$.

\begin{lemma}[{Ternary bracket}]\label{Prop:mu3}
	\begin{displaymath}
		{\bmu_3} =
		\bpi_3
		-\frac{1}{2} [\bS,[\bS, \bpi_1]]
		- \frac{1}{6} [\bS^{2},\bpi_1]
		~.
	\end{displaymath}
\end{lemma}
\begin{proof}
 Employing the definition of the Richardson-Nijenhuis product, one can express the ternary bracket in definition \ref{def:vinolinfty} as
\begin{equation}\label{Eq:mu3}
	\mu_3 = -\dfrac{1}{3} \pairing_+\ca \mu_2 ~.
\end{equation}
 (The explicit definition of $\ca$, see equation \eqref{Eq:RNProducts}, ensures that multiplying on the left by a binary bracket, not necessarily skew-symmetric, is a well defined operation valued in graded skew-symmetric multilinear maps.) 
Indeed, equation \eqref{Eq:mu3} can be deduced by inspection on homogeneous elements, as follows. 
When evaluated on degree $0$ elements of $\Vspace$, $\mu_3$ reads as:
	$$
		\mu_3(e_1,e_2,e_3) = -T_\omega(e_1,e_2,e_3) = -\frac{1}{3} \langle[e_1,e_2]_\omega,e_3 \rangle_+ + \cyc
	~,
	$$
	{while in other degrees, for any $f$ such that $\deg(f)\neq 0$, it reads}
	\begin{displaymath}
	\begin{split}
		\mu_3(f,e_1,e_2) &=
		-\frac{1}{6}\left[
			\iota_{X_1}\left(\frac{\mathcal{L}_{X_2}}{2}f\right)
		- \iota_{X_2}\left(\frac{\mathcal{L}_{X_1}}{2}f\right)
		+ \iota_{[X_1,X_2]}f
		\right] 
		\\
		&=
		-\frac{1}{6}\left[
			\iota_{X_1}\mu_2(e_2,f)	- \iota_{X_2}\mu_2(e_1,f)
		+ \iota_{[X_1,X_2]}f
		\right] 
		\\
		&=
		-\frac{1}{3}\left[
			\langle \cdot,\cdot \rangle_+ \circ ( \mu_2 \otimes \Unit)
			\right]
			\big(
			(f,e_1,e_2)-(f,e_2,e_1)+(e_1,e_2,f)
			\big) 
		\\
		&=
		-\frac{1}{3}\left[
			\langle \cdot,\cdot \rangle_+ \circ ( \mu_2 \otimes \Unit)
			\right]
			(f,e_1,e_2)
			+ \cyc .
	\end{split}
	\end{displaymath} 
Restricting to $\Aspace$ we get that
	\begin{displaymath}
		\begin{split}
		\mu_3 \vert_\Aspace
		\equal{Eq: \eqref{Eq:mu3}}&
		-\dfrac{1}{3} \pairing_+ \ca \mu_2  
		\\
		\equal{Eq: \eqref{Eq:mu2}}&
		-\dfrac{1}{3} \pairing_+ \ca \Big( \pi_2 + \pi_1\ca \pairing_- 
		- \pairing_- \ca \pi_1 \Big)  
		\\
		\equal{\phantom{Eq: \eqref{Eq:mu2}}}&
		-\frac{1}{3} \pairing_+ \ca \pi_2 
		- \frac{1}{3}\Big( \pairing_+\ca\pi_1\ca \pairing_- 
		- \pairing_+\ca\pairing_- \ca \pi_1 \Big) 
		\\
		\equal{Eq: \eqref{Eq:pi3}}&
		\pi_3
		-\frac{1}{3} \pairing_- \ca \pi_2 + \frac{1}{3}\Big( \pairing_-\ca\pi_1\ca \pairing_- 
		- \pairing_-\ca\pairing_- \ca \pi_1 \Big)
		\end{split}
	\end{displaymath} 
   where in the last equation we used that 
   $\pairing_-\ca \eta = -\pairing_+ \ca \eta$ for any multilinear map $\eta$ in degree non zero, and that
   \begin{equation}\label{Eq:pi3}
	\pairing_+ \ca \pi_2 = \pairing_- \ca \pi_2 - 3\, \pi_3
	~.
	\end{equation}
	Equation \eqref{Eq:pi3} can be checked by {by probing it} on elements $\Aelem_i=f_i + \pair{X_i}{\alpha_i} \in \Aspace$, namely:
		\begin{displaymath}
			\begin{split}
			\Big[\big(\pairing_+ - \pairing_-\big) \ca \pi_2\Big] (\Aelem_1,\Aelem_2,\Aelem_3)
			=&~
			\iota_{X_3} \pi_2(\Aelem_1,\Aelem_2) + \cyc 
			\\
			=&~
			\omega(X_1,X_2,X_3) + \cyc 
			\\
			=&~
			3 \omega(X_1,X_2,X_3)
			\\
			=&~
			- 3 \pi_3 (\Aelem_1,\Aelem_2,\Aelem_3)
			~.
			\end{split}
		\end{displaymath}
	The claim {of the proposition} follows after applying the d\'ecalage:
	\begin{displaymath}
	 \begin{split}
		{\bmu_3}
		\equal{\phantom{Eq: \eqref{Eq:mu2}}}&
		\bpi_3
		-\frac{1}{3} \bS \cs \bpi_2 + \frac{1}{3} \bS\cs\bpi_1\cs \bS
		-\frac{1}{3} \bS\cs\bS \cs \bpi_1 
		\\
		\equal{Eq: \eqref{Eq:pairing-pi2}}&
		\bpi_3
		-\frac{1}{3} [\bS, \bpi_2] 
		+ \frac{1}{6}\Big( 
		-[\bS,\bpi_2] + \bpi_1\cs\bS^{ 2}
		+\bS^{ 2}\cs \bpi_1
		\Big)
		-\frac{1}{3} \bS^{ 2} \cs \bpi_1
		\\
		\equal{\phantom{Eq: \eqref{Eq:mu2}}}&		
		\bpi_3
		-\frac{1}{2} [\bS, \bpi_2] 
		+ \frac{1}{6}\Big( 
		 \bpi_1\cs\bS^{ 2}
		-\bS^{ 2}\cs \bpi_1
		\Big)
		\\		
		\equal{\phantom{Eq: \eqref{Eq:mu2}}}&		
		\bpi_3
		-\frac{1}{2} [\bS, \bpi_2] 
		- \frac{1}{6} [\bS^{ 2},\bpi_1] 
		\\
		\equal{Lemma \ref{Prop:TernaryCommutator}}&		
		\bpi_3
		-\frac{1}{2} [\bS,[\bS, \bpi_1]] 
		- \frac{1}{6} [\bS^{ 2},\bpi_1] 
		~,
	 \end{split}
	\end{displaymath}
	where in the second equality we used equation \eqref{Eq:pairing-pi2} to rewrite the term $\frac{1}{3} \bS\cs\bpi_1\cs \bS$.
\end{proof}

According to the next corollary, the multibrackets of Vinogradov's $L_\infty[1]$ structure on $\Aspace[1]$ can be obtained, using the   product  $\cs$, from the four elements  $\{\bS, \bpi_1, \bpi_2, \bpi_3\}$, regarding each of them as a map $S^{\ge 1}(\Aspace[1])\to \Aspace[1]$.
This is in perfect analogy with the corresponding statement about Rogers' $L_\infty[1]$ structure, just before Proposition \ref{cor:higherpi}.

\begin{proposition}\label{lem:mun}
	For all $n\ge 3$:
	\begin{displaymath}
		\bmu_{{n}} = 			
		3~
		\left(
			\frac{2^{n{-}1}}{(n{-}1)!}B_{n{-}1}	
		\right)
			~\bS^{n{-}3} \cs \bmu_3
			~.
	\end{displaymath}
\end{proposition}	
\begin{proof}
	According to Equation 	\eqref{eq:VinoMultibrakAllaZambon_1}, the explicit value of $\mu_n(\Velem_1,\dots,\Velem_n)$  is a sum of two terms which can be rewritten employing the anchor operator $\rho$, such that $\rho(\Velem_i) = X_i$ for any  $\Velem_i=f_i\oplus e_i \in \Vspace$. 
	{We consider the two terms separately.}
	
	i) The first summand reads as
	\begin{displaymath}
		\begin{split}
			\sum_{i=1}^n (-1)^{i-1}\mu_n \left( \Velem_i,\rho(\Velem_1),\dots,\widehat{\rho(\Velem_i)},\dots,\rho(\Velem_n)\right)
			=
			\left[
				\mu_n \circ 
				\big(\Unit\otimes \rho^{\otimes(n{-}1)}\big) 
			\right]
				\circ 
				P_{1,n{-}1}
				~(\Velem_1,\dots,\Velem_n)				
		\end{split}
	\end{displaymath}	
	noticing  that $\sigma=(i,1,\dots,\widehat{i},\dots,n)\in \ush{1,n{-}1}$ and $|\sigma|=(-1)^{i+1}$.
	Explicitly, see equation \eqref{eq:VinoMultibrakAllaZambon_2}, the term  on the r.h.s. 
	in square brackets can be realized by contraction with several vector fields: 
	\begin{displaymath}
		\begin{split}
			\mu_n &\circ\big(\Unit\otimes \rho^{\otimes(n{-}1)}\big)~
			(\Velem_0,\Velem_1,\dots,\Velem_{n{-}1})=
			\\
			=&~			
			\mu_n( \Velem_0,X_1\dots,X_{n{-}1} )
			\\
			=&~
			c_n			
			\sum_{1\leq i < j \leq n{-}1} (-1)^{i+j+1}
			\iota_{\rho(\Velem_{n{-}1})}\dots\widehat{\iota_{\rho(\Velem_j)}}\dots
			\widehat{\iota_{\rho(\Velem_i)}}\iota_{\rho(\Velem_1)}
			~[\Velem_0,\rho(\Velem_i),\rho(\Velem_j)]_3 
			\\
			=&~
			c_n
			\left(
				-\varsigma(n-3)\frac{2^{n-3}}{(n-3)!}
			\right)			
			\sum_{1\leq i < j \leq n{-}1} (-1)^{i+j+1}
			\pairing_-^{\ca (n-3)}
			\circ
			\left([\cdot,\cdot,\cdot]_3\otimes\Unit_{n-3}\right)
			\circ 
			\left(\Unit\otimes \rho^{\otimes n{-}1}\right)\\
			&\qquad~
			(\Velem_0,\Velem_i,\Velem_j,\Velem_1,\dots,\widehat{\Velem_i},\dots,\widehat{\Velem_j}\dots,\Velem_{n{-}1})
			\\
			=&~
			3~d_n
			~
			\pairing_-^{\ca (n-3)}
			\circ
			([\cdot,\cdot,\cdot]_3\otimes\Unit_{n-3})
			\circ 
			(\Unit\otimes \rho^{\otimes n{-}1})
			\circ 
			(\Unit\otimes P_{2,n-3})
			~(\Velem_0,\Velem_1,
			\dots,\Velem_{n{-}1}),
		\end{split}
	\end{displaymath}	
where $[\cdot,\cdot,\cdot]_3 = -T_0$ denotes the ternary bracket  associated to the untwisted  Vinogradov Algebroid, just as in Definition \ref{def:vinolinfty}.
	
Here $c_n$ is the coefficient defined in equation \eqref{eq:UglyCoefficient} and, in the last equality, we noticed that $(-1)^{i+j+1}=|\sigma|$ with $\sigma=(i,j,1,\dots,\widehat{i},\dots,\widehat{j},\dots,n{-}1)\in \ush{2,n-3}$.
	
	Further, $d_n$ is given by 
		\begin{displaymath}
		\begin{split}
		d_n =~		\frac{c_n}{3} 		
		\left(
				-\varsigma(n-3)\frac{2^{n-3}}{(n-3)!}
		\right)
				=~
		\frac{2^{n{-}1}}{(n{-}1)!}B_{n{-}1}
		~.
		\end{split}
	\end{displaymath}

	Therefore
	\begin{displaymath}
		\begin{split}
			\mu_n
			&\circ \big(\Unit\otimes \rho^{\otimes(n{-}1)}\big)\circ
			P_{1,n{-}1}
			\\
			=&~
			3~d_n~		
			\pairing_-^{\ca(n-3)}
			\circ \left( [\cdot,\cdot,\cdot]_3\otimes\Unit_{n-3} \right)
			\circ \left(\Unit\otimes \rho^{\otimes n{-}1} \right)
			\circ \left(\Unit\otimes P_{2,n-3}\right)
			\circ P_{1,n{-}1}
			\\
			=&~
			3~d_n~		
			\pairing_-^{\ca(n-3)}
			\circ \left( [\cdot,\cdot,\cdot]_3\otimes\Unit_{n-3} \right)
			\circ \left(\Unit\otimes \rho^{\otimes n{-}1} \right)
			\circ P_{1,2,n-3}
			\\
			=&~
			3~d_n~	
			\pairing_-^{\ca(n-3)}
			\circ \left(
				\left([\cdot,\cdot,\cdot]_3\circ \Unit\otimes \rho^{\otimes 2} 
				\circ P_{1,2}
			\right)
			\otimes \rho^{\otimes (n-3)}\right)
			\circ P_{3,n-3}
			\\
			=&~
			3~d_n~	
			\pairing_-^{\ca(n-3)}
			\circ \left([\cdot,\cdot,\cdot]_3
			\otimes \rho^{\otimes (n-3)}\right)
			\circ P_{3,n-3}
			\\
			=&~
			3~d_n~		
			\left(
			(-1)^{n-3}		
			\pairing_-^{\ca(n-3)}
			\circ \left(
				[\cdot,\cdot,\cdot]_3
				\otimes \Unit_{n-3}\right)
				\circ P_{3,n-3}
			\right)
			\\			
			=&~
			3~d_n~	
			\left(
			\pairing_-^{\ca(n-3)}
			\ca
			[\cdot,\cdot,\cdot]_3
			\right)
			~.
		\end{split}
	\end{displaymath}
	The last three equalities follow respectively from the observations that $\mu_3\circ \big(\Unit\otimes \rho^{\otimes 2}\big) \circ P_{1,2} = \mu_3$ (see the definition of the ternary bracket), 
	that $\pairing_-^{n-3}\circ \big(\alpha \otimes \rho^{\otimes (n-3)}\big) = \pairing_-^{n-3}\circ (\alpha \otimes \Unit_{n-2})$ for any element $\alpha\in \Vspace$ such that $\rho(\alpha) = 0$, 
	and that $(-1)^{n-3}=1$ for any $n\geq 3$ odd. 

ii) According to equation \eqref{eq:VinoMultibrakAllaZambon_1}, the second summands reads as
	$$
	\left(
			(-1)^{\frac{n{+}1}{2}}\cdot n \cdot B_{n{-}1} 
	\right)
	\iota_{\rho(\Velem_n)}\dots\iota_{\rho(\Velem_1)} \omega
	~.
	$$
	Restricting this term to $\Aspace$, one can employ Lemma  \ref{lemma:InsertionsAsPairing} to rewrite the previous term as follows.
	Consider elements $\Aelem_i$ in $\Aspace$, one has:
	\begin{align*}
		\Big(
			(-1)^{\frac{n{+}1}{2}}\cdot &n \cdot B_{n{-}1} 
		\Big)
		~\iota_{\rho(\Aelem_n)}\dots\iota_{\rho(\Aelem_1)} \omega
		=
		\\
		=&
		-\varsigma(n)
		\left(
			-\frac{2^{n-3}}{n!}3!
		\right)
		\left(
			(-1)^{\frac{n{+}1}{2}}\cdot n \cdot B_{n{-}1} 
		\right)
		\left(
			\pairing_-^{\ca (n-3)} \ca \pi_3
		\right)
		~ (\Aelem_1,\dots,\Aelem_n)
		\\
				=&
		-\frac{3}{2} d_n
		~
		\pairing_-^{\ca (n-3)} \ca \pi_3
		~ (\Aelem_1,\dots,\Aelem_n)~.
	\end{align*}
	Summing up the two terms coming from i) and ii), one can conclude that
	\begin{displaymath}
			\mu_n \vert_{\Aspace}=
			3 ~d_n ~
			\pairing_-^{\ca (n-3)} \ca \left([\cdot,\cdot,\cdot]_3 -\frac{1}{2}\pi_3 \right)
	\end{displaymath}	
	for any $n\ge 3$.
	Noting that $[\cdot,\cdot,\cdot]_3 -\frac{1}{2}\pi_3=\mu_3 \vert_\Aspace$, we conclude that
		\begin{displaymath}
			\mu_n \vert_\Aspace 
			=
			3 ~d_n ~
			\pairing_-^{\ca (n-3)} \ca \mu_3 \vert_\Aspace
			~.
	\end{displaymath}	 	
	The claim follows after d\'ecalage.
\end{proof}

As a consequence of Lemma \ref{Prop:mu2}, Lemma \ref{Prop:mu3}, and Proposition \ref{lem:mun}, 
we   conclude that the multibrackets of Vinogradov's $L_\infty[1]$ structure on $\Aspace[1]$ can be also obtained, using the   product  $\cs$, from the  same set of elements $\{\bS, \bpi_1, \bpi_2, \bpi_3\}$ that give rise to the shifted Rogers' multibrackets  $\{\bpi_k\}$.
 
 
\section{The proof of Proposition \ref{prop:sys} }\label{sec:system}

In this appendix we  prove Proposition \ref{prop:sys}. We do so by proving the following stronger statement:

\begin{proposition}\label{prop:syspol}
 Let $n\geq {5}$ be odd. Let $d$ and $b_k$ as in Lemma \ref{lem:techMarco2}. Then there exist real numbers $\{a_{k_3k_2}\}$ and $\{a_{k_4k_3}\}$   such
that one has an equality between the coefficients appearing in
\begin{equation}\label{eq:231pol}
2\bS^{n-1}\cs \bpi_1 - 3\bS^{n-2}\cs\bpi_1 \cs\bS + \bS^{n-3}\cs\bpi_1\cs \bS^2  
\end{equation} 
and the coefficients of the terms of the form $\bS^{j}\cs\bpi_1 \cs\bS^{n-1-j}$ obtained writing out the commutators in
\begin{equation}\label{eq:polydeveloppol}
	\begin{split}
	d[\bS^{n{-}1} ,\bpi_1]
 	&+
	\sum_{\substack{k \text{ even}\\ 2\le k \le N}}
	b_k[\bS^k,[\bS^{n{-}1-k},\bpi_1]]+
	\\
	&+
	\sum_{\substack{k_3\ge k_2\ge 1\\ k_3+k_2 =n{-}2}}
	a_{k_3k_2}  [\bS^{k_3},[\bS^{k_2},[\bS ,\bpi_1]]]
	\\
	&+
	\sum_{\substack{k_4\ge   k_3\ge 1 \\ k_4+k_3=n{-}3}}
	a_{k_4k_3}  [\bS^{k_4},[\bS^{k_3},[\bS,[\bS,\bpi_1]]]]
	~,
	\end{split}
\end{equation}
with $N=(n{-}1)/2$.
\end{proposition}

\begin{remark}
i) An instance of ``writing out the commutators''  is:
$$[\bS^{k},[\bS^{n{-}1-k},\bpi_1]]=
\bS^{n-1}\cs\bpi_1  -\bS^{k}\cs\bpi_1 \cs\bS^{n{-}1-k}-\bS^{n{-}1-k}\cs\bpi_1 \cs\bS^{k}+\bpi_1 \cs\bS^{n{-}1}.$$
ii)  The statement of Proposition \ref{prop:syspol} is stronger than the one of
Proposition \ref{prop:sys}, for the following reason. The terms $\bS^{j}\cs\bpi_1 \cs\bS^{n-1-j}$, for $j=0,\dots, n$, are elements of the graded vector space of linear maps
$S^{\ge 1}(\Aspace[1])\to \Aspace[1]$. These elements are not linearly independent: for instance, one has
$$\bS^3\cs \bpi_1 \cs \bS-3 \bS^2\cs \bpi_1 \cs \bS^2+
3 \bS\cs \bpi_1 \cs \bS^3-\bpi_1 \cs \bS^4=0,$$ as a consequence of 
Lemma \ref{Prop:QuaternaryCommutator} and Remark \ref{rem:two}.
In other words, asking that \eqref{eq:231pol} and \eqref{eq:polydeveloppol} have the same ``polynomial decomposition'' with respect to the  $\bS^{j}\cs\bpi_1 \cs\bS^{n{-}1{-}j}$ is a stronger condition than requiring that both linear maps agree when evaluated on arbitrary elements of $S^{\geq 1}(\Aspace[1])$.

iii) The coefficients $\{a_{k_3k_2}\}$ and $\{a_{k_4k_3}\}$ in the statement of Proposition \ref{prop:syspol} are unique, as our proof  below shows (see the last paragraph before \Ss \ref{subsec:rewriting}).
\end{remark}

In the rest of this appendix, we provide a proof for Proposition \ref{prop:syspol}.
Write  $$E_j:=\bS^{j}\cs\bpi_1 \cs\bS^{n{-}1-j}$$
for $j=0,\dots, n-1$.
 Proposition \ref{prop:syspol} is a statement about the coefficients of $E_{n{-}1},\dots,E_0$ in the expressions \eqref{eq:231pol} and \eqref{eq:polydeveloppol}. It can be cast as the existence of a solution of an inhomogeneous linear system of $n$ equations in the variables $\{a_{k_3k_2}\}$ and $\{a_{k_4k_3}\}$, where on the right hand side we put the coefficients of the terms corresponding  to 
  \begin{itemize}
\item $2E_{n{-}1} -3E_{n-2}+E_{n-3}$,
\item    $-d[\bS ,\bpi_1]$,
\item   
$-b_k[\bS^k,[\bS^{n{-}1-k},\bpi_1]]$ for $2\le k\le N, \;k \text{ even}$.
\end{itemize}
 We   denote this system by
\begin{equation}\label{eq:sysgen}
  \mathcal{M} \cdot \vec{a}=\mathcal{R}.
\end{equation}

This is best explained by means of examples.
\begin{example}\label{ex:n5}
We spell out the system \eqref{eq:sysgen} for $n=5$ (hence $N=2)$. 
Equation \eqref{eq:polydeveloppol} reads
\begin{equation*}
	\begin{split}
	d[\bS^{4} ,\bpi_1]
 	&+
	 	b_2[\bS^2,[\bS^{2},\bpi_1]]+
	\\
	& 
	+	a_{21}  [\bS^{2},[\bS,[\bS,\bpi_1]]]
	\\
	&
	+	a_{11}  [\bS,[\bS,[\bS,[\bS,\bpi_1]]]]
	~,
	\end{split}
\end{equation*}  
	where the first two coefficients involve the Bernoulli numbers, and are given  by $d=\frac{1}{4}$ and $b_2=\frac{5}{8}$.
Writing this out in terms of the $E_k$, we see that	the system  \eqref{eq:sysgen}   reads as follows (from top to bottom, the rows correspond respectively to the coefficient of the terms $E_{4},\dots, E_0$):
$$\left(\begin{array}{rr}
1 &1 \\
-2 &-4 \\
0 &6 \\
2 &-4 \\
-1 &1 \\
\end{array}\right)
\left(\begin{array}{c}
a_{21}	 \\
a_{11}	\\ 
\end{array}\right)
=
\left(\begin{array}{c}
\frac{9}{8} \\
-3	\\
\frac{9}{4}\\
0\\
-\frac{3}{8}\\
\end{array}\right)
=
\left(\begin{array}{c}
 2	\\ 
-3	 \\
 1	 \\
 0	 \\
0	  \\
\end{array}\right)
-\frac{1}{4}
\left(\begin{array}{c}
 1	 \\
 0	 \\
 0	 \\
 0	  \\
-1	 \\
\end{array}\right)
-\frac{5}{8}
\left(\begin{array}{c}
  1	 \\
 0 \\
 -2	 \\
 0	  \\
 1	 \\
\end{array}\right)
.
$$
The statement of Proposition \ref{prop:syspol} is that this linear system has a solution.
One can indeed check directly that there is a (unique) solution, given by
$(\frac{3}{4},\frac{3}{8})$.\end{example}

\begin{example}\label{ex:n7}
We spell out the system \eqref{eq:sysgen} for $n=7$ (hence $N=3)$.
Equation \eqref{eq:polydeveloppol} reads
\begin{equation*}
	\begin{split}
	d[\bS^{6} ,\bpi_1]
 	&+
	 	b_2[\bS^2,[\bS^{4},\bpi_1]]+
	\\
	& 
	+	a_{32}  [\bS^{3},[\bS^2,[\bS,\bpi_1]]]
	+	a_{41}  [\bS^{4},[\bS,[\bS,\bpi_1]]]
	\\
	&		+	a_{22}  [\bS^{2},[\bS^{2},[\bS,[\bS,\bpi_1]]]]
	+	a_{31}  [\bS^{3},[\bS,[\bS,[\bS,\bpi_1]]]]
	~,
	\end{split}
\end{equation*}  
	where the first two coefficients 
	are given  by $d=\frac{1}{6}$ and $b_2=\frac{7}{16}$.
Writing this out in terms of the $E_k$, we see that	the system  \eqref{eq:sysgen}   reads as follows:
$$\left(\begin{array}{rrrr}
  1	& 1	& 1	& 1	\\ 
 -1	& -2	&-2	&-3	\\        
 -1	&1	 &-1	&3	\\    
 0	 &0	&4	& -2	\\         
 1	 &-1	 &-1	&3	 \\   
 1	 &2	&-2	&-3	 \\       
 -1	&-1	& 1	& 1	\\  
\end{array}\right)
\left(\begin{array}{c}
a_{32}	 \\
a_{41}	\\ 
a_{22}	  \\
a_{31}	 \\
\end{array}\right)
=
\left(\begin{array}{c}
\frac{67}{48} \\
-3	\\
\frac{23}{16}\\
0\\
\frac{7}{16}\\
0\\
-\frac{13}{48}\\
\end{array}\right)
=
\left(\begin{array}{c}
 2	\\ 
-3	 \\
 1	 \\
 0	 \\
0	  \\
 0	  \\
0	 \\
\end{array}\right)
-\frac{1}{6}
\left(\begin{array}{c}
 1	 \\
 0	 \\
 0	 \\
 0	  \\
0	 \\
 0	 \\
-1	 \\
\end{array}\right)
-\frac{7}{16}
\left(\begin{array}{c}
  1	 \\
 0 \\
 -1	 \\
 0	  \\
 - 1	\\  
 0  \\
 1	 \\
\end{array}\right)
.
$$
This linear system has a (unique) solution, given by
$(\frac{1}{6},\frac{2}{3},\frac{3}{16},\frac{3}{8})$.
\end{example}

 Our task is to show that  the system \eqref{eq:sysgen} has a solution, for any odd integer $n\ge 9$ (the cases $n=5$ and $n=7$ were already treated in the examples above). The main difficulty is that there is no clear pattern for the entries of $\mathcal{M}$ as $n$ varies, and that the right hand side 
$\mathcal{R}$ involves non-trivially the Bernoulli numbers.  
It turns out that 
$\mathcal{M}$ is a $n\times (n-3)$ matrix, hence the system of linear equations it represents is overdetermined. Our approach will be to describe precisely the column space  of $\mathcal{M}$, which has dimension $n-3$,
and to show that $\mathcal{R}$ always lies in the column space.

\subsection{Rewriting the system}\label{subsec:rewriting}
 
We now transform the  linear system \eqref{eq:sysgen} to an equivalent one, by performing the following elementary operations to the matrix $\mathcal{M}$ -- which we recall has $n=2N+1$ rows -- and to $\mathcal{R}$: for all $k\le N$,
\begin{itemize}
\item replace the $k$-th row $R_k$ by $\frac{1}{2}(R_k+R_{n+1-k})$,
\item replace the $(n+1-k)$-th row $R_{n+1-k}$  by $\frac{1}{2}(R_k-R_{n+1-k})$.
\end{itemize}
(The middle row $R_{N+1}$ remains unchanged.)
We denote by $\mathcal{M}'$ the resulting matrix, and by $\mathcal{R}'$ the resulting right hand side.

\begin{remark}\label{rem:comm}
  We have
$$[\bS^{k},[\bS^{n{-}1-k},\bpi_1]]=E_{n{-}1}-E_k-E_{n{-}1-k}+E_0.$$
One computes easily that   given integers $k_1,k_2,k_3\ge 1$ with sum $n{-}1$, one has  
$$[\bS^{k_3},[\bS^{k_2},[\bS^{k_1},\bpi_1]]]=E_{n{-}1}-\sum_{i=1}^3 E_{n{-}1-k_i}  + \sum_{i=1}^3 E_{k_i}  - E_0.$$
Similarly, given integers $k_1,k_2,k_3,k_4\ge 1$ with sum $n{-}1$, one has 
$$[\bS^{k_4},[\bS^{k_3},[\bS^{k_2},[\bS^{k_1},\bpi_1]]]]=
E_{n{-}1}-\sum_{i=1}^4 E_{n{-}1-k_i} + \sum_{1\le i<j\le 4} E_{k_i+k_j} - \sum_{i=1}^4 E_{k_i}  + E_0.$$
\end{remark}

\begin{remark}\label{rem:zeros}
Notice that in the   expressions for the iterated commutators of $\bS^{\bullet}$ with $\bpi_1$ appearing in Remark \ref{rem:comm}, the coefficients of $E_k$ and $E_{n{-}1-k}$ are equal when we have an even number of powers of $\bS^{\bullet}$, and are equal up to a sign otherwise.
Consequently, the   matrix  $\mathcal{M}'$ has the following property:
\begin{itemize}
\item[i)]  the columns associated  to $[\bS^{k_3},[\bS^{k_2},[\bS,\bpi_1]]]$  have
zeros in the first $N+1$ rows (those corresponding to $E_{n{-}1},\dots,E_N$)
\item[ii)]  the columns associated  to $[\bS^{k_4},[\bS^{k_3},[\bS,[\bS,\bpi_1]]]]$  have
zeros in the last $N$ rows (those corresponding to $E_{N-1},\dots,E_0$)
\end{itemize}
\end{remark}

Now denote by $\mathcal{M}'_{top}$ the matrix consisting of the first $N+1$ rows of  $\mathcal{M}'$, and by $\mathcal{M}'_{bot}$  the matrix consisting of the last $N$ rows. A key observation is that 
\begin{itemize}
\item the columns of $\mathcal{M}'_{top}$ listed in item i) of Remark \ref{rem:zeros} are identically zero,
\item the columns of   $\mathcal{M}'_{bot}$ listed in item ii) of Remark \ref{rem:zeros} are identically zero.
\end{itemize} 
Because of this, the system $\mathcal{M}' \cdot \vec{a}=\mathcal{R}'$ decouples into  two systems, the top and the bottom one. Thus, to prove Proposition \ref{prop:syspol}, it suffices to show that both the top and bottom system have a solution. We will do this in Corollaries \ref{cor:bot} and \ref{cor:top}, for $n\ge 9$, thus concluding our proof of  Proposition \ref{prop:syspol}.

The following equivalences summarize our discussion so far
\footnote{The symbols $\mathcal{R}'_{top}, \mathcal{R}'_{bot}, \vec{a}_{top}, \vec{a}_{bot}$ are defined analogously to the above.}:
\begin{align*}
  \mathcal{M} \cdot \vec{a}=\mathcal{R} &\text{ has a solution} \Leftrightarrow 
 \mathcal{M}'\cdot \vec{a}=\mathcal{R}' \text{ has a solution}\\
 & \Leftrightarrow
\mathcal{M}'_{top} \cdot \vec{a}_{top}=\mathcal{R}'_{top} \text{ has a solution and } \mathcal{M}'_{bot} \cdot \vec{a}_{bot}=\mathcal{R}'_{bot} \text{ has a solution}.
\end{align*}

\begin{example}
When $n=7$, in view of Example \ref{ex:n7}, the system $\mathcal{M}'\cdot \vec{a}=\mathcal{R}'$ reads
 $$\left(\begin{array}{rrrrr}
  0	& 0	& 1	& 1	\\ 
0	& 0	&-2	&-3	\\        
0	&0	 &-1	&3	\\    
 0	 &0	&4	& -2	\\  
 \hline       
-1	 &1	 &0	&0	 \\   
-1	 &-2	&0	&0	 \\       
1	&1	& 0	& 0	\\  
\end{array}\right)
\left(\begin{array}{c}
a_{32}	 \\
a_{41}	\\ 
a_{22}	  \\
a_{31}	 \\
\end{array}\right)
=
\left(\begin{array}{c}
\frac{9}{16} \\
-\frac{3}{2}\\
\frac{15}{16}\\
0\\
\hline
\frac{1}{2}\\
-\frac{3}{2}\\
\frac{5}{6}\\
\end{array}\right)
=
\left(\begin{array}{c}
 1	\\ 
-\frac{3}{2}	 \\
\frac{1}{2}	 \\
 0	 \\
\frac{1}{2}	  \\
-\frac{3}{2}	  \\
1	 \\
\end{array}\right)
-\frac{1}{6}
\left(\begin{array}{c}
 0	 \\
 0	 \\
 0	 \\
 0	  \\
0	 \\
 0	 \\
1	 \\
\end{array}\right)
-\frac{7}{16}
\left(\begin{array}{c}
  1	 \\
 0 \\
 -1	 \\
 0	  \\
0	\\  
 0  \\
0	 \\
\end{array}\right).
$$
 \end{example}
 The horizontal line separates the top system from the bottom system.

\subsection{The bottom system}

{The results of this subsection hold for odd $n\ge 7$. After removing all columns consisting of zeros, $\mathcal{M}'_{bot}$ is a $N\times (N-1)$ matrix whose columns correspond to 
$[\bS^{k_3},[\bS^{k_2},[\bS,\bpi_1]]]$ with $k_3\ge k_2\ge 1$ and $k_3+k_2 =2N-1$. Notice that the admitted pairs $(k_3,k_2)$ are   given by 
$$ (N,N-1),\; (N+1,N-2),\; (N+2,N-3),\dots,\; (2N-3,2)
,\; (2N-2,1).$$
As one sees\footnote{The proof is analog to the one of Lemma \ref{lemma:m'top} but simpler.} using Remark \ref{rem:comm}, $\mathcal{M}'_{bot}$ reads as follows (for the sake of readability we highlighted in boldface the elements lying on a ``diagonal''):} 
$$\mathcal{M}'_{bot}=\left(\begin{array}{ccccccc}
{\bf -1} & 1 & 0 &\dots &\dots& 0 & 0  \\
0 & {\bf -1} & 1 & \dots &\dots & 0 & 0   \\
0 & 0 & {\bf -1} & \dots &\dots & 0 & 0  \\
0 & 0 & 0 & \dots &\dots & 0 & 0   \\
\dots & \dots & \dots &\dots & \dots & \dots & \dots  \\
0 & 0 & 0 & \dots &\dots & 1 & 0   \\
0 & 0 & 0 & \dots &\dots & {\bf -1} & 1   \\
-1 & -1 & -1 & \dots &\dots & -1 & {\bf -2}   \\
1 & 1 & 1 & \dots &\dots & 1 & 1 
\end{array}\right)
$$
The $N-1$ columns of $\mathcal{M}'_{bot}$ are linearly independent, and each of them is perpendicular to the vector $(1,2,3,\dots, N)^{\top}$. Hence:
\begin{lemma}
The column space of the matrix $\mathcal{M}'_{bot}$  equals the $N-1$-dimensional subspace of $\RR^N$ which is orthogonal to the vector $(1,2,3,\dots, N)^{\top}$.
\end{lemma}

The right hand side of the bottom system has no contribution from terms corresponding to the $b_k[\bS^k,[\bS^{n{-}1-k},\bpi_1]]$ (see Remark \ref{rem:zeros}). Using that $d=\frac{1}{n{-}1}$, it reads  
$$
\mathcal{R}'_{bot}=\left(\begin{array}{c}0 \\\dots\\0 \\ 1/2 \\-3/2 \\\frac{n-2}{n{-}1}\end{array}\right).
$$
Since $\mathcal{R}'_{bot}$ is also orthogonal to 
$(1,2,3,\dots, N)^{\top}$, we obtain:
\begin{corollary}\label{cor:bot}
For all odd integers $n\ge 7$, 
The   system $\mathcal{M}_{bot}' \cdot \vec{a}=\mathcal{R}'_{bot}$ has a solution. 
\end{corollary}

\subsection{The top system}

Let $n\ge 9$. After removing all columns consisting of zeros, $\mathcal{M}'_{top}$ is a $(N+1)\times (N-1)$ matrix whose columns correspond to 
$[\bS^{k_4},[\bS^{k_3},[\bS,[\bS,\bpi_1]]]]$
with  $k_4\ge   k_3\ge 1$ and  $k_4+k_3=2N-2$. Notice that the admitted pairs $(k_4,k_3)$ are   given by 
$$ (N-1,N-1),\; (N,N-2),\; (N+1,N-3),\dots, \; (2N-4,2)
,\; (2N-3,1).$$

\begin{lemma}\label{lemma:m'top}
For $n\ge 11$, the matrix $\mathcal{M}'_{top}$ reads as follows (for the sake of readability we highlighted in boldface the elements lying on a ``antidiagonal''):
$$\mathcal{M}'_{top}=\left(\begin{array}{cccccccccc}
1	& 1	 &1  &1	 &\dots &\dots 	&\dots & 1	 &1	& 1\\
-2	&-2	&-2	&-2 &\dots	&\dots &\dots &-2	&-2	&-3 \\
 1	& 1	& 1	& 1	 &\dots	&\dots &\dots & 1	& 0	 &{\bf 3} \\
 0	& 0	& 0	& 0 &\dots &\dots &\dots &	-1	 &{\bf 2}	&-1 \\
 0	& 0	& 0	& 0  &\dots	&\dots &\dots & {\bf 2}	&-1	& 0\\
 0	& 0	& 0	& 0 &\dots	&\dots &\dots &-1	 &0	& 0\\
 \dots & \dots &\dots &\dots &\dots & \dots & \dots & \dots  & \dots &\dots  \\
 0 & 0 & 0  &-1& \dots &\dots &\dots & 0 & 0  & 0 \\
 0	& 0	&-1	& {\bf 2} &\dots 	&\dots &\dots & 0	 &0 &	 0\\
 0	&-1	& {\bf 2}	&-1 &\dots	&\dots &\dots & 0	& 0	& 0\\
-2	& {\bf 2}	&-1	& 0 &\dots &\dots &\dots &	 0	& 0	 &0\\
 {\bf 4}	&-2	& 0	& 0 &\dots&\dots  &\dots & 	 0	& 0	& 0\\
\end{array}\right)
$$  
\end{lemma}
 \begin{proof}
Let us first make the more general  assumption $n\ge 9$, i.e. $N\ge 4$. 
Let $k_4\ge   k_3\ge 1$ such that  $k_4+k_3=n{-}3$.
By Remark \ref{rem:comm} we know that  
$[\bS^{k_4},[\bS^{k_3},[\bS,[\bS,\bpi_1]]]]$ equals
\begin{equation}\label{eq:geN}
  (E_{n{-}1}-2E_{n-2}+E_{n-3})-E_{n-1-k_3}-E_{n-1-k_4}+2E_{k_3+1}+2E_{k_4+1}-E_{k_4}
\end{equation}
plus a linear combination of the $E_k$'s with $k< N$.
By construction, the corresponding column of $\mathcal{M}'_{top}$ consists
of the coefficients
corresponding to $E_k$ with $k \ge N$. Thus the terms in the round bracket in 
\eqref{eq:geN} contribute to $\mathcal{M}'_{top}$, and we have to determine how each of the remaining five terms contributes. The statement of the Lemma follows by specializing \eqref{eq:geN} to three cases: 
\begin{itemize} 
\item[1)] $k_3=N-1=k_4$, i.e. the first column of $\mathcal{M}'_{top}$. 
\item[2)] $k_3=N-2,k_4= N$, i.e. the second column of $\mathcal{M}'_{top}$. It is only in this case that we need $n\ge 11$. (For $n=9$ we obtain $(1,-2,0,2,-2)^{\top}$ as the second column of $\mathcal{M}'_{top}$.)
\item[3)] $k_3\le N-3$, i.e. all columns of $\mathcal{M}'_{top}$ except for the first two.
\end{itemize}
\end{proof}

\begin{lemma}\label{lem:corank2}. 
The column space of the matrix $\mathcal{M}'_{top}$  equals the $N-1$-dimensional subspace of $\RR^{N+1}$ which is orthogonal to the vectors
$$
v_1=\left(\begin{array}{c}1 \\ 1\\\dots\\1 \\ 1/2  \end{array}\right)
\quad\quad\text{and}\quad\quad
v_2=\left(\begin{array}{c}N^2 \\ (N-1)^2\\\dots\\4 \\ 1 \\0 \end{array}\right).
$$
\end{lemma}
\begin{proof}
Assume first $N\ge 5$, in which case $\mathcal{M}'_{top}$ is given in Lemma \ref{lemma:m'top}. The columns of $\mathcal{M}'_{top}$ are linearly independent, as one checks easily considering first the bottom rows of the matrix, so the column space has dimension $N-1$. 
It is immediate that the columns $\mathcal{M}'_{top}$ are orthogonal to $v_1$. 
The columns  are also orthogonal to $v_2$: for the first two columns this can be checked directly, and for the other columns it is a consequence of the fact that for any integer $k$, the expression $k^2-2(k-1)^2+(k-2)^2$ is constant equal to $2$. 

For $N=4$, i.e. $n=9$, one checks the statement directly, using  
the expression for $\mathcal{M}'_{top}$ obtained in the proof of Lemma \ref{lemma:m'top}.
\end{proof}

We now consider the right hand side $\mathcal{R}'_{top}$ of the top system.
Notice that term corresponding to $d[\bS ,\bpi_1]$ does not contribute. 
For any even $k$ with $2\le k\le N$, Remark \ref{rem:comm} shows that 
$$[\bS^{k},[\bS^{n{-}1-k},\bpi_1]]=
\begin{cases}
E_{n{-}1}-E_{n{-}1-k}+\text{(combinations of $E_j$ with $j<N$)} \quad \text{ if } k< N\\
E_{n{-}1}-2E_N\quad\;+\text{(combinations of $E_j$ with $j<N$)} \quad \text{ if } k= N.
\end{cases}
$$

Therefore
the $\mathcal{R}'_{top}\in \RR^{N+1}$ reads as follows, respectively in the case $N$ is odd\footnote{For $5\le k\le N-3$, the dotted entry corresponding to $E_{2N-k}$ equals $0$ if $k$ is odd, and equals $b_k$ if $k$ is even.}(left) or even (right):
$$\mathcal{R}'_{top}=
\left(\begin{array}{c}1- \sum_{2\le k\le N, \;k \text{ even}}b_k \\
-3/2\\
1/2+b_2\\
0 \\
b_4\\
\dots\\
\dots\\ 
0 \\
 b_{N-1}\\
 0
 \end{array}\right),
 \quad\quad \quad\quad
 \mathcal{R}'_{top}=
\left(\begin{array}{c}1- \sum_{2\le k\le N, \;k \text{ even}}b_k \\
-3/2\\
1/2+b_2\\
0 \\
b_4\\
\dots\\
\dots\\ 
0 \\
2b_{N}\\
 \end{array}\right).
$$

\begin{lemma}\label{lem:rhscorank2}
The vector $\mathcal{R}'_{top}$ is orthogonal to the  vectors $v_1$ and $v_2$ in Lemma \ref{lem:corank2}.  
\end{lemma}
\begin{proof}
 It is immediate that $\mathcal{R}'_{top}$  is orthogonal to $v_1$,
 so it suffices to prove the statement for $v_2$.
 
 {\it First case: $N$ is odd.} The inner product $\mathcal{R}'_{top}\cdot v_2 $ reads
$$
 N^2\left( 1- \sum_{\substack{2\le k \le N\\ k \text{ even}}}
 b_k\right)+(N-1)^2(-3/2)+(N-2)^2(1/2+b_2)
 +\sum_{\substack{ 4\le k \le N\\ k \text{ even}}}
 (N-k)^2b_k,$$
and it vanishes if{f}
\begin{equation}\label{eq:oddinnerzero}
  \sum_{\substack{ 2\le k \le N\\k \text{ even}}}
  k(2N-k)b_k=N+1/2.
\end{equation}
Observe that, since $N$ is odd and $k$ even, in the sum above we always have $k<N$.
Recall that for even $k<N$ we have (see Lemma \ref{lem:techMarco2})
$$k(2N-k)b_k=-\frac{1}{B_{n{-}1}}\binom{n{-}1}{k}B_kB_{n{-}1-k},$$
and that $N=\frac{n{-}1}{2}$.
Hence  equation \eqref{eq:oddinnerzero} is equivalent to
 $$ 
 \sum_{\substack{ 2\le k \le N-1\\k \text{ even}}}
 \binom{n{-}1}{k}B_kB_{n{-}1-k}=-\frac{n}{2}B_{n{-}1}.
$$
By symmetry reasons, in turn this is equivalent to
\begin{equation}\label{eq:double}
\sum_{\substack{2\le k \le 2N-2\\k \text{ even}}}
  \binom{n{-}1}{k}B_kB_{n{-}1-k}=-nB_{n{-}1}.
\end{equation} 
Notice that the sum goes up to $2N-2=n-3$.
Now, equation \eqref{eq:double} holds; indeed, it is exactly Euler's formula (see Remark 
\ref{Euler summation}) 
with $r=n{-}1$, as one sees using the fact that the Bernoulli numbers $B_l$ vanish for odd $l\ge 3$.

 {\it Second case: $N$ is even.}   The inner product $\mathcal{R}'_{top}\cdot v_2 $ is given by the same formula as above, and again vanishes if{f} equation \eqref{eq:oddinnerzero} is satisfied. Recall that
the definition of $b_N$ includes a factor of $\frac{1}{2}$, which is not present in $b_k$ for $k<N$ (see Lemma \ref{lem:techMarco2}). Hence for  even $N$
equation \eqref{eq:oddinnerzero} is equivalent to
 $$ \sum_{\substack{2\le k \le N-2\\k \text{ even}}}
 \binom{n{-}1}{k}B_kB_{n{-}1-k}
+\frac{1}{2} \binom{n{-}1}{N}B^2_N
 =-\frac{n}{2}B_{n{-}1}.
$$
By symmetry reasons this is equivalent to \eqref{eq:double}; notice that for this the factor of $\frac{1}{2}$ in the ``middle summand'' is really necessary.
As seen above, equation \eqref{eq:double} holds thanks to Euler's formula.
\end{proof}

 \begin{remark}\label{Euler summation}
The \emph{Euler product sum identity}, see e.g. \cite[Eq. (1.2)]{Dilcher1996a} or  \cite[Eq. (1.3)]{Hu2018}, reads:
\begin{equation*}
	\sum_{i=0}^{r}\binom{r}{i} B_i B_{r-i} = - (r{-}1)B_r -r B_{r{-}1} 
	\qquad (r\geq 1)~.
\end{equation*}
This implies, by putting outside of the summation the first two and last two values, that
\begin{displaymath}
	\sum_{i=2}^{r-2}\binom{r}{i} B_i B_{r-i} = -(r{+}1)B_r~.
	\qquad (r\geq 4)~.
\end{displaymath}

\end{remark}

From Lemma \ref{lem:corank2} and Lemma \ref{lem:rhscorank2} we immediately obtain:
\begin{corollary}\label{cor:top}
For all odd integers $n\ge 9$, the  system $\mathcal{M}_{top}' \cdot \vec{a}=\mathcal{R}'_{top}$ has a solution. 
\end{corollary}

\bibliographystyle{habbrv} 

 
\end{document}